\documentclass{article}
\usepackage[letterpaper, left=1in, right=1in, top=1in, bottom=1in]{geometry}
\usepackage{authblk}

\PassOptionsToPackage{hypertexnames=false}{hyperref}  
\usepackage{parskip}

\usepackage[colorlinks=true, linkcolor=blue!70!black, citecolor=blue!70!black,urlcolor=black,breaklinks=true]{hyperref}
\usepackage{microtype}
\usepackage{hhline}

\makeatletter
\newcommand{\neutralize}[1]{\expandafter\let\csname c@#1\endcsname\count@}
\makeatother

\usepackage{algorithm}

\usepackage{natbib}
\bibpunct{(}{)}{;}{a}{,}{,}

\usepackage{amsthm}
\usepackage{mathtools}
\usepackage{amsmath}
\usepackage{bbm}
\usepackage{amsfonts}
\usepackage{amssymb}

\usepackage{xpatch}

\newtheorem*{theorem*}{Theorem}
\newtheorem*{lemma*}{Lemma}
\newtheorem{theorem}{Theorem}
\newtheorem{lemma}{Lemma}

\newtheorem{assumption}{Assumption}
\newtheorem{corollary}{Corollary}

\newtheorem{fact}{Fact}

\newtheorem{example}{Example}
\theoremstyle{remark}

\newtheorem{definition}{Definition}
\theoremstyle{definition}

\newtheorem{remark}{Remark}

\AtBeginEnvironment{remark}{%
  \pushQED{\qed}%
}
\AtEndEnvironment{remark}{\popQED\endexample}

\makeatletter
  \renewenvironment{proof}[1][Proof]%
  {%
   \par\noindent{\bfseries\upshape {#1.}\ }%
  }%
  {\qed\newline}
  \makeatother

\xpatchcmd{\proof}{\itshape}{\normalfont\proofnameformat}{}{}
\newcommand{\proofnameformat}{\bfseries}

\usepackage[nameinlink,capitalize]{cleveref}

\crefformat{equation}{#2(#1)#3}
\Crefformat{equation}{#2(#1)#3}
\usepackage{mdframed}
\usepackage{lipsum}

\newenvironment{fexam}
  {\begin{mdframed}\begin{example}}
  {\end{example}\end{mdframed}}
\Crefformat{figure}{#2Figure #1#3}
\Crefname{assumption}{Assumption}{Assumptions}
\Crefformat{assumption}{#2Assumption #1#3}
\usepackage[customcolors]{hf-tikz}
\usepackage{crossreftools}
\newenvironment{frem}
  {\begin{mdframed}\begin{remark}}
  {\end{remark}\end{mdframed}}
\Crefformat{figure}{#2Figure #1#3}
\Crefname{assumption}{Assumption}{Assumptions}
\Crefformat{assumption}{#2Assumption #1#3}
\usepackage[customcolors]{hf-tikz}
\usepackage{crossreftools}
\usepackage{xparse}

\ExplSyntaxOn
\DeclareDocumentCommand{\XDeclarePairedDelimiter}{mm}
 {
  \__egreg_delimiter_clear_keys: 
  \keys_set:nn { egreg/delimiters } { #2 }
  \use:x 
   {
    \exp_not:n {\NewDocumentCommand{#1}{sO{}m} }
     {
      \exp_not:n { \IfBooleanTF{##1} }
       {
        \exp_not:N \egreg_paired_delimiter_expand:nnnn
         { \exp_not:V \l_egreg_delimiter_left_tl }
         { \exp_not:V \l_egreg_delimiter_right_tl }
         { \exp_not:n { ##3 } }
         { \exp_not:V \l_egreg_delimiter_subscript_tl }
       }
       {
        \exp_not:N \egreg_paired_delimiter_fixed:nnnnn 
         { \exp_not:n { ##2 } }
         { \exp_not:V \l_egreg_delimiter_left_tl }
         { \exp_not:V \l_egreg_delimiter_right_tl }
         { \exp_not:n { ##3 } }
         { \exp_not:V \l_egreg_delimiter_subscript_tl }
       }
     }
   }
 }

\keys_define:nn { egreg/delimiters }
 {
  left      .tl_set:N = \l_egreg_delimiter_left_tl,
  right     .tl_set:N = \l_egreg_delimiter_right_tl,
  subscript .tl_set:N = \l_egreg_delimiter_subscript_tl,
 }

\cs_new_protected:Npn \__egreg_delimiter_clear_keys:
 {
  \keys_set:nn { egreg/delimiters } { left=.,right=.,subscript={} }
 }

\cs_new_protected:Npn \egreg_paired_delimiter_expand:nnnn #1 #2 #3 #4
 {
  \mathopen{}
  \mathclose\c_group_begin_token
   \left#1
   #3
   \group_insert_after:N \c_group_end_token
   \right#2
   \tl_if_empty:nF {#4} { \c_math_subscript_token {#4} }
 }
\cs_new_protected:Npn \egreg_paired_delimiter_fixed:nnnnn #1 #2 #3 #4 #5
 {
  \mathopen{#1#2}#4\mathclose{#1#3}
  \tl_if_empty:nF {#5} { \c_math_subscript_token {#5} }
 }
\ExplSyntaxOff

\XDeclarePairedDelimiter{\supnorm}{
  left=\lVert,
  right=\rVert,
  subscript=\infty
  }

\usepackage[utf8]{inputenc} 
\usepackage[T1]{fontenc}    
\usepackage{url}            
\usepackage{booktabs}       
\usepackage{amsfonts}       
\usepackage{nicefrac}       
\usepackage{microtype}      
\usepackage{authblk}
\usepackage{tocloft}            

\usepackage{enumitem}

\usepackage{breakcites}
\usepackage{dsfont}
\usepackage{etoolbox}
\usepackage{comment}
\newtoggle{draft}
\togglefalse{draft}

\usepackage{color-edits}

\usepackage{mathrsfs}

\usepackage{algorithm}
\usepackage{verbatim}
\usepackage[noend]{algpseudocode}

\usepackage{multicol}

\usepackage{colortbl}
\usepackage{bbm}
\usepackage{setspace}

\usepackage{transparent}

\usepackage{inconsolata}
\usepackage[scaled=.90]{helvet}
\usepackage{xspace}

\usepackage{pifont}

\usepackage{tikz-cd}
    \newcommand{\lp}{\left(}
    \newcommand{\rp}{\right)}

    \newcommand{\E}{{\mathbb E}}
    
    \newcommand{\R}{{\mathbb R}}
    \renewcommand{\P}{{\mathbb P}}

    \newcommand{\PP}{\mathbb{P}}

    \renewcommand{\b}[1]{\boldsymbol{\mathbf{#1}}}
    \renewcommand{\vec}[1]{\b{#1}}

    \newcommand{\cG}{\mathcal{G}}

    \newcommand{\eps}{\epsilon}

    \newcounter{rcnt}[section]

    \def\argmin{\mathop{\rm argmin}}
    
    \newcommand{\Ef}{\mathrm{E}_{f^{*}}}

    \newcommand{\Vol}{\operatorname{Vol}}
    
    \newcommand{\Var}{\operatorname{Var}}

    \newcommand{\Span}{\operatorname{Span}}

    \newcommand{\EE}{\mathbb E}

    \newcommand{\QQ}{\mathbb Q}
    
    \newcommand{\erm}{\widehat{f}_n}
    \newcommand{\ft}{f^{*}}

    \newcommand{\nos}[2]{\left\|{#1}\right\|^{#2}}
    \newcommand{\inner}[1]{\left\langle #1 \right\rangle}
    \def\ddefloop#1{\ifx\ddefloop#1\else\ddef{#1}\expandafter\ddefloop\fi}
    \def\ddef#1{\expandafter\def\csname c#1\endcsname{\ensuremath{\mathcal{#1}}}}
    \ddefloop ABCDEFGHIJKLMNOPQRSTUVWXYZ\ddefloop

    \newcommand{\GE}[1]{\mathcal{G}_{n}({#1})}

    \newcommand{\Sn}{\mathbb{S}^{n-1}}

\newcommand{\bana}{(\cB(\cX),\| \cdot \|)}
\newcommand{\banad}{(\cB(\cX),\| \cdot \|_*)}
\newcommand{\lerm}{\widehat{w}_p}
\newcommand{\lft}{w_*}
\newcommand{\ermp}{\widehat{w}_n}

\title{Minimum Norm Interpolation via the Local Theory of Banach Spaces: The Role of $2$-Uniform Convexity}
\author[1]{Gil Kur}
\author[2]{Pierre Bizeul}

\affil[1]{Institute for Machine Learning, ETH Z\"urich}
\affil[2]{Department of Mathematics, Weizmann Institute of Science}
\begin{document}
\date{}
\maketitle

\begin{abstract}
The minimum-norm interpolator (MNI) framework has recently attracted considerable attention as a tool for understanding generalization in overparameterized models, such as neural networks. In this work, we study the MNI under a $2$-uniform convexity assumption, which is weaker than requiring the norm to be induced by an inner product; in this setting, the MNI typically does not admit a closed-form solution. At a high level, we show that this condition yields an upper bound on the MNI bias in both linear and nonlinear models. We further show that this bound is sharp for overparameterized linear regression when the norm's unit ball is in isotropic or John's position and the covariates are i.i.d.\ sub-Gaussian, for example, when each covariate vector has i.i.d.\ Rademacher entries. Finally, under the same assumption on the covariates, we prove sharp generalization bounds for the $\ell_p$-MNI when $p \in \bigl(1 + C/\log d, 2\bigr]$. To the best of our knowledge, this is the \emph{first} work to establish sharp bounds for non-Gaussian covariates in linear models when the norm is not induced by an inner product. This work is deeply inspired by classical work on $K$-convexity and recent work on the geometry of $2$-uniformly convex and isotropic convex bodies.\footnote{A preliminary work 
``Minimum Norm Interpolation Meets The Local Theory of Banach Spaces'' appeared at the 2024 International Conference on Machine Learning. This work has additional results, a detailed discussion, simplified proofs, and improved accessibility for non-experts in functional analysis. This is the first work in a line of work on the relationship between the local theory of Banach spaces and Minimum Norm Interpolation.}
\end{abstract}

\section{Introduction}

Experiments with neural networks have revealed a phenomenon that defies traditional statistical intuition, according to which regularization is critical for large models when fitting noisy data. Instead, in the overparameterized regime, interpolators that achieve zero training error can still generalize well and do not benefit from sacrificing data fit; in other words, interpolation can be harmless (see, e.g., the experiments in \cite{nakkiran2021}). The effort to explain this counterintuitive observation (see, e.g., \cite{zhang2021understanding,belkin2018understand}) has given rise to a line of work on \emph{benign overfitting} that
seeks to establish 
generalization bounds for \emph{interpolating} overparameterized models. In regression---the focus of this paper---interpolation corresponds to achieving zero squared loss on the training data.

As there are infinitely many interpolating solutions in overparameterized regimes, the specific choice of interpolator can drastically affect generalization performance.
The most commonly studied choice in the literature is the minimum-norm interpolator (MNI)---a natural choice when the ground truth has a simple structure, such as a small norm in a Banach space.
For additional motivation, first-order methods for the squared loss, initialized at zero, typically exhibit an implicit bias toward such
minimum-norm solutions (i.e., they converge to them)~\cite{oravkin2021optimal,shamir2022implicit,efron-LARS}. 

The statistical analysis of the mean squared error (MSE) of the MNI has primarily focused on overparameterized linear regression with the $\ell_2$ norm and on reproducing kernel Hilbert spaces (RKHSs), where the MNI has a closed-form solution; in both cases, the norm is induced by an inner product.
In overparameterized linear models, the literature includes asymptotic results on the \textbf{$\ell_2$}-MNI in the proportional regime $d/n\to\gamma \in (1,\infty)$, where $\gamma$ is an absolute constant and $d$ and $n$ denote the numbers of parameters and samples, respectively (cf. \cite{hastie2022surprises,ghorbani-2021,mei2022generalization}), as well as non-asymptotic results in the overparameterized regime $d/n\to\infty$.

In more detail, \cite{bartlett2020benign,tsigler2023benign,lecue2023geometrical,chinot2020robustness,muthukumar2020} prove finite-sample bounds for the $\ell_2$-MNI that vanish
when the eigenvalues of the data covariance matrix decay rapidly.
These proofs exploit the inner-product structure, which allows an explicit analysis of the closed-form solution. Moreover, consistency heavily relies on eigenvalue decay in the covariate distribution, because the $\ell_2$-MNI cannot capture ``structural assumptions'' and, under isotropic covariates in high dimensions, suffers from high bias in the bias--variance decomposition of the MSE.

Moving away from the $\ell_2$-MNI, \cite{koehler2021uniform} introduced a framework for analyzing MNIs in overparameterized linear regression when no closed-form solution is available (i.e., when the norm is not induced by an inner product). Their approach applies only to Gaussian covariates and crucially relies on the Convex Gaussian Minimax Theorem (CGMT)
\cite{gordon1988,thramp2015}.
Follow-up works \citep{donhauser2022fast,wang2022tight} used this technique
to establish sharp rates for the $\ell_p$-MNI when $p \in [1,2]$; see \Cref{Ex:P} below for further details. We also refer to \cite{chinot2020robustness,muthukumar2020} for work on the $\ell_1$-MNI.

To our knowledge, all works analyzing MNIs with non-inner-product norms rely on the Gaussianity of the covariates, and their proofs use the CGMT, which offers no geometric insight and mostly leads to direct computations. In nonlinear models, the phenomenon of harmless interpolation, or benign overfitting, is much less well understood. Most work so far has considered settings in which linearization is a good approximation, such as specific RKHSs (cf. \cite{liang2020multiple,aerni23}), or local interpolation schemes (cf. \cite{belkin2019does,belkin2018overfitting}). Despite these efforts, a comprehensive theoretical understanding of nonlinear models remains open.
The main motivation of this work is to address the following questions:
\begin{enumerate}
    \item Can we generalize bounds on the statistical performance of MNIs under a condition less restrictive than requiring an inner-product norm?
    \item In linear models, can we develop a geometric analysis that extends beyond Gaussian covariates, for example, to sub-Gaussian covariates?
\end{enumerate}

To address the first question, we relate the generalization properties of MNIs for regression under additive \textbf{isotropic} Gaussian noise\footnote{In most of our results, this assumption can be relaxed to isotropic log-concave or Rademacher noise.} to the (local) geometric properties of the Banach space. These include uniform convexity and smoothness, type and cotype, and the position of the norm; see the books of \cite{artstein2015asymptotic,artstein2022asymptotic} and the references therein. Roughly speaking, we upper-bound the bias of the MNI when its underlying norm is $2$-uniformly convex, and this bound is sharp in linear models when the norm is in an appropriate position, e.g., isotropic or John's position.

To answer the second question, we provide a sharp bound on the MSE of the $\ell_p$-MNI under sub-Gaussian (s.g.) covariates, extending the Gaussian-covariate result of \cite{donhauser2022fast}. To the best of our knowledge, this is the \textbf{first} work to prove benign overfitting for s.g. covariates in linear models when the norm is not an \textbf{inner product}.

Before proceeding, we mention that our work is deeply inspired by the influential work of Maurey and Pisier on probability in Banach spaces~\cite{maurey1976series,pisier1977nouveau} and by more recent work on the geometry of $2$-uniformly convex bodies~\cite*{klartag2008volume,milman2009gaussian,Paorlp}. We also build on recent developments concerning isotropic convex bodies~\cite{milman2015mean,klartag2025affirmative,bizeul2025slicing}.
 
\paragraph{Organization:}
In \Cref{ss:MNI}--\Cref{ss:pos}, we provide the background needed for the overview of our contributions in \Cref{ss:con}. In \Cref{S:Seittng}, we provide additional preliminaries. In \Cref{sec:Results}, we state all our results. We then discuss them in \Cref{sec:Discussion}, present open problems in \Cref{S:OP}, give a proof sketch of the main theorems in \Cref{S:PI}, and present all proofs in \Cref{S:Pfs}.
\paragraph{Notation:} Throughout this work, $C,C_1,C_2 \geq 0$ and $c_1,c_2,c_3 \in (0,1)$ are absolute constants. For an arbitrary argument vector $a$, we write $C(a) \geq 0$ and $c(a) \in (0,1)$ for constants that depend only on $a$; these constants \textbf{may} change from line to line.
For any measure $\QQ$ on $\cX$, we use $\| \cdot \|_{\QQ}$ to denote the $L_2(\QQ)$-norm.
We use standard Landau notation (also known as big-$O$, big-$\Omega$, and big-$\Theta$ notation), which hides only absolute constants and, in particular, constants independent of all model parameters. Similarly, $\asymp,\lesssim,\gtrsim$ denote equalities and inequalities up to a universal multiplicative constant. $B_d$ denotes the Euclidean ball in $\R^d$ of radius one, and $\mathbb{S}^{d-1} = \partial B_d$ is its unit sphere. The measure $\gamma_n$ denotes the Gaussian measure on $\R^n$, and $\sigma_n$ denotes the uniform measure on $\Sn$. Whenever $p,q \in [1,\infty)$ appear together, they satisfy $1/p+1/q = 1$.

\subsection{Minimum-Norm Interpolator}\label{ss:MNI}

Let $\PP$ be a probability measure on a domain $\cX$. We consider a Banach space $\bana$ with $\cB(\cX) \subset L_2(\PP)$, and denote its unit ball by
\begin{equation}
    \cF:= \{ f \in \cB(\cX): \|f\| \leq 1 \}.
\end{equation}
Given data points

$\{(x_i,y_i)\}_{i=1}^{n}$, the MNI is defined by
\begin{equation}
\label{eq:minnorm}
     \erm(\vec x,\vec y):=\argmin_{f \in \cB(\cX): \ \vec{f}=\vec{y}}\| f \|,
\end{equation}
where $\vec f:= (f(x_1),\ldots,f(x_n))$, $\vec y = (y_1,\ldots,y_n)$, and $\vec x:=(x_1,\ldots,x_n)$.
In words, among all functions in $\cB(\cX)$ that interpolate the observations $\vec y$ at the data points $\vec x$, we choose one with the smallest norm. In general, the solution in \eqref{eq:minnorm} may not be unique, but a mild uniform convexity of the norm guarantees uniqueness. From a computational perspective, depending on the setting, the solution may be computed efficiently or may arise as the implicit bias of a first-order method for a convex optimization problem.

Note that the MNI differs from a standard constrained ERM, which searches only over the function class $\cF$; for example, the constrained least-squares estimator is
\begin{equation}\label{Eq:LSE}
\bar{f}_n  \in  \argmin_{f \in \cF}\sum_{i=1}^{n}(y_i - f(x_i))^2.
\end{equation}

\textbf{In this work}, we study the statistical performance of the MNI in the well-specified regression model
\begin{equation}\label{Eq:Model}
     Y = \ft(X) + \xi,
\end{equation}
where $X \sim \PP$, $\xi \sim N(0,1)$, and $\ft \in \cF$. Our dataset $\cD := \{(X_i,Y_i)\}_{i=1}^n$ consists of $n$ i.i.d. samples from \eqref{Eq:Model}. We use the notation $\vec X = (X_1,\ldots,X_n)$, $\vec Y = (Y_1,\ldots,Y_n)$, and $\vec \xi = (\xi_1,\dots,\xi_n)$ and, unless stated otherwise, write $\erm:=\erm(\cD)$. Throughout this work, we make two assumptions on our model:
\begin{assumption}\label{A:One}
    Both the $L_2(\PP)$ and $\|\cdot\|$ norms of $\ft$ are of order one, i.e., $\| \ft \|_{\PP} \asymp \|\ft\| \asymp 1$.
\end{assumption}
This assumption rules out the pathological case in which the norms of the true signal $\ft$ vanish with the number of samples $n$. In high dimensions, such an assumption is crucial. Next, we need an ``inductive bias'' assumption:
\begin{assumption}\label{A:NS}
There exists a sufficiently large absolute constant $C \geq 1$ such that, with probability at least $1-n^{-2}$ over $\vec X,\vec \xi$, it holds that
\begin{equation}\label{Eq:ANS}
 \|\erm(\vec X,\vec \xi)\| \geq C.
\end{equation}
\end{assumption}
Under this assumption, if $C \geq 1$ is sufficiently large, the norm exhibits an inductive bias toward $\ft$. To see this, note that $\|\erm(\vec X,\vec \ft)\| \lesssim 1$ with probability one, since $\ft$ is a feasible interpolator. The inequality in the assumption indicates that the norm has an ``inductive bias'' toward the ground truth $\ft$: interpolating pure noise requires a larger norm than interpolating the underlying true signal $\ft \in \cF$. A motivating example is the $\ell_p$-MNI in overparameterized linear regression for $p \geq 1$; see \Cref{Ex:P} below for further details.
\subsection{A ``new'' mean squared error decomposition}
First, for any $f \in \cB(\cX)$, let $\cP_{\ft}(f)$ denote the orthogonal projection of $f$ onto $\operatorname{span}\{\ft\}$ in $L_2(\PP)$, i.e.,
\[
\cP_{\ft}(f):=\inner{f,\frac{\ft}{\|\ft\|_{\PP}}}_{L_2(\PP)}\frac{\ft}{\|\ft\|_{\PP}}.
\]
Let $(\ft)^{\perp}$ denote the subspace orthogonal to $\ft$. We consider the following decomposition of the MSE:
\begin{equation}
\begin{aligned}
\label{eq:decompose}
\E_{\cD} \|\erm - \ft\|_{L_2(\PP)}^2
&=\underbrace{\|\E_{\cD} \erm - \ft\|_{\PP}^2}_{B^2(\erm)} + \underbrace{\E \|\erm - \E_{\cD} \erm\|_{\PP}^2}_{\mathrm{Var}_{\cD}[\erm]} \\
&= B^2(\erm) + \mathrm{Var}_{\vec X}\!\left[\E_{\vec \xi}[\erm\mid\vec X]\right] + \underbrace{\E_{\vec X}\mathrm{Var}_{\vec \xi}[\erm\mid\vec X]}_{T_2} \\
&= \underbrace{\E\|\cP_{\ft}(\E_{\vec \xi}[\erm\mid\vec X]) - \ft\|_{\PP}^2}_{E_1} + \underbrace{\E \|\cP_{(\ft)^{\perp}}(\E_{\vec\xi}[\erm\mid\vec X])\|_{\PP}^2}_{E_2} + T_2
\\&:= T_1 + T_2.
\end{aligned}
\end{equation}
Here we used the bias--variance decomposition, the law of total variance, and Pythagoras' theorem in $L_2(\PP)$. We refer to $T_1$ and $T_2$ as the \textbf{structural error} and the \textbf{noise-effect error}, respectively.
For regularized estimators in the underparameterized setting, it is more natural to analyze the bias and variance separately when bounding the MSE. Roughly speaking, in that case, the bias can be interpreted as an approximation error of $\ft$ induced by the function space, and the variance as the effect of noise.
However, in overparameterized models, i.e., when $\erm$ interpolates the observations $\vec Y$, we need a different approach. Define $f_{\vec X}:=\cP_{(\ft)^{\perp}}(\E_{\vec\xi}[\erm\mid\vec X])$ and note that
\begin{equation}
    \label{eq:decomplevel2}
    T_1=\E\|\cP_{\ft}(\E_{\vec \xi}[\erm\mid\vec X]) - \ft\|_{\PP}^2 + \E \|f_{\vec X}\|_{\PP}^2,
\end{equation}

The first term in \eqref{eq:decomplevel2}, denoted by $E_1$, measures how much ``energy'' the MNI preserves from the original signal $\ft$; i.e., it measures the shrinkage of $\ft$ due to undersampling. Intuitively, in overparameterized models, this term is typically nonzero, and elementary convexity (at least under sufficient regularity of the norm) shows that the MNI must shrink $\ft$.

Since $\E_{\vec \xi}[\erm\mid\vec X]$ interpolates $\vec{\ft}$ and $\cP_{\ft}(\E_{\vec \xi}[\erm\mid\vec X])=(1-\lambda)\ft$ for some $\lambda \leq 1$, the function $f_{\vec X}$, which is orthogonal to $\ft$ in $L_2(\PP)$, interpolates $\lambda \vec{\ft}$. Hence, the second term in \eqref{eq:decomplevel2}, denoted by $E_2$, measures how much $L_2(\PP)$-energy is added by $f_{\vec X}$, in expectation, to the MNI. The $2$-uniform convexity helps us upper-bound $\|f_{\vec X}\|$, and techniques from the empirical-process literature allow us to upper-bound the $L_2(\PP)$-norm of $f_{\vec X}$. Finally, the noise-effect error $T_2$ measures the energy added to the MNI by the noise in the observations.
\subsection{Gaussian and Rademacher complexities}
Next, we recall two fundamental quantities in statistical learning theory (cf. \cite{bartlett2002rademacher,mendelson2014learning}) and in high-dimensional geometry, known as Gaussian and Rademacher complexities.
\begin{definition}\label{Def:GW}
The Gaussian complexity of a set $\cH \subset L_2(\PP)$ is defined by
\[
    \cG_n(\cH):= n^{-1}\E_{\vec X,\vec \xi}\sup_{h \in \cH} \inner{\vec h,\vec \xi}.
\]
Here, $\vec \xi \sim N(0,I_n)$ is independent of $\vec X$, and $\inner{\vec x,\vec y}$ denotes the $\ell_2$ inner product between $\vec x,\vec y \in \R^n$.
\end{definition}
\begin{definition}
The Rademacher complexity of a function class $\cH \subset L_2(\PP)$ is defined by
\[
 \cR_n(\cH):= n^{-1}\E_{\vec X,\vec \eps}\sup_{h \in \cH} \inner{\vec h,\vec \eps}.
\]
Here, $\vec \eps$ has i.i.d. Rademacher coordinates and is independent of $\vec X$.
\end{definition}
To motivate these quantities statistically, recall that, under suitable assumptions, the risk of the constrained least-squares estimator over $\cF$ defined above is bounded by the Gaussian complexity, i.e.,
\[
    \sup_{\ft \in \cF}\E_{\cD} \|\bar{f}_n - \ft\|_{\PP}^2 \lesssim \GE{\cF},
\]
see for example \cite{chatterjee2014new}. The Rademacher complexity also has a statistical motivation. For example, the classical Giné--Zinn symmetrization lemma (cf. \citep[Thm.~2.1]{koltchinskii2011oracle}) implies that
\begin{equation}\label{Eq:Symetrization}
    2^{-1}\cR_n(\cH_{c})\leq \E \sup_{f \in \cH} |\int f \, d(\PP_n - \PP)| \leq 2 \cR_n(\cH),
\end{equation}
where $\cH_{c}:= \{h - \int h \, d\PP: h \in \cH\}$, and $\PP_n$ denotes the empirical measure over $\vec X$, i.e., $\PP_n := n^{-1}\sum_{i=1}^{n}\delta_{X_i}$.

The contraction lemma (cf. \citep[Thm.~2.3]{koltchinskii2011oracle}) shows that, for any $L$-Lipschitz function $\psi$,
\begin{equation}\label{Eq:contraction}
     \cR_n(\psi \circ \cH) \lesssim L \cdot \cR_n(\cH) \text { where }  
     \psi \circ \cH := \{ \psi \circ h: h \in \cH\}.
\end{equation}

Jensen's inequality together with the maximal inequality for $n$ i.i.d. standard Gaussian variables implies that
 $$
   \cR_n(\cF) \lesssim  \cG_n(\cF) \lesssim \sqrt{\log(n)} \cR_n(\cF).
 $$

\subsection{On Curvature Notions in Banach Spaces (Part I)}
First, we introduce $q$-uniform convexity ($UC(q)$), for $q \in [2,\infty)$, and its dual notion, $p$-uniform smoothness ($US(p)$), for $p \in (1,2]$. These notions first appeared in the influential work of \cite{clarkson1936uniformly}; see also
 \citep{ledoux2013probability}.
 
\begin{definition}\label{Def:UC}
\label{def:UC}
For $q \in [2,\infty)$, the Banach space $\bana$ is \textbf{$q$-uniformly convex} with constant $t>0$ if, for all $f,g \in \cB(\cX)$,
\[
   \nos{\frac{f+g}{2}}{q} + t\nos{\frac{f-g}{2}}{q} \leq \frac{\|f\|^q + \|g\|^q}{2}.
\]
\end{definition}
\textbf{Our work focuses} on the case $q=2$, i.e., when the norm is $2$-uniformly convex ($UC(2)$) with constant $t \geq 0$. This condition plays a key role in high-dimensional geometry because it implies concentration of Lipschitz functionals; see also \cite{gromov1983topological,gromov1987generalization}. To give some intuition for a $UC(2)$ norm, one may think of it in $\R^d$ as providing a lower bound on the smallest singular value of the Hessian on its unit ball. By the parallelogram identity, any inner-product norm is $UC(2)$ with constant $1$. The $\ell_p$ norm, for $p \in (1,2]$ in any dimension, is $UC(2)$ with constant $t = \Theta(p-1)$ (cf. \citep[Ch.~10]{pisier2016martingales}). To present our results, we also need the dual notion of uniform convexity, namely, uniform smoothness.

 \begin{definition}
\label{def:US}
For $p \in (1,2]$, the Banach space $\bana$ is \textbf{$p$-uniformly smooth} with constant $s > 0$ if, for all $f,g \in \cB(\cX)$,
\[
   \nos{\frac{f+g}{2}}{p} + s\nos{\frac{f-g}{2}}{p} \geq \frac{\|f\|^p + \|g\|^p}{2}.
\]
\end{definition}

The fact that the optimal exponents are $q=2$ for uniform convexity and $p=2$ for uniform smoothness is already visible in dimension one. Moreover, the optimal constants are $t \le 1$ for uniform convexity and $s \ge 1$ for uniform smoothness, and these bounds are sharp. They are attained in Hilbert spaces, where both inequalities hold with equality by the parallelogram identity. We conclude this part with a known classical fact (cf. \citep{lindenstrauss2013classical}): 
\begin{fact}
$\bana$ is $UC(q)$ with constant $t>0$ $\iff$ $\banad$ is $US(p)$
with constant $t^{-p/q}=t^{-(p-1)}$.
\end{fact}

\subsection{Two Motivating Examples}\label{S:Ex}
For concreteness, we now provide two examples that motivated this work; see \cite{ledoux2013probability,pisier1999volume} for more examples. The first is the standard linear regression model with respect to $\ell_p$ norms for $p\in [1,\infty)$.

\vspace{3mm}
\begin{fexam}\label{Ex:P}
Assume that $d > n$, and consider the $\ell_p$ norm in $\R^d$, i.e., its unit ball is defined via
\[ 
\cF:= B_p^d := \{w \in \R^{d}: \|w\|_{p} \leq 1 \}. 
\]
The {\bf $\ell_p$-MNI} is defined via
\[
    \lerm := \argmin_{\vec Xw = \vec Y}\|w\|_{p}.
\]
Throughout, we assume that $X = (X^{(1)},\ldots,X^{(d)})$ is an isotropic vector in $\R^d$ with i.i.d. \emph{symmetric} entries with bounded sub-Gaussian norm, i.e., $X^{(1)},\ldots,X^{(d)}$ are i.i.d., $X^{(1)} \overset{d}{=} -X^{(1)}$, and
\[
\text{(s.g.)} \qquad \Pr(|X^{(1)}| \geq t) \leq 2\exp(-t^2/C^2), \qquad t \geq 0,
\]
for some constant $C>0$. Note that in this model $\|\cdot\|_{\PP} = \|\cdot\|_2$.
In \Cref{T:VarLinOpt} below, we show that with high probability
\[
\|\lerm(\vec X,\vec \xi)\| \asymp \sqrt{n} \cdot \max\left\{\sqrt{p-1}, \log(d/n)^{-1/2}
\right\} \cdot d^{1/p-1}.
\]
Therefore, to satisfy \Cref{A:NS}, we also require that for $C_1 \geq 0$
\[
n\log(d)^{C_1} \lesssim d \lesssim n^{q/2}\log(d)^{-C_1},
\]
and assume that the signal $\lft \in \R^d$ is $O(1)$-sparse, i.e., $\|\lft\|_0 = O(1)$. Note that, as $p \to 2$, we cannot satisfy this assumption, whereas, as $p \to 1$, the dimension $d$ can be more than polynomial in $n$.
\end{fexam}

The second example is the Sobolev norm, which plays a key role in nonparametric statistics; see \cite{tysbakovnon,gine2021mathematical}. In this work, we do not focus on nonparametric MNIs; we refer to the recent work of \cite{karhadkar2026harmful} and the references therein. See also \Cref{ss:renorming} below.
\vspace{3mm}
\begin{fexam}\label{Ex:NP}
Fix a domain $\Omega \subset \R^d$ and let $\PP = \mathrm{Unif}(\Omega)$, $p \in [1,\infty)$, and $k \in \mathbb{N}$. For a multi-index $\alpha=(\alpha_1,\ldots,\alpha_d) \in \mathbb{N}_0^d$ with $|\alpha| \leq k$, define
\[
D^{\alpha}f := \frac{\partial^{|\alpha|}f}{\partial x_1^{\alpha_1}\cdots \partial x_d^{\alpha_d}}.
\]
The Sobolev space $W^{k,p}(\Omega)$ consists of functions whose weak derivatives $D^{\alpha}f$, for $|\alpha| \leq k$, lie in $L_p(\PP)$, and is equipped with the norm
\[
    \|f\|_{k,p}:= \left(\sum_{\alpha \in \mathbb{N}_0^d:\,|\alpha| \leq k} \| D^{\alpha}f \|_{L_p(\PP)}^{p}\right)^{1/p}.
\]
The MNI associated with this norm is well defined only when $W^{k,p}(\Omega)$ is compactly embedded in $L_2(\PP)$.
\end{fexam}

 \subsection{Positions}\label{ss:pos}
For any \emph{symmetric} convex set $K \subset \R^d$, i.e., $K = -K$, with nonempty interior, its Minkowski norm is defined by
\[
\|z\|_{K}:=\inf\{r>0: z\in rK\}.
\] 
Its polar body, $K^{\circ}$, is defined by
\[
K^{\circ} := \{ x \in \R^{d}: \sup_{y \in K}\inner{x,y} \leq 1 \}.
\]
The dual norm $\|\cdot\|_{K^{\circ}}$ is also referred to as the support function $h_K$ of the original body $K$, because, for a unit vector $\vec \xi$, it measures the distance from the origin to the supporting hyperplane of $K$ in that direction.

The notion of a \emph{position} is fundamental in high-dimensional convex geometry. Given a convex body $K\subset\R^d$, which we assume to be symmetric, one seeks an invertible linear map $T$ such that the image $TK$ is more ``regular'' with respect to a chosen functional.

In \textbf{finite-dimensional} Banach spaces, which include linear models, the \textbf{position} plays a key role in the statistical performance of the MNI. Here, we introduce a few standard examples:
\begin{itemize}
\item  $K$ is in \textbf{John's} position if among all ellipsoids contained in $K$, the ball $B_2^d$ has maximal volume.

\item To introduce Pisier's $\mathbf{\ell}$-position, we define the mean norm and width (respectively) by
\[
 M_s(K) :=\int_{\mathbb{S}^{d-1}} \|\vec \xi\|_{K} \, d\sigma_d \quad \text{and} \quad M_s^*(K) = M_s(K^{\circ}) =\int_{\mathbb{S}^{d-1}} \|\vec \xi\|_{K^{\circ}} \, d\sigma_d.
\]
It is not hard to verify that, for any symmetric convex set $K$,
\[
    1\leq M_s(K)M_s^*(K).
\]
At the same time, one intuitively expects $M_s^*(K)M_s(K) \lesssim 1$ when $K$ does not lie in a low-dimensional subspace and is ``balanced.'' Remarkably, the works of \cite{pisier1977nouveau,figiel1979projections} show that, for every $K \subset \R^d$, there exists an invertible linear transformation $T$ such that $TK$ is in $\ell$-position and
\[
    M_s(TK)M_s^*(TK) \lesssim \log(d).
\]

\item $K$ is in (Milman's) \textbf{M}-position if there is a constant $C > 0$, independent of $K$ and $d$, such that
\[
\max\{\cN(K,B_d),\cN(B_d,K),\cN(K^{\circ},B_d),\cN(B_d,K^{\circ})\} \le C^d,
\]
where $\cN(A,B)$ denotes the minimal number of translates of $B$ required to cover $A$. We remark that $M$-position is not unique---even up to rotation.

\item $K$ is in \textbf{isotropic} position if the uniform measure on $K$ has zero mean and identity covariance. Recent progress shows that isotropic position is much closer to $M$- and $\ell$-positions than one might first expect. The resolution of the slicing problem by \cite{klartag2025affirmative,bizeul2025slicing}, building on the work of \cite{guan2024note}, implies that isotropic convex bodies are in $M$-position up to a universal constant. The recent result of \cite{bizeul2025distances} and the mean-width estimate of \cite{milman2015mean} further show that isotropic position is as good as Pisier's $\ell$-position, in the sense that
\[
    M_s(K)M_s^*(K) \lesssim \log^{3}(d).
\]
\end{itemize}

The unit balls of commonly used norms---such as the $\ell_p$ and the top-$k$ norms---satisfy the standard positions up to scaling. Those norms can therefore misleadingly suggest that the choice of position is irrelevant. For general norms, however, the situation is much more subtle, and the position plays a key role in its geometry (cf. \cite{giannopoulos2000extremal}). \textbf{Throughout this work,} we say that a norm is in a certain position when its unit ball is in that position up to normalization, i.e., when $\lambda \cF$ is in that position for some $\lambda > 0$. 

\subsection{Overview of Contributions}\label{ss:con}
\begin{itemize}
    \item In \Cref{T:UC2}, we provide an ``unlocalized'' upper bound on the structural error under the $UC(2)$ assumption and mild regularity conditions. Roughly speaking, we show that
    \[
        T_1^{\Omega} \lesssim (s/t) \cdot \cG_n(\cF)^{p}.
    \]
    where $T_1^{\Omega}$ has the same definition as $T_1$, only the expectation is taken on a regular event $\Omega$, of large probability. See the statement of Theorem \Cref{T:UC2} for a more precise definition. Surprisingly, under nontrivial smoothness ($p=1$), it aligns with the classical unlocalized bound on the MSE of the least-squares estimator (cf. \cite{chatterjee2014new}). This bound should be viewed as the crudest bound on $T_1^{\Omega}$, and therefore we refer to it as the ``basic inequality'' (cf. \cite{van2000empirical}).
   
    \item In \Cref{T:UC2Loc}, we provide sharp rates for the structural error under the $UC(2)$ assumption in linear models when the norm is in isotropic or John's position and the covariates have i.i.d. symmetric sub-Gaussian entries. Roughly speaking, we show that
    \[
         T_1 \lesssim (s/t)^2 \cdot \cG_n(\cF)^{2p}.
    \]
    This bound should be viewed as the most optimistic bound on $T_1$, and therefore it requires quite strong assumptions. Furthermore, in \Cref{C:ExmapleOne}, we show that, for the $\ell_p$-MNI defined in \Cref{Ex:P} with an $O(1)$-sparse ground truth, this bound is tight up to logarithmic factors, i.e.,
    \[
        T_1 \lesssim \frac{d^{2p-2}}{n^{p}},
    \]
    recovering the estimate in \cite{donhauser2022fast}. In \Cref{sec:Discussion}, we discuss the challenges of extending this theorem to weaker notions of curvature; in our context, this means a norm of cotype $2$ (see \Cref{Def:Cotype}) in an appropriate position.
   
    \item In \Cref{T:VarLinOpt}, we provide a sharp bound on the noise-error term $T_2$ of the $\ell_p$-MNI; its proof is the most technical argument in this work. Roughly speaking, we prove that for $1 < C_1 \leq C_2$ that
  \[
     C_1^{\frac{1}{p-1}} \cdot \frac{n}{d} \lesssim T_2 \lesssim   C_2^{\frac{1}{p-1}} \cdot \frac{n}{d}
  \]
    under isotropic Gaussian covariates for $p \geq 1 + \Omega(\log(d)^{-1})$, matching the result of \cite{donhauser2022fast}. \textbf{However,} our proof also holds for covariates with i.i.d. symmetric isotropic sub-Gaussian entries. In \Cref{sec:Discussion}, we discuss the connection of our results to the work of \cite{Paorlp} and the challenges of extending this result to a $UC(2)$ norm in an appropriate position.
    \item In \Cref{T:UC2Var}, we show that, when the norm has cotype $2$, one can obtain a ``reverse'' Efron--Stein inequality (cf. \cite{boucheron2013concentration}) for the expected conditional variance of the MNI, namely, the $T_2$ error term.

\end{itemize}

\section*{Acknowledgments}
GK conducted part of this work during a visit to the IDEAL Institute, hosted by Lev Reyzin and supported by NSF grant ECCS-2217023. He is also grateful to his postdoctoral host, Andreas Krause, for encouraging him to finalize the extended version of this manuscript. On a more personal note, he warmly acknowledges the \textbf{GAFA community} for its support and for many stimulating discussions throughout this project. In particular, he thanks Grigoris Paouris and Alexandros Eskenazis for their many insightful conversations, their generosity with their time, and the inspiring visits that deepened his understanding of geometric functional analysis and helped him become more involved in this community.

We acknowledge Almut R\"odder---GK's former and outstanding intern---for identifying several inaccuracies in the original proof of Theorem \ref{T:VarLinOpt} and for helping to relax the original small-ball assumption. We also deeply thank Reese Pathak for his helpful comments, which improved the manuscript. We further thank Dan Mikulincer and Konstantin Donhauser for their useful suggestions.

\section{Additional Preliminaries}\label{S:Seittng}
We now introduce concepts from high-dimensional geometry that will be used to analyze the generalization properties of the MNI defined in \eqref{eq:minnorm}.

Let $\cF_n$ be the random coordinate projection of $\cF$ associated with $\vec X$, namely the random convex and symmetric set
\begin{equation}\label{Eq:RndPrj}
    \cF_n:=\{(f(X_1),\ldots,f(X_n)):f\in\cF\}\subset\R^n.
\end{equation}
We will repeatedly use the following fact.
\begin{fact}\label{F:NID}
For every realization of $\vec X$ and every $\vec z\in\R^n$ for which the interpolation problem is feasible,
\begin{equation}\label{Eq:NormIncu}
    \|\erm(\vec X,\vec z)\|=\|\vec z\|_{\cF_n}.
\end{equation}
\end{fact}
Next, for a centrally symmetric convex body $K\subset\R^n$, define the Gaussian mean of its induced norm by
\[
   M_g(K):=\int_{\R^n}\|\vec\xi\|_K\,d\gamma_n(\vec\xi).
\]
We write $M_g^*(K):=M_g(K^\circ)$ for the Gaussian mean of the dual norm. By the polar decomposition of a standard Gaussian vector,
\begin{equation}
    M_g(K)
    =\E\|\vec\xi\|_2\,M_s(K)
    =\left(\sqrt n+O(n^{-1/2})\right)M_s(K).
\end{equation}

Throughout the paper, we use the shorthand notation
\[
    M_n(\cF):=\E_{\vec X}M_g(\cF_n)
    \qquad\text{and}\qquad
    M_n^*(\cF):=\E_{\vec X}M_g^*(\cF_n),
\]
and define
\[
    R_{MM^*}(\cF)
    :=\frac{M_n(\cF)M_n^*(\cF)}{n}
    \asymp
    \E_{\vec X}M_s(\cF_n)\cdot
    \E_{\vec X}M_s^*(\cF_n).
\]
The quantity $R_{MM^*}$, motivated by the classical $MM^*$ estimate (cf. \cite{artstein2015asymptotic}), plays a key role in the statistical performance of the MNI\@. For a centrally symmetric convex body $K\subset\R^n$, let
\[
    b(K):=\sup_{\theta\in\mathbb S^{n-1}}\|\theta\|_K
    \qquad\text{and}\qquad
    R_{bM^*}(K):=b(K)M_s^*(K).
\]
For the random coordinate projection, we use the averaged quantity
\[
    R_{bM^*}(\cF)
    :=\E_{\vec X}b(\cF_n)\cdot
      \E_{\vec X}M_s^*(\cF_n).
\]
Thus, $1/b(K)$ is the Euclidean inradius of $K$, and $R_{MM^*}(\cF)\lesssim R_{bM^*}(\cF)$. The factor $b(K)$ is needed when applying the CCP because it is the Lipschitz (in terms of $\ell_2$-norm) constant of the map $x\mapsto\|x\|_K$.

Next, we define the localization radius of the MNI,  designed to capture the correct scale of the term $T_2$.
\begin{definition}
For sufficiently large constant $C > 1$, define $r_*:=r_*(\cF,n,\PP)$ to be
\begin{equation}\label{Eq:localR}
    r_*:=\inf\left\{r\geq0:
    M_n\bigl(\{f\in\cF:\|f\|_{\PP}\leq r\}\bigr)
    \leq C \cdot M_n(\cF)\right\}.
\end{equation}
\end{definition}
The numerical factor $2$ is arbitrary and can be replaced by any fixed constant larger than $1$. Under suitable regularity, standard localization arguments (see \cite{kur2021minimal}) yield
\[
    T_2
    =\E_{\vec X}\Var_{\vec\xi}\left[\erm\mid\vec X\right]
    \gtrsim
    \bigl(r_*M_n(\cF)\bigr)^2.
\]
Intuitively, this says that a typical minimum-norm interpolator of pure noise must have squared $L_2(\PP)$-norm at least of order $\bigl(r_*M_n(\cF)\bigr)^2$. For Hilbertian geometry, and more generally for $US(2)$ norms with $s=\Theta(1)$, suitable regularity also yields
\[
    T_2\asymp\bigl(r_*M_n(\cF)\bigr)^2.
\]
In the setting of \Cref{Ex:P}, this can be verified directly for $2\leq p<\infty$. One would hope that, under $2$-uniform convexity and in a suitable position,
\[
    T_2\lesssim C(t)\bigl(r_*M_n(\cF)\bigr)^2.
\]
To illustrate this quantity, consider the overparameterized model in \Cref{Ex:P} with $1<p\leq2$. In the relevant regime,
\[
  M_n(B_p^d)
  \asymp
  \sqrt n\,\max\left\{\sqrt{p-1},\log(d/n)^{-1/2}\right\}d^{1/p-1}.
\]
As shown in the proof of \Cref{T:VarLinOpt}, when $p\geq1+\Omega(1/\log d)$,
\[
    r_*=d^{-1/p+1/2}\exp\left(\Theta\left(\frac{1}{p-1}\right)\right).
\]
Since $\ell_p$ is $UC(2)$ with $t=\Theta(p-1)$, this indicates that the dependence of $C(t)$ may be exponential in $1/t$.
\subsection{Regularity Assumptions}
We introduce two regularity assumptions on the function class and the covariate distribution. The first ensures that the Gaussian mean and dual mean of $\cF_n$ are, with high probability, comparable to their expectations. This is useful because we do not assume that $\cF$ is uniformly bounded, has a bounded envelope, or that the covariates satisfy stronger structural conditions.
\begin{assumption}\label{A:DeviationOfSuperma}
With probability at least $1-n^{-2}$ over $\vec X$, the random set $\cF_n$ satisfies
\[
    M_g(\cF_n)\asymp M_n(\cF)
    \qquad\text{and}\qquad
    M_g^*(\cF_n)\asymp M_n^*(\cF).
\]
\end{assumption}

\begin{definition}
The lower isometry remainder $\cI_L(n,\cH,\PP)$ is the infimum of all $\delta\geq0$ such that, with probability at least $1-n^{-2}$ over $\vec X$,
\[
    \forall f\in\cH,\qquad
    \|f\|_{\PP}^2\leq4\|f\|_{\PP_n}^2+\delta.
\]
\end{definition}

Our second regularity assumption controls this remainder and therefore lets us transfer empirical $L_2$ control to population $L_2$ control.
\begin{assumption}\label{A:SB}
There exists a universal constant $C_4>0$ such that
\[
    \cI_L(n,\cF,\PP)\leq C_4\cG_n(\cF)^2.
\]
\end{assumption}
For comparison, symmetrization and the contraction principle give
\[
    \E_{\vec X}\sup_{f\in\cF}
    \left|\|f\|_{L_1(\PP)}-\|f\|_{L_1(\PP_n)}\right|
    \lesssim \cR_n(\cF).
\]
The small-ball method of \cite{mendelson2014learning} shows that, under suitable regularity, one can strengthen such first-moment control to the lower $L_2$ estimate
\[
    \|f\|_{\PP}^2
    \lesssim
    \|f\|_{\PP_n}^2+\cR_n(\cF)^2,
    \qquad \forall f\in\cF,
\]
with high probability.\footnote{In some cases, the lower isometry remainder can be significantly smaller than the Rademacher complexity $\cR_n(\cF)$. This occurs, for instance, for Donsker classes (see \cite{van2000empirical}; for a precise formulation, see \citep[Def.~2.1]{mendelson2014learning}), and is commonly referred to as a localization phenomenon. In high-dimensional settings, or for non-Donsker classes, and in particular under $2$-uniform convexity in a suitable position, the localized bound typically matches $\cR_n(\cF)$ up to a multiplicative constant depending on the $UC(2)$ constant $t \geq 0$.}
A standard sufficient route to such a bound is the small-ball property
\[
    \forall f\in\cF:\qquad
    \Pr_{X\sim\PP}\left(|f(X)|\geq c_1\|f\|_{\PP}\right)\geq c_2,
\]
for constants $c_1,c_2\in(0,1)$, together with the corresponding localized-complexity control; see \cite{mendelson2017extending}. This applies, for example, to classes of linear functionals under isotropic i.i.d.\ sub-Gaussian covariates.

\subsection{Curvature Notions in Banach Spaces: Part II}
\label{subsec:Curvature_B_spaces}

Following the first part, we note that uniform convexity and uniform smoothness are strong pointwise notions of curvature: their defining inequalities must hold for every pair of elements in the Banach space. We now introduce the weaker, averaged notions of \emph{type} $p\in[1,2]$ and \emph{cotype} $q\in[2,\infty)$.
 
\begin{definition}\label{Def:Cotype}
The Banach space $\bana$ has \textbf{cotype} $q\in[2,\infty)$ with lower constant $t>0$ if, for every $m\geq1$ and $f_1,\ldots,f_m\in\cB(\cX)$,
\[
   t^q\sum_{i=1}^m\|f_i\|^q
   \leq
   \E_{\vec\eps}\left\|\sum_{i=1}^m\eps_i f_i\right\|^q,
\]
where $\eps_1,\ldots,\eps_m$ are independent Rademacher random variables.
\end{definition}
\begin{definition}
The Banach space $\bana$ has \textbf{type} $p\in[1,2]$ with upper constant $s>0$ if, for every $m\geq1$ and $f_1,\ldots,f_m\in\cB(\cX)$,
\[
   \E_{\vec\eps}\left\|\sum_{i=1}^m\eps_i f_i\right\|^p
   \leq
   s^p\sum_{i=1}^m\|f_i\|^p.
\]
\end{definition}

Hilbert spaces have both $T(2)$ and $CO(2)$ with constants $s=t=1$. In the converse direction, Kwapie\'n's theorem (for a short proof, see \cite{yamasaki1984simple}) implies that a Banach space with type $2$ and cotype $2$ is isomorphic to a Hilbert space. More quantitatively, if the corresponding upper and lower constants are $s$ and $t$, respectively, then every finite-dimensional subspace $E\subset\cB(\cX)$ satisfies
\[
    d_{BM}(E,\ell_2^{\dim E})\lesssim \frac{s}{t}.
\]

We next recall the relation between type and cotype. Let $1<p\leq2$ and let $q=p/(p-1)$ be its conjugate exponent. A classical duality result (see \cite{milman1986asymptotic}) gives the following.
\begin{fact}
If $\bana$ has type $p$ with upper constant $s$, then its dual has cotype $q$ with lower constant at least $c/s$, where $c>0$ is universal.
\end{fact}
Unlike the duality between uniform convexity and uniform smoothness, the reverse implication from cotype to type requires an additional geometric parameter. The work of \cite{pisier2006duality} relates this loss to the $K$-convexity constant. We use the following informal formulation; a formal definition is given in \Cref{sec:Discussion}.
\begin{definition}[$K$-convexity (informal)]\label{Def:KconvInf}
The constant $K_{\|\cdot\|}$ measures the loss in reverse type--cotype duality: if $\bana$ has cotype $q$ with lower constant $t>0$, then $\banad$ has type $p$ with upper constant at most
$
    C(p)\frac{K_{\|\cdot\|}}{t}.
$

\end{definition}
We record several useful facts about $K_{\|\cdot\|}$:
\begin{fact}
$K_{\|\cdot\|}=K_{\|\cdot\|_*}$, where $\|\cdot\|_*$ is the dual norm.
\end{fact}
\begin{fact}
If $\bana$ is $UC(2)$ with constant $t>0$, then $K_{\|\cdot\|}\lesssim t^{-1/2}$.
\end{fact}
A Banach space has a finite $K$-convexity constant if and only if it has nontrivial type; see, for example, \cite{milman2007remark}. A quantitative form of this equivalence is the following.
\begin{fact}
For every $\kappa<\infty$, there exist $p(\kappa)>1$ and $s(\kappa)<\infty$ such that $K_{\|\cdot\|}\leq\kappa$ implies that $\bana$ has type $p(\kappa)$ with upper constant $s(\kappa)$.
\end{fact}

Consequently, $2$-uniform convexity with constant $t>0$ implies nontrivial type, with parameters depending only on $t$. We conclude with two standard quantitative bounds:
\begin{enumerate}
    \item For every norm $\|\cdot\|$ on $\R^d$,
    \[
        K_{\|\cdot\|}
        \lesssim
        \log\left(1+d_{BM}(\|\cdot\|,\ell_2^d)\right)
        \lesssim
        \log(d+1).
    \]
    \item If $\|\cdot\|$ has type $p>1$ with upper constant $s$, then
    \[
         K_{\|\cdot\|}\lesssim C(p,s).
    \]
    In particular, if $\|\cdot\|$ is $UC(2)$ with constant $t$, then $K_{\|\cdot\|}\lesssim t^{-1/2}$.
\end{enumerate}
We conclude this part with the following remark.
\begin{remark}\label{R:Type}
Cotype is usually defined using Rademacher random variables, but an equivalent definition uses Gaussian random variables. For every $q\in[2,\infty)$, the two formulations are equivalent up to constants depending only on $q$; see \citep[Theorem~5.4.1]{artstein2022asymptotic}.
\end{remark}
\subsubsection{Examples}
\paragraph{$\ell_p$ norms in $\R^d$.}
The following standard facts about $\ell_p$-spaces can be found, for example, in \citep{ball1994sharp,pisier2016martingales,artstein2022asymptotic}:
\begin{itemize}
    \item When $1<p\leq2$, $\ell_p^d$ is $UC(2)$ with constant $t=p-1$ and $US(p)$ with constant $s=1$. It also has $T(p)$ and $CO(2)$ with absolute constants. At the endpoint $p=1$, $\ell_1^d$ is not uniformly convex, although it still has cotype $2$ with an absolute constant.
    \item When $2\leq p<\infty$, $\ell_p^d$ is $UC(p)$ with constant $t=1$ and $US(2)$ with constant $s=p-1$. It has cotype $p$ with an absolute constant and type $2$ with constant
    \[
        s\lesssim \sqrt{\min\{p,\log(d+1)\}}.
    \]
    \item If $1<p\leq2$ and $q=p/(p-1)$, then
    \[
        K_{\ell_p^d}=K_{\ell_q^d}
        \lesssim
        \min\left\{(p-1)^{-1/2},\sqrt{\log(d+1)}\right\}.
    \]
\end{itemize}
\paragraph{Sobolev spaces.}
Let
\[
N_k:=\#\{\alpha\in\mathbb N_0^d:\ |\alpha|\leq k\}=\binom{d+k}{k},
\]
and define
\[
T:W^{k,p}(\Omega)\to L_p(\Omega;\mathbb R^{N_k}),\qquad
Tf=(D^\alpha f)_{|\alpha|\leq k},
\]
where $\mathbb R^{N_k}$ is equipped with its $\ell_p^{N_k}$-norm. Then
\[
\|Tf\|_{L_p(\Omega;\ell_p^{N_k})}
=
\left(\sum_{|\alpha|\leq k}\|D^\alpha f\|_{L_p(\Omega)}^p\right)^{1/p}
=
\|f\|_{W^{k,p}(\Omega)},
\]
so $T$ is an isometric embedding of $W^{k,p}(\Omega)$ onto a closed linear subspace of $L_p(\Omega;\ell_p^{N_k})$.

Hence, for $1<p<\infty$, $W^{k,p}(\Omega)$ inherits the uniform convexity and uniform smoothness power-type estimates of the ambient $L_p$-space. In particular, it has type $\min\{p,2\}$ and cotype $\max\{p,2\}$, with constants inherited from the ambient space.
\paragraph{Schatten norms.}
For $1\leq p<\infty$, let $S_p$ be the class of compact operators $A$ on a Hilbert space such that
\[
    \sum_j s_j(A)^p<\infty,
\]
where
\[
    s_1(A)\geq s_2(A)\geq\cdots\geq0
\]
are the singular values of $A$, i.e., the eigenvalues of $|A|=(A^*A)^{1/2}$. The Schatten $p$-norm is
\[
    \|A\|_{S_p}=\left(\sum_j s_j(A)^p\right)^{1/p}.
\]
For $p=\infty$, one sets $\|A\|_{S_\infty}=s_1(A)$, the usual operator norm. Thus, the Schatten $p$-norm is the $\ell_p$-norm of the singular-value sequence:
\[
    \|A\|_{S_p}=\|s(A)\|_{\ell_p}.
\]
For $1<p\leq2$, the noncommutative Clarkson--Ball--Carlen--Lieb inequality gives
\[
\left\|\frac{A+B}{2}\right\|_{S_p}^2
+(p-1)\left\|\frac{A-B}{2}\right\|_{S_p}^2
\leq
\frac{\|A\|_{S_p}^2+\|B\|_{S_p}^2}{2}.
\]
Thus, $S_p$ is $UC(2)$ with constant $t=p-1$ for every $1<p\leq2$; see \cite{ball1994sharp}.

\subsection{Convex Concentration Property}\label{ss:ccp}
In overparameterized linear regression, our proof relies crucially on the convex concentration property (CCP); see \cite{adamczak2015note}. We say that a random vector $X\sim G$ in $\R^m$ satisfies the CCP with constant $L>0$ if, for every convex $1$-Lipschitz function $F:\R^m\to\R$ with $\E|F(X)|<\infty$,
\begin{equation}\label{Eq:CCP}
    \Pr_{X\sim G}\left(\left|F(X)-\E F(X)\right|\geq u\right)
    \leq 2\exp\left(-\frac{u^2}{L^2}\right),
    \qquad u>0.
\end{equation}
The CCP implies sub-Gaussian concentration for all linear functionals, but the converse is false. For example, a random vector with independent entries bounded in absolute value by $M$ satisfies the CCP with $L=O(M)$ by Talagrand's convex distance inequality; see \cite{boucheron2013concentration}. More importantly, an isotropic random vector with independent entries having a bounded sub-Gaussian norm satisfies the CCP with $L\lesssim\sqrt{\log(m)}$; see \citep[Lemma~1.4]{klochkov2020uniform}. Finally, the CCP is weaker than dimension-free concentration for all Euclidean Lipschitz functionals; product Rademacher measures provide a standard example of this distinction; see \cite{ledoux2001concentration}.

\section{Main Results}
  \label{sec:Results}
  For simplicity, in the first two theorems we assume that $\|\ft\|_{\PP}=\|\ft\|=1$. In the first theorem, the expectation is taken on the intersection of the events appearing in its assumptions. We also implicitly restrict to the nontrivial regime $s=o(M_n(\cF)^p)$ when $p>1$; otherwise, our bounds are of order one. The implicit constants in the upper bounds below may depend on the constants in the assumptions.
  
  The following theorem is our ``unlocalized'' upper bound (the basic inequality, in the terminology of \cite{van2000empirical}) on the \textbf{structural error} for $UC(2)$ norms:

\begin{theorem}[Basic inequality]\label{T:UC2}
Let $\bana$ be $UC(2)$ with constant $t>0$. Under Assumptions
\ref{A:One}--\ref{A:SB}, there exists an event $\Omega$, depending only
on $\vec X$, with probability at least $1-Cn^{-2}$ such that, if
\[
    T_1^{\Omega}
    :=
    \E_{\vec X}\left[
        \left\|
            \E_{\vec\xi}[\erm\mid\vec X]-\ft
        \right\|_{\PP}^2
        \mathbbm 1_{\Omega}
    \right].
\]
Then, up to a nonessential logarithmic factor,
\begin{equation}\label{Eq:BiasUC2}
    T_1^{\Omega}
    \lesssim
    \frac{R_{MM^*}\cG_n(\cF)}{t}.
\end{equation}
If, in addition, $\bana$ is $US(p)$ for some $p\in(1,2]$ with constant
$s>0$, then, up to a nonessential logarithmic factor,
\[
    T_1^{\Omega}
    \lesssim
    \frac{sR_{MM^*}^{2-p}\cG_n(\cF)^p}{t}.
\]
\end{theorem}

Under $2$-uniform convexity and no additional smoothness assumption (formally $p=1$), \Cref{T:UC2} agrees with the ``unlocalized'' risk bound for constrained least squares on $\cF$; see \cite{van2000empirical,chatterjee2014new} and the references therein. Although it may seem that greater smoothness of $\cF$ reduces $T_1$, the opposite can hold, especially in linear models. In the examples below, greater smoothness can lead to larger Gaussian complexity; see \Cref{ss:tr} for details.

To go beyond the basic inequality, we need to use the structure of the covariates and the \textbf{position} of the norm. We show that, in linear models, sharp bounds can be obtained when the norm is in a ``good'' position and $X$ has i.i.d.\ symmetric sub-Gaussian entries, as in \Cref{Ex:P}. We identify $\cF$ with $K$, the unit ball of a norm on $\R^d$, and $\ft$ with $\lft\in\R^d$. We make the following strong assumption, motivated by the $\ell_p$ ball with a $1$-sparse vector $\lft$:
\begin{assumption}\label{a:sym}
Let $K\subset\R^d$ be the unit ball of the norm and let $H=(\lft)^\perp$. We assume that $K$ is
invariant under the reflection about $H$. Namely for every $a\in\R$ and every $w\in H$,
\[
    a\lft+w\in K
    \qquad\Longleftrightarrow\qquad
    a\lft-w\in K.
\]
In other words, every affine section $K\cap(a\lft+H)$ is centrally
symmetric around the point $a\lft$.
\end{assumption}
The following proof-of-concept theorem gives an optimistic bound on $T_1$, that is tight for \Cref{Ex:P}:

\begin{theorem}[Sharp localized bound]\label{T:UC2Loc}
Assume that $d\geq C_1n$ for a sufficiently large absolute constant
$C_1>0$. Let $(\R^d,\|\cdot\|_K)$ be $UC(2)$ with constant $t>0$.
Under Assumptions \ref{A:One}, \ref{A:NS}, and \ref{a:sym}, suppose first
that the covariates are isotropic Gaussian. If $K$ is in isotropic
position, then
\[
    T_1
    =
    \widetilde{O}\left(
        \frac{\cG_n(K)^2+n^{-1}}{t^2}
    \right).
\]
If, in addition, $(\R^d,\|\cdot\|_K)$ is $US(p)$ for some $p\in(1,2]$
with constant $s>0$, then
\[
    T_1
    =
    \widetilde{O}\left(
        \frac{s^2\cG_n(K)^{2p}+n^{-1}}{t^2}
    \right).
\]
If instead the covariates are isotropic with i.i.d.\ symmetric
entries with bounded sub-Gaussian norm, the same bounds
hold provided that $\lft$ is $1$-sparse.\footnote{The result can be
extended to $O(1)$-sparse vectors with more technical efforts.}
\end{theorem}
We note that similar bounds holds when $K$ is in \textbf{John's} position using the results of \cite{klartag2008volume}, see \Cref{R:John}.  And also in the proportional Gaussian regime $d\asymp n$, under Gaussian covariates, it suffices to
assume that $K$ is in $M$-position and polynomially comparable to the euclidean ball. See the proof for precise statements. Before proceeding, we state a lemma used in the proof of \Cref{T:UC2Loc} that may be of independent interest.
\begin{lemma}\label{Lem:BMstar}
Let $P=d^{-1/2}\vec X$ be a sub-Gaussian random projection, assume that $\sqrt{d}K$ is in isotropic position, and let $d\geq C_3 n$ for a sufficiently large absolute constant $C_3>0$. Then, with probability at least $1-\exp(-cn)$,
\[
    1\lesssim b^{-1}(PK)\leq M_s^*(PK)\lesssim \log(d)^3,
    \qquad
    1\lesssim b^{-1}(P(K^{\circ}))\leq M_s^*(P(K^{\circ}))\lesssim \log(d)^2.
\]
\end{lemma}
This implies that $R_{bM^*}=\Tilde{O}(1)$ for the MNI with respect to both $\|\cdot\|_K$ and $\|\cdot\|_{K^{\circ}}$. We briefly discuss the proof in \Cref{Rem:isotropiccool}.

Our next theorem gives a lower bound on the noise error term $T_2$, namely the expected conditional variance, for $CO(2)$ norms in terms of the $K$-convexity constant of $\bana$.
 
For any fixed vector $\vec v\in\R^n$ and $r\geq0$, define
 \[
   \Psi_{n}(\vec v,r):= \mathrm{Med} \left[\min_{\{ f \in r \cdot \cF: \vec f = \vec v\}}\|f\|_{\PP}\right],
\]
Thus, $\Psi_n(\vec v,r)$ is the \emph{median} (over $\vec X$) of the minimum $L_2(\PP)$ norm among functions in $r\cdot\cF$ that interpolate $\vec v$ on $\vec X$. We use the shorthand $\vec e_1=(1,0,\ldots,0)\in\R^n$.
We use the shorthand
\[
    \vec e_1=(1,0,\ldots,0)\in\R^n.
\]

\begin{theorem}\label{T:UC2Var}
Assume that $\bana$ is
$CO(2)$ with constant $t>0$ and finite $K$-convexity constant
$K_{\|\cdot\|}$. Under Assumptions~\ref{A:One}, \ref{A:NS}, and
\ref{A:DeviationOfSuperma},
\begin{equation}\label{Eq:UC2CV}
    T_2
    \gtrsim
    n\,
    \Psi_n\left(
        \vec e_1,
        \frac{C K_{\|\cdot\|}M_n(\cF)}
             {t\sqrt n}
    \right)^2,
\end{equation}
where $C>0$ is an absolute constant.
\end{theorem}
This result can be interpreted as a ``reverse'' Efron--Stein inequality for minimum norm interpolators: it lower-bounds the variance by the marginal contribution of each point $X_i$. Intuitively, the expected conditional variance in $L_2(\PP)$ is bounded below by the $L_2(\PP)$ norm needed to interpolate appropriately scaled $1$-spikes. In the proof, we assume for simplicity that the MNI is a unique solution. 

\subsubsection*{Remarks}
\begin{remark}
 We refer to the proof of \Cref{T:UC2}, and in particular to \Cref{rmk:log}, for a precise statement and discussion of the logarithmic factors appearing in its bound.
\end{remark}

\begin{remark}
We believe that, without additional regularity assumptions, the upper bound in \Cref{T:UC2} is sharp, since it uses little of the covariate structure. We refer to \cite{bartlett2005local,van2000empirical} for localized bounds for ERM.
\end{remark}

 \begin{remark}
We believe that $UC(2)$ is not the optimal notion for \Cref{T:UC2,T:UC2Loc}. We conjecture that, when $\|\cdot\|$ is in the correct position, $CO(2)$ is the appropriate curvature assumption.
\end{remark}
\begin{remark}\label{Rem:isotropiccool}
The proof of \Cref{Lem:BMstar} uses Dudley's integral, the Blaschke--Santal\'o inequality, the finite-volume-ratio version of Kashin's theorem from \cite{litvak2004random}, and the fact that an isotropic $K$ is, up to constants, also in $M$- and $\ell$-position, as discussed above. In John's position, \citep[Cor.~4.8]{klartag2008volume} yields
$
    R_{bM^*} \lesssim \log(d)/\sqrt{t}.
$
\end{remark}

\begin{remark}
Without \Cref{a:sym}, the $n^{-1}$ term is replaced by a $n^{-1/2}$ term. Under Gaussian covariates and sufficient regularity, however, one can obtain the lower rate $\sqrt{(E_2+T_2)/n}\ll1/\sqrt n$ whenever $E_2+T_2\to0$.
\end{remark}

\begin{remark}
The $UC(2)$ assumption is needed for sub-Gaussian covariates to control both $E_1$, the shrinkage factor, and $E_2$. Under \emph{Gaussian} covariates, however, one can control $E_1$ without cotype or uniform convexity when the norm is in isotropic position. We believe that such curvature assumptions are necessary to control $E_2$; this is proved in a later work of the first author \cite{Gaussian}.
\end{remark}
\begin{remark}\label{R:John}
When $K$ is in John's position, then \Cref{T:UC2Loc} holds with the following bounds:
\[
    T_1
    \lesssim
    \frac{\log(d)^8}{t^6}\cG_n(K)^2
    +
    \frac{\log(d)^{10}}{t^6n}.
\]
Under the same $US(p)$ assumption,
\[
    T_1
    \lesssim
    \frac{s^2\log(d)^{2p+6}}{t^{p+5}}\cG_n(K)^{2p}
    +
    \frac{\log(d)^{10}}{t^6n}.
\]
\end{remark}
\subsection*{On the $\ell_p$-MNI}
From now on, we consider $\ell_p$ regression in the setting of \Cref{Ex:P}, with $p\in[1+C/\log d,2]$, and write $q=p/(p-1)$. As mentioned above, we extend the results of \cite{donhauser2022fast} to i.i.d.\ symmetric sub-Gaussian covariates and a $1$-sparse ground truth $\lft$; the same argument extends to $O(1)$-sparse ground truths. Applying \Cref{T:UC2Loc} and following its proof gives the following corollary.
 \begin{corollary}\label{C:ExmapleOne}
Let $\|\lft\|_0=\|\lft\|_2=1$, let $p\in[1+C/\log d,2]$, and assume that $n\log(d)^C\leq d\leq n^{q/2}\log(d)^{-C}$. Then
\[
    T_1 =  \widetilde{O}\left(\frac{d^{2p-2}}{n^{p}}\right).
\]
 \end{corollary}
 \begin{frem}
When the covariates are isotropic Gaussian, this bound is tight:
$T_1=\Theta(d^{2p-2}/n^p)$. The corollary also extends to $O(1)$-sparse $\lft$. The upper dimension restriction ensures Assumption \ref{A:NS}.
 \end{frem}
 
Our next theorem controls the noise error term $T_2$, which constitutes the most technical part of this work. For $1<p\leq2$, let
\[
    \beta(p):=\frac{\exp(C_0/(p-1))}{\sqrt{p-1}},
\]
where $C_0>0$ is an absolute constant. 
\begin{theorem}\label{T:VarLinOpt}
Assume that $\|\lft\|_0=O(1)$, that $p\in[1+C/\log d,2]$ for a sufficiently large absolute constant $C>0$, that $q=p/(p-1)\leq n$, and that the observation noise is standard Gaussian.
\begin{itemize}
\item If $X\sim N(0,I_d)$ and $n\beta(p)\lesssim d$, then there exist absolute constants $1<C_1<C_2$ such that
\[
    C_1^{1/(p-1)}\frac{n}{d}
    \lesssim T_2 \lesssim
    C_2^{1/(p-1)}\frac{n}{d}.
\]
\item If $X$ is distributed as in \Cref{Ex:P}, and $d\geq C_4n$, then
\[
    T_2
    \lesssim
    \left[
        \left(\frac{C_4}{p-1}\right)^{1/(p-1)}
        +
        \frac{n}{d\log(d/n)}
        \left(C_4L^2\log(d)\right)^{1/(p-1)}
    \right] \cdot \frac{n}{d},
\]
where $C_4 \geq 0$ is a universal constant and $L$ is the CCP constant of $X$\footnote{In this work, in order to reduce the burden of notation, we assume that $X$ is s.g. with a constant of order one. In the general case, $C_4 \geq 0$ is a constant that only depends on the s.g. constant, and that one may always assume that  $L\lesssim\sqrt{\log(n)}$ (cf. \cite{klochkov2020uniform}).}.
\end{itemize}
\end{theorem}

The endpoint scales relevant for interpreting the two bounds are $p=1+\varepsilon_{\mathrm G}(d)$ in the Gaussian case and $p=1+\varepsilon_{\mathrm{sg}}(d)$ in the sub-Gaussian case, where
\[
    \varepsilon_{\mathrm G}(d)
    \asymp\frac{1}{\log\log d},
    \qquad
    \varepsilon_{\mathrm{sg}}(d)
    \asymp\frac{\log\log\log d}{\log\log d}.
\]
When $d\asymp n^2$ and the displayed factors are polylogarithmic, the MNI may attain a parametric rate up to logarithmic factors.

An interesting consequence of the proof is related to \citep[Thm.~1.1 and \S 3.3.2]{Paorlp}, which shows that, for $q\asymp(p-1)^{-1}$ and $p\geq1+C/\log d$,
\[
    \widetilde{\Var}_{\vec \xi\sim N(0,I_d)}\left(\|\vec \xi\|_q\right)
    \asymp \frac{2^q}{q^2d},
\]
where $\widetilde{\Var}(Z):=\Var(Z)/(\E Z)^2$ denotes normalized variance. We prove the following.
\begin{corollary}\label{Cor:VarPaou}
Fix $\vec \xi\in\Sn$ and assume that the covariates are isotropic Gaussian. If $p\in(1+\frac{C}{\log d},2]$ and $n\log(d)^C\leq d\leq n^{q/2}\log(d)^{-C}$, then
\[
\widetilde{\Var}_{\vec X}\|\lerm(\vec X,\vec \xi)\|_{p} \lesssim \frac{C^{\frac{1}{p-1}}}{d}.
\]
\end{corollary}
This result should be compared with the bound obtained from Gaussian Lipschitz concentration, which is significantly larger as $p\to1$. The use of a fixed $\vec \xi$ follows from rotational invariance: $U\vec X\stackrel{d}{=}\vec X$ for every $U\in O(n)$, so the law depends on $\vec \xi$ only through $\|\vec \xi\|_2$.

\subsubsection*{Remarks}
Our proof yields several additional insights on the behavior of the $\ell_p$-MNI.
\begin{remark}
Our proof implies that it also holds that 
\[
    E_2 \lesssim \min\left\{\frac{C_2^{\frac{1}{p-1}} \cdot n}{d}, \frac{d^{2p-2}}{n^{p}} \right\}
\]
for standard Gaussian observation noise.\footnote{The proof of \Cref{T:VarLinOpt} that appeared in the preliminary work relied on the small-ball assumption and did not exploit the CCP for i.i.d.\ sub-Gaussian covariates. The new proof is significantly simpler and removes the small-ball assumption.}
\end{remark}
\begin{remark}
Our proof controls the magnitude of all entries of the $\ell_p$-MNI. \textbf{For example}, the $\ell_{\infty}$ norm of the MNI when the ground truth is zero is very close to the MSE. Specifically, 
\[
\E \|\lerm\|_{\infty} \lesssim \frac{\log(d)^{C/\varepsilon_{G}} \cdot \E\|\lerm\|_{2}}{\sqrt{d}} 
\]
Our proof implies that to upper bound the lowest  
    $[d,C n/\log(d/n)]$ 
entries the \textbf{sub-Gaussianity constant suffices}. However, the \textbf{CCP constant is needed} to  upper bound the top
$
    [cn/\log(d/n),1]
$
entries.
\end{remark}

\section{Discussion}
\label{sec:Discussion}

\subsection{Tradeoffs between type $p$ and Gaussian complexity}\label{ss:tr}
As briefly discussed above, \Cref{T:UC2} may suggest that the upper bound improves under additional smoothness, i.e., when the norm is $US(p)$ for $p>1$ with constant $s^p \geq 0$, which also implies type $p$ with constant $s \geq 0$. We now provide more details on why this is not necessarily true. Here, we consider overparameterized linear models under isotropic s.g. covariates, for which $\|\cdot\|_{\PP} = \|\cdot\|_2$.

First, we argue that when the class contains $m$ ``well-separated'' signals in $L_2(\PP)$, there is a price to be paid: the Gaussian complexity of the class is at least of order $s^{-1}\tfrac{m^{1-1/p}}{\sqrt{n}}$.
Therefore, a larger $p$ does not automatically lead to an improved upper bound in \Cref{T:UC2}.

To see this, let $m$ be the maximal number of elements $f_1,\ldots,f_m \in \cF$ that are orthogonal and well-separated, namely
\[
    \forall 1 \leq i < j \leq m, \qquad \inner{f_i,f_j} = 0 \quad \text{and} \quad \|f_i - f_j\|_2 \gtrsim 1,
\]
where $\inner{\cdot,\cdot}$ denotes the inner product associated with the $\ell_2$ norm.
Then, if $(\R^d,\|\cdot\|)$ has type $p$ with constant $s \geq 0$, the Gaussian complexity of $\cF$ is lower-bounded by
\begin{equation}
\label{eq:gclower}
     \GE{\cF} \gtrsim s^{-1} \cdot \sqrt{\frac{\max\{m^{2-2/p},\log(m)\}}{n}},
\end{equation}
when $m = \Tilde{\Omega}(n)$ and $m = \Tilde{O}(n^{q/2})$. In particular, for type $2$ spaces, we have $\GE{\cF} \gtrsim \sqrt{m/n}$. In contrast, for a space that is both cotype $2$ and type $1$, such as $\ell_1^d$, we pay only a logarithmic price in the number of different signals, which is tight in Example \ref{Ex:P}.

To simplify the presentation, we now give the argument for type $2$ spaces; the lower bound \eqref{eq:gclower} for general type $p$ follows analogously. Indeed, in a $T(2)$ space with constant $s$, we have
\begin{align*}
    \left(\E \left\|\frac{\sum_{i=1}^{m}\eps_i f_i}{\sqrt{m}}\right\|\right)^2 &
    \asymp \frac{\E \|\sum_{i=1}^{m}\eps_i f_i \|^2}{m} 
     \lesssim s^2,
\end{align*}
where we used Kahane's inequality (cf. \cite{milman1986asymptotic}) and the fact that $\|f_i\| \leq 1$ for all $1 \leq i \leq m$.
Moreover, by the orthogonality of $\{f_1,\ldots,f_m\}$, we have
\[
   \E \left\|\frac{1}{\sqrt{m}} \sum_{i=1}^{m} \eps_i f_i\right\|_2 \asymp 1.
\]
Combining the last two displays implies that the unit ball $\cF$ contains a set of $2^{\Omega(m)}$ functions, obtained from the random sign combinations above, that are pairwise separated by $\Theta(s^{-1})$ in $\|\cdot\|_2$. Hence, under isotropic s.g. covariates, which also satisfy the small-ball assumption, each fixed pair $f,g \in \cF$ satisfies
\[
    \|f - g\|_{\PP_n} \geq c_4 \|f - g\|_2 \gtrsim s^{-1},
\]
with probability at least $1-\exp(-cn)$. By the union bound, we obtain $2^{\Omega(\min\{m,n\})}$ sign mixtures that are well-separated and belong to $\cF_n$ with high probability.

By Sudakov's minoration inequality (cf. \cite{wainwright2019high}), it immediately follows that
\[
    \GE{\cF}\gtrsim s^{-1} \cdot \sqrt{\frac{\min\{m,n\}}{n}}.
\]
Therefore, we conclude that additional smoothness may increase Gaussian complexity and reduce the possibility of benign overfitting.
\subsection{The roles of $UC(2)$ and $R_{MM^*}$}
\label{ss:UC2}
Here, we provide rough intuition for the approach that led to the results in \Cref{sec:Results}. Consider $\ft \in \cF$ satisfying the assumptions above, and define the section orthogonal to $\ft$ by
\begin{equation}
       \cF^* := \{f \in \cF: \inner{f,\ft}_{\PP} = 0\}.
\end{equation}
Let $\cF_n^*$ denote its projection under $\vec X$, and let $\|\cdot\|_{\cF_n^*}$ be the induced Minkowski norm on $\R^n$. Anderson's lemma (cf. \cite{gardner1995geometric}) implies that, for any fixed $\vec X$,
\[
    \E_{\vec \xi}\|\vec \xi + \delta \cdot \vec \ft\|_{\cF_n^*}^2 \geq \E_{\vec \xi}\|\vec \xi\|_{\cF_n^*}^2.
\]
Clearly, we would like the last inequality to be strict, since we want this random norm to be sensitive to the presence of an additional signal rather than pure noise. The optimistic case is linear regression with isotropic Gaussian covariates and $\|\lft\| \asymp \|\lft\|_2 \asymp 1$. In this case, for $\delta \ll 1$,
\[
    \E_{\vec \xi,\vec X}\|\vec \xi + \delta \cdot \vec \lft\|_{\cF_n^*}^2 \approx  \lp 1 + \delta^2 \rp \cdot \E\|\vec \xi\|_{\cF_n^*}^2,
\]
where we use that $\delta \vec \lft \sim N(0,\delta^2 I_n)$ is independent of the random norm $\|\cdot\|_{\cF_n^*}$, and that $\cF_n^* \sim U\cF_n^*$ for every fixed rotation matrix $U$. This simple observation suggests that the covariates provide curvature, since the displayed relation holds for any \emph{norm}. If the norm is induced by an inner product, then the same inequality holds for any covariate distribution.

To control the structural error $E_1$, we need the lower-bound direction of the display above. We aim to show that the fact that the sampled signal $\ft$ is uncorrelated with every function in $\cF^*$ implies that $\|\vec \xi + \delta \cdot \vec \ft\|_{\cF_n^*}$ behaves like $\|\vec \xi'\|_{\cF_n^*}$, where
\[
\vec \xi' \sim N\left(0,\left(1+\Theta(\delta^2)\right)I_n\right).
\]
Here, we use $\|f\|_{\PP} \asymp \|f\|_{\PP_n} \asymp 1$.
In other words, $\vec \ft$ ``behaves'' as additional independent noise.
Our solution for obtaining this sensitivity without relying heavily on the covariate distribution is to impose $2$-uniform convexity on the unit ball $\cF$. The main idea of \Cref{T:UC2} is to show that, with high probability,
\[
    \E_{\vec \xi}\|\vec \xi + \delta \cdot \vec \ft\|_{\cF_n^*}^2 \geq \lp1+ct\delta^2 R_{MM^*}^{-2}\rp\E\|\vec \xi\|_{\cF_n^*}^2,
\]
where the factor $R_{MM^*}$ arises from the symmetrization and contraction lemmas \cite{koltchinskii2011oracle}. Therefore, if $R_{MM^*}$ is small, then $2$-uniform convexity gives us sensitivity to an additional input. We \emph{conjecture} that the same holds for $CO(2)$ norms \textbf{in a suitable position}, rather than only for $UC(2)$ norms.
\subsection{Quantitative Anderson's Theorem}\label{s:QA}

A basic ingredient in our analysis is Anderson's inequality, which implies that for any norm $\|\cdot\|$ and any $x \in \R^d$,
\[
    \E \|\vec \xi + x\|^2 \geq \E \|\vec \xi\|^2 ,
\]
where $\vec \xi$ is an isotropic Gaussian vector. In this paper, we assume that the norm is $UC(2)$ with constant $t>0$, which yields the quantitative improvement
\[
    \E \|\vec \xi + x\|^2 \geq \E \|\vec \xi\|^2 + t\|x\|^2 ,
\]
valid for any symmetric noise. However, the assumption of $2$-uniform convexity appears to be stronger than what is intrinsically needed. Indeed, $UC(2)$ is a \emph{pointwise} curvature requirement on the norm, whereas the inequality above only probes curvature \emph{on average} under the noise distribution. Evidence for this comes from $\ell_p$ spaces. In \cite{Anderson}, we prove that for every $p\in[1,2]$ and $x\in\R^d$,
\[
    \E \|\vec \xi + x\|_p^2 
    \geq 
    \E \|\vec \xi\|_p^2 + c\|x\|_p^2 ,
\]
for an absolute constant $c>0$ independent of $d$ and $p$. Thus, a dimension-free quantitative Anderson inequality holds with an absolute constant, even though the $UC(2)$ constant of $\ell_p$ is of order $p-1$ and therefore degenerates as $p \downarrow 1$ (and $\ell_1$ is not uniformly convex).

In contrast, the situation depends strongly on the distribution of the noise. For example, if $\vec \eps \sim U(\{-1,1\}^d)$ is Rademacher noise and we work with the $\ell_1$ norm, then, for the diagonal direction $u=(1/d,\dots,1/d)$ with $\|u\|_1=1$ and $t \leq d$,
\[
\E\|\vec \eps + t u\|_1^2 
= d^2 + \frac{t^2}{d}
= \E\|\vec \eps\|_1^2 + \frac{1}{d}\|t u\|_1^2.
\]
Thus, sensitivity to an additional signal deteriorates with dimensionality. It is plausible that for Gaussian noise, a norm with cotype $2$ satisfies a dimension-free quantitative Anderson inequality once the body is placed in a suitable position. Note that the choice of position plays a crucial role. 

For example, consider the norm whose unit ball is
\[
K=[-1,1]\times B_2^{n-1} \subset \R^n.
\]
Although $K$ is in \emph{John's position} and its norm has cotype $2$ with an absolute constant, one easily checks that
\[
\E\|e_1+\vec \xi\|^2
\le
\E\|\vec \xi\|^2 + C e^{-c n}.
\]
For comparison, rescaling the Euclidean factor gives the body
\[
\widetilde K = [-1,1] \times \sqrt{n}\, B_2^{n-1},
\]
for which the two factors contribute on the same Gaussian scale. The \emph{isotropic position} is likewise insufficient. In dimension $2n$, consider
\[
    K=[-\sqrt 3,\sqrt 3]^n \times \sqrt{n+2}\,B_2^n .
\]
This product is isotropic, since both factors are isotropic. Nevertheless, for the vector
\[
    x=(0,\sqrt{n+2}\,e_1), \qquad \|x\|_K=1,
\]
in the direction of the Euclidean factor, one can check that
\[
    \E\|\vec \xi + x\|_K^2
    \le
    \E\|\vec \xi\|_K^2 + C e^{-cn}\|x\|_K^2 .
\]
The underlying phenomenon is a scaling discrepancy from the Gaussian perspective: the $\ell_\infty$ factor typically contributes a quantity of order
$\sqrt{\log n}$ to the product norm, whereas the Euclidean factor contributes
only a quantity of order one. Thus, a unit displacement inside the Euclidean
factor is almost always hidden by the $\ell_\infty$ factor, except on an exponentially small event.

We refer the reader to the forthcoming manuscripts \cite*{Anderson,AndersonTwo}, which investigate such quantitative Anderson inequalities. In particular, these works suggest considering the positions obtained from
\[
\argmin_{T \in \mathrm{SL}_n(\mathbb R)} \EE\|G\|_{TK}^2 \text{ and }
\argmin_{T \in \mathrm{SL}_n(\mathbb R)} \EE\|G\|_{TK}.
\]
These positions are close to the min-norm position.
\subsection{Limitations in linear models}
The theorems in this work do not fully exploit Gaussian noise to obtain sharper bounds, and Theorems \ref{T:UC2Loc} and \ref{T:VarLinOpt} rely heavily on high-probability events in the covariates. This leads us to use $2$-uniform convexity rather than average notions of curvature, such as cotype $2$. Theorem \ref{T:VarLinOpt} uses specific properties of the $\ell_p$ norm; however, based on \cite{klartag2008volume}, we expect that it should hold for every $UC(2)$ norm in a suitable position. The main bottleneck to extending Theorem \ref{T:VarLinOpt} to other positions is that sharp $\eps$-covering-number bounds for the unit balls of $UC(2)$ norms, in John's position or any other regular position, are not known. The best available bound\footnote{Personal communication with G. Paouris} appears in \cite{milman2015mean} for isotropic position, but it does not yield the desired bound of $\widetilde{O}(n/d)$ under a sufficiently large $UC(2)$ constant. We believe that none of the classical positions would yield optimal bounds and that finding a new position is required for optimal performance of the MNI under $UC(2)$ norms.
\subsection{The MNI has no renorming property under large noise}\label{ss:renorming}
Here, we show that when $\sigma \asymp 1$, the MSE of the MNI does not satisfy a renorming theorem. In this setting, one can find two norms that are equivalent up to an absolute constant, yet their MSEs differ significantly. To support this claim, \cite{wang2022tight} first showed that, for the $\ell_1$-MNI with $\lft \equiv 0$,
\[
    \E \|\widehat{w}_1\|_2^2 = T_2  \asymp \log(d/n)^{-1}.
\]
For simplicity, assume that $d \gtrsim n^{1+\gamma}$ and consider the $\ell_p$ norm in $\R^d$, where
$
    p = 1 + \frac{C(\gamma)}{\log(d)}.
$
This norm is equivalent to $\ell_1$ up to a constant that depends only on $\gamma$. The proof of Theorem \ref{T:VarLinOpt} implies that
\[
    T_2 \lesssim n^{-c_1}
\]
for some $c_1 = c_1(C(\gamma)) \geq 0$, and therefore the claim follows.

To introduce the second example, we first define the $2$-Firey sum, an operation that preserves $2$-uniform convexity; see \cite{klartag2008volume}.
\begin{fact}[$2$-Firey Sum.] Given two norms $\|\cdot\|_{\cF_1}$ and $\|\cdot\|_{\cF_2}$ that are $UC(2)$ with constant $t \geq 0$ on the same space, their $2$-Firey sum is the norm defined by
\[
    \|f\|_{\cF}:= \sqrt{\inf_{f = g + h}\{\|g\|_{\cF_1}^2 + \|h\|_{\cF_2}^2\}}
\]
and is also $UC(2)$ with constant $t \geq 0$.
\end{fact}
Consider \Cref{Ex:NP}. The Sobolev-norm MNI may be inconsistent as $n \to \infty$ \citep{karhadkar2026harmful}, since the $L_2$ norm of the ``bumps'' is of order one. Yet, using the $2$-Firey sum defined above, one can combine the Sobolev norm with a kernel-induced norm having ``small bumps'' and obtain a norm equivalent to $W^{k,p}(\Omega)$ up to an absolute constant, for which the MNI is a minimax-optimal estimator.
\subsection{Pisier's K-Convexity and the MNI}
\label{R:Kconvexity}
First, for the reader's convenience, we define the $K$-convexity constant with respect to the Gaussian measure (cf. \citep[Ch.~6]{artstein2015asymptotic}). For any map $F:\R^n \to \bana$, its $\ell$-norm is defined by
\begin{equation*}
    \ell(F) = \ell_{\|\cdot\|}(F):= \sqrt{\int \|F(\vec \xi)\|^2 \, d\gamma_n(\vec \xi)}.
\end{equation*}
Note that if $\|x\|^2=x^{\top}Ax$ for a positive-definite matrix $A$ and $F$ is the identity map, then $\ell(F)^2=\operatorname{tr}(A)$.
The linearization of $F$ is defined by
\[
    L_F(\vec \xi):= \sum_{i=1}^n \alpha_i \xi_i + \int_{\R^n} F(\vec z) \, d\gamma_n(\vec z),
\]
where
\[
    \alpha_i = \int_{\R^n} \frac{\xi_i}{2} \left(F(\vec \xi_+^{(i)}) - F(\vec \xi_-^{(i)})\right) d\gamma_n(\vec \xi) \text{ and } \vec \xi_{\pm}^{(i)} = (\xi_1,\ldots,\pm \xi_i,\ldots,\xi_n).
\]
In other words, we project $F$ onto the subspace spanned by Hermite polynomials of degrees zero and one. 

Consider the Banach space $
    \left(\{F:\R^n \to \cB(\cX)\}, \ell_{\|\cdot\|}\right)
$
and define the linearization operator by
\[
    \cL(F):=L_F.
\]
The (Gaussian) $K$-convexity constant is the operator norm of $\cL$, i.e.,
    $ K_{\|\cdot\|}:= \sup_{\ell(F) \leq 1} \ell(\cL(F))$.
Equivalently,
\begin{equation}\label{Eq:Kconv}
   \forall F:\R^n \to \bana : \  \ell(L_F) \leq K_{\|\cdot\|} \ell(F).
\end{equation}
\paragraph{Key motivation for this work:} 
Here, we provide an application of Pisier's $K$-convexity argument, which is not commonly used in statistics and learning theory. Fix a realization of $\vec X$ and linearize the MNI as a map from the noise in $\R^n$ to $\cB(\cX)$. Namely, consider the operator
\[
\widehat{L}_n:=L(\erm \mid \vec X):\R^n \to \cB(\cX).
\]
It is not hard to verify that $\widehat{L}_n$ produces an interpolator for every $\vec \xi \in \R^n$. Hence, by definition,
\[
   \|\erm(\vec X,\vec \xi)\| \leq  \|\widehat{L}_n(\vec \xi)\|,
\]
and, in particular,
\[
    \ell(\erm) \leq \ell(\widehat{L}_n).
\]
When $\bana$ is an RKHS, $\widehat{L}_n = \erm$ by the \textbf{representer theorem}. What can we say when the norm is not induced by an inner product? Here, the $K$-convexity constant $K_{\|\cdot\|}$ enters, and
\[
    K_{\|\cdot\|}^{-1} \ell(\widehat{L}_n) \leq \ell(\erm) \leq \ell(\widehat{L}_n).
\]
In finite-dimensional linear models, $K_{\|\cdot\|}$ is at most $O(\log d)$. Under nontrivial type, i.e., type $1+\delta$ with constants $s,\delta>0$ independent of the dimension and $n$, it is at most $C(s,\delta)$.

Namely, \textbf{in terms of the $\ell$-norm}, the MNI is almost equivalent to an ellipsoid, somewhat resembling the representer theorem on average. More interestingly, the $K$-convexity argument shows that this ellipsoid is obtained by averaging the local behavior of the MNI; in high dimensions, one expects that
\[
    \alpha_{i} \approx \int_{\sqrt{n} \cdot \Sn}\frac{|\xi_i|}{2} \cdot \left(\erm(\vec \xi_{+}^{(i)}) - \erm(\vec \xi_{-}^{(i)})\right)d\sigma,
\]  
Remarkably, \textbf{on average}, this linearization captures the global behavior of the MNI $\erm$. For further intuition, consider the linear model case and let
\[
    E:=\Span\{\alpha_i\}_{i=1}^n \qquad \text{and} \qquad \cE:=\Span\{\vec X_i\}_{i=1}^n.
\]
The $K$-convexity argument implies that
\[
   M_{s}(\cP_{\cE}(K)) \leq  M_{s}(\cP_{\cE}(K \cap E)) \lesssim \log(d) M_{s}(\cP_{\cE}(K)),
\]
where $\cP$ is the linear projection onto the subspace $\cE$. This means that there exists a section $K \cap E$ that, on average, is as complex as the entire body, and this section can be found in a computationally efficient manner, provided that $\erm$ can be computed efficiently.
\subsubsection{Example of $\ell_p$-linear regression}
Let us apply Theorem \ref{T:UC2Var} to Example \ref{Ex:P} and, for simplicity, take $p \in [1,2]$.
In this case, the $K$-convexity constant of $\ell_p$ equals
\[
 K_p \asymp \min\{(p-1)^{-1/2},\sqrt{\log(d)}\},
\]
and
\[
    M_n(B_p^d) \asymp \sqrt{n} \cdot \max\left\{\sqrt{p-1}, \log(d/n)^{-1/2}
\right\} \cdot d^{1/p-1}.
\]
It is not hard to verify that, for isotropic Gaussian (or, more generally, sub-Gaussian) covariates, any interpolator of $\vec e_1$ satisfies
\[
\|w\|_{2}^2 \gtrsim 1/d,
\]
which is consistent with a $\Theta(d)$-dense vector, i.e., one whose entries are mostly balanced. The proof of Theorem \ref{T:VarLinOpt} implies that
$
    \lerm
$
is approximately $d \cdot \exp(O(-1/(p-1)))$-sparse, but its first-degree Hermite coefficients $\alpha_i$ may be $\Theta(d)$-dense, especially when $d \gg n^2$. The $K$-convexity factor in the theorem's bound, i.e.,
\[
\frac{C \cdot  K_{\|\cdot\|} \cdot M_n(\cF)}{\sqrt{n}}
\]
is essential for controlling these $\alpha_i$, since, as the proof below shows, they have $\ell_p$ norm of order $K_{\|\cdot\|} \cdot M_n(B_p^d)$.
 Therefore, for any $p \in [1,2]$, we have that 
    \[
        \E\mathrm{Var}\left[\widehat{w}_p \mid \vec X\right] \gtrsim n/d.
    \]
\section{Open Problems}\label{S:OP}
Here, we present a few open problems arising from our work:
\begin{itemize}
    \item Motivated by the possibility of relaxing the $UC(2)$ assumption in \Cref{T:UC2} and \Cref{T:UC2Loc} to cotype $2$, we formulate the following problem. Given a norm $\|\cdot\|$, let $c_{\|\cdot\|}$ be the largest constant for which the following quantitative version of Anderson's inequality holds:
\[
\E \|\vec \xi + x\|^2
\ge
\E \|\vec \xi\|^2 + c_{\|\cdot\|}\|x\|^2
\]
for every $x \in \R^d$. Is it true that, for every cotype $2$ norm with constant $t \geq 0$, there exists a position such that $c_{\|\cdot\|} \gtrsim c(t)$?
    \item Consider overparameterized linear regression under isotropic Gaussian covariates. Assume that the norm defining the MNI is $UC(2)$ with constant $t \geq 0$, but that we may choose its position, i.e., apply a linear transformation $T$ to its unit ball and consider the corresponding MNI.

Does there exist a linear transformation $T$ such that the MNI whose unit ball is $TK$ satisfies
\[
    T_2 \lesssim C^{1/t} \cdot (n/d).
\]
Equivalently, is the $\ell_p$ case the worst-case scenario, in the sense that the worst $UC(2)$ norm with constant $t \geq 0$ is $\ell_{t+1}$? For $\ell_p$, as mentioned above, all the ``nice'' positions are identical. However, as shown in \cite{klartag2008volume}, positions can be far apart for arbitrary $UC(2)$ norms.
    \item The proof of Theorem \ref{T:VarLinOpt} applies only to the $\ell_p$-MNI and relies heavily on $\ell_p$ geometry. Can one find a proof that applies, for example, to unconditional or one-symmetric norms?
\end{itemize}

\section{Proof Ideas}\label{S:PI}
Here, we provide proof sketches to the important theorems in this work.
\subsection{Proof Ideas for \Cref{T:UC2} and \Cref{T:UC2Loc}}
\paragraph{Notation.}
For brevity, write $\erm(\vec z):=\erm(\vec X,\vec z)$ for
$\vec z\in\R^n$. For a convex body $K$, $\|\cdot\|_K$ denotes the norm
whose unit ball is $K$. Recall the definition of $\cF_n$ in
\eqref{Eq:RndPrj} and Fact~\ref{F:NID}.
\subsubsection{The basic inequality}
Fix $\vec X$ and define the quotient norm
\[
    \|\vec z\|_n:=\|\erm(\vec z)\|,
    \qquad
    m_{\vec X}:=\E_{\vec\xi}\|\vec\xi\|_n.
\]
For $r\in[1,2]$, set
\[
    \Delta_r
    :=
    \E_{\vec\xi}\|\vec\xi+\vec\ft\|_n^r
    -
    \E_{\vec\xi}\|\vec\xi\|_n^r.
\]
Anderson's inequality gives $\Delta_r\geq0$. Without a smoothness
assumption, $\Delta_1\leq\|\vec\ft\|_n\leq1$. If the norm is $US(p)$
with constant $s$, the symmetrized smoothness inequality gives
$\Delta_p\lesssim s$. The moment-comparison lemma in the proof then yields,
up to the logarithmic factor in \Cref{T:UC2},
\[
    \Delta_2\lesssim m_{\vec X}
    \quad\text{or}\quad
    \Delta_2\lesssim s\,m_{\vec X}^{2-p},
\]
respectively. On the regular event, $m_{\vec X}\asymp M_n(\cF)$.

Let
\[
    w_+:=\erm(\vec\xi+\vec\ft),
    \qquad
    w_-:=\erm(-\vec\xi+\vec\ft).
\]
The vector $(w_+-w_-)/2$ interpolates $\vec\xi$. Therefore, minimality and
$2$-uniform convexity imply
\[
    \frac{\|w_+\|^2+\|w_-\|^2}{2}
    \geq
    \|\vec\xi\|_n^2
    +t\left\|\frac{w_++w_-}{2}\right\|^2.
\]
After averaging over the Gaussian noise and applying Jensen's inequality,
\[
    t\left\|
        \E_{\vec\xi}[\erm(\vec\xi+\vec\ft)\mid\vec X]
    \right\|^2
    \lesssim\Delta_2.
\]
Set
\[
    f_{\vec X}:=
    \E_{\vec\xi}[\erm(\vec\xi+\vec\ft)\mid\vec X].
\]
Since $f_{\vec X}$ interpolates $\vec\ft$, the function
$f_{\vec X}-\ft$ vanishes on the sample. The lower-isometry assumption
therefore converts the preceding Banach-norm estimate into an
$L_2(\PP)$ estimate. Using
\[
    M_n^*(\cF)=n\cG_n(\cF),
    \qquad
    R_{MM^*}=M_n(\cF)\cG_n(\cF),
\]
and averaging over $\vec X$ gives
\[
    T_1^{\Omega}
    \lesssim
    \frac{R_{MM^*}\cG_n(\cF)}{t},
\]
and, under $US(p)$,
\[
    T_1^{\Omega}
    \lesssim
    \frac{sR_{MM^*}^{2-p}\cG_n(\cF)^p}{t}.
\]

\subsubsection{Localization}
We now work in the linear model. Set
\[
    H:=(\lft)^\perp,
    \qquad
    \vec\eta:=\vec X\lft,
    \qquad
    \ermp:=\erm(\vec\eta+\vec\xi).
\]
For $a\in\R$ and $\vec z\in\R^n$, define
\[
    m_a(\vec z)
    :=
    \inf\left\{
        \|a\lft+u\|_K:
        u\in H,\ \vec Xu=\vec z
    \right\}.
\]
Every interpolator of $\vec\eta+\vec\xi$ can be written as
$(1-\delta)\lft+u$, where $u\in H$ and
$\vec Xu=\vec\xi+\delta\vec\eta$. Hence the MNI selects
\[
    \delta_*
    \in
    \operatorname*{argmin}_{\delta\in\R}
    m_{1-\delta}(\vec\xi+\delta\vec\eta).
\]
This one-dimensional variational formulation is the main idea of the
localized proof.

Let $\Lambda$ control the relevant $R_{bM^*}$ parameters of the projected
central section and its polar. The geometric part of the proof gives
$\Lambda\lesssim\log(d)^3$ in isotropic position and
$\Lambda\lesssim\log(d)/\sqrt t$ in John's position. If
\[
    \|\vec z\|_n
    :=
    \inf\{\|u\|_K:u\in H,\ \vec Xu=\vec z\},
    \qquad
    M:=\E_{\vec\xi}\|\vec\xi\|_n,
\]
then, on the same regular event,
\[
    M\cG_n(K)\lesssim\Lambda,
    \qquad
    M^{-1}\lesssim\Lambda\cG_n(K).
\]

\paragraph{Bounding the shrinkage term $E_1$.}
The quantitative Anderson estimate used in the proof has the form
\[
    m_a(\vec\xi+\mu\vec\eta)
    \geq
    \left(
        1+\frac{ct\mu^2}{\Lambda^2}
        -\frac{C|\mu|\Lambda\log(d)}{\sqrt n}
    \right)m_a(\vec\xi).
\]
Nearby sections satisfy
\[
    0\leq
    m_1(\vec\xi)-m_{1-\delta}(\vec\xi)
    \lesssim
    \rho\,\delta,
    \qquad
    \rho:=\min\{1,sM^{1-p}\},
\]
where $\rho=1$ without a smoothness assumption. Comparing these two
estimates at the minimizer $\delta_*$ gives
\[
    |\delta_*|
    \lesssim
    \delta_U
    :=
    \frac{C}{t}
    \left(
        \Lambda^2\frac{\rho}{M}
        +\frac{\Lambda^3\log(d)}{\sqrt n}
    \right).
\]
Since $\delta_*$ is the error in the signal direction,
$E_1\lesssim\delta_U^2$.

\paragraph{Bounding $E_2$ via $2$-uniform convexity.}
Let $\widetilde w_n(\vec\xi)\in H$ be the minimizer defining
$m_1(\vec\xi)$, and set $\overline w_n:=\lft+\widetilde w_n$. The vectors
$\ermp$ and $\overline w_n$ interpolate the same observations. The
variational estimate above bounds their norm gap, and uniform convexity
therefore gives
\[
    \|\ermp-\overline w_n\|_K
    \lesssim
    \delta_U M.
\]
Their difference lies in $\ker(\vec X)$, so the lower-isometry estimate
implies
\[
    \|\ermp-\overline w_n\|_2^2
    \lesssim
    \Lambda^2\delta_U^2.
\]
By the reflection symmetry in \Cref{a:sym},
$\widetilde w_n(-\vec\xi)=-\widetilde w_n(\vec\xi)$, and hence
$\E_{\vec\xi}\widetilde w_n=0$. It follows that
\[
    E_2\lesssim\Lambda^2\delta_U^2.
\]

Combining the two terms and using the bounds relating $M$, $\Lambda$, and
$\cG_n(K)$ gives
\[
    T_1
    \lesssim
    \frac{\Lambda^8}{t^2}
    \left(
        \cG_n(K)^2+
        \frac{\log(d)^2}{n}
    \right),
\]
and, under $US(p)$,
\[
    T_1
    \lesssim
    \frac{1}{t^2}
    \left(
        s^2\Lambda^{2p+6}\cG_n(K)^{2p}
        +
        \frac{\Lambda^8\log(d)^2}{n}
    \right).
\]
Substituting the position-dependent value of $\Lambda$ yields the bounds in
\Cref{T:UC2Loc}. For Gaussian covariates, rotational invariance removes the
sparsity assumption; in the proportional Gaussian regime, the same argument
applies in $M$-position. 
\begin{remark}
 The term
$E_1$ measures shrinkage in the signal direction and is governed by the
scalar parameter $\delta_*$. Under Gaussian covariates, its control can
exploit an internal rotational invariance of the random quotient norm.

To explain this, fix the noise vector
$\vec\xi\in\sqrt n\Sn$ and one of the localized sections
$B\subset H$ appearing in the argument. Let
$\widetilde{\vec X}$ denote the Gaussian block of the design acting on $H$,
and define
\[
    \|z\|_{\widetilde{n}}
    :=
    \|z\|_{\widetilde{\vec X}B}.
\]
For every $U\in O(n)$,
$U\widetilde{\vec X}\stackrel{d}{=}\widetilde{\vec X}$. Therefore, the map
\[
    z\longmapsto
    \E_{\widetilde{\vec X}}
    \|z\|_{\widetilde{n}}
\]
is rotationally invariant and positively homogeneous. It follows that for every fixed $z\in\R^n$ it holds
\[
    \E_{\widetilde{\vec X}}
    \|z\|_{\widetilde{n}}
    =
    \frac{\|z\|_2}{\sqrt n}
    \E_{\widetilde{\vec X}}
    \|\vec\xi\|_{\widetilde{n}}.
\]
Let $g_{\lft}:=\vec X\lft$. The vector $g_{\lft}$ is Gaussian and is
independent of $\widetilde{\vec X}$. Hence, conditionally on $g_{\lft}$,
\[
    \E_{\widetilde{\vec X}}
    \|\vec\xi+\mu g_{\lft}\|_{\widetilde{n}}
    =
    \frac{\|\vec\xi+\mu g_{\lft}\|_2}{\sqrt n}
    \E_{\widetilde{\vec X}}
    \|\vec\xi\|_{\widetilde{n}}.
\]
Thus, in expectation, Gaussianity supplies the desired bound, and a position should be used to ensure a high-probability bound. Therefore, $E_1$ does not rely on uniform convexity of $K$, when the covariates give an additional ``curvature ''.
\end{remark}

\begin{remark}
The role of $UC(2)$ is more structural for $E_2$. This term measures the
component of the conditional mean transverse to $\lft$, and the proof must
convert a small gap between the norms of two interpolators into a bound on
their distance. Uniform convexity gives precisely this stability estimate for
$\ermp$ and $\overline w_n$. Gaussianity improves the shift comparison used
for $E_1$, but it does not by itself control the average geometry of the
near-minimizing caps that governs $E_2$. It is plausible that a cotype-$2$
condition, or sufficiently strong symmetries of the norm, could replace
$UC(2)$ in this step; the present proof uses $UC(2)$ to obtain a direct,
uniform bound.
\end{remark}

\subsection{Proof Ideas for \Cref{T:VarLinOpt}}\label{S:ThmLinear}

\paragraph{Notation.}
We retain the notation $p$, $q$, and $\beta(p)$ from \Cref{T:VarLinOpt}, and set
\[
    P:=\frac{\vec X}{\sqrt d},
    \qquad
    \lambda_p:=d^{\frac1p-\frac12}.
\]
For $j\in[d]$, let $P_j$ denote the $j$-th column of $P$. For
$S\subset[d]$, let $P_S$ denote the submatrix of $P$ formed by the columns
indexed by $S$. We write
\[
    \widetilde B_p^d:=\lambda_p B_p^d,
    \qquad
    \|v\|_p:=\|v\|_{\widetilde B_p^d}
    =\lambda_p^{-1}\|v\|_{\ell_p^d},
\]
where the second notation is used only in this section and in the proof of
\Cref{T:VarLinOpt}. We also set
\[
    \eps_d:=p-1,
    \qquad
    \cC_d:=\left[-\frac1{\sqrt d},\frac1{\sqrt d}\right]^d.
\]
Finally, throughout this section and its proof, we abbreviate $M_s$ and
$M_s^*$ by $M$ and $M^*$, respectively, and write
$R(K):=\sup_{x\in K}\|x\|_2$.

\subsubsection*{Preliminaries}

We first recall the version of Dvoretzky's theorem used in the Gaussian
argument; see, for example, \cite{artstein2015asymptotic}.

\begin{lemma}\label{Lem:Dvo}
Let $K\subset\R^d$ be a symmetric convex body, and let $A$ be an $n\times d$
random matrix with i.i.d.\ entries distributed as $N(0,1/d)$. Then, with
probability at least $1-2\exp(-cn)$,
\[
\begin{aligned}
    \left(
        1-C\frac{R(K)}{M^*(K)}\sqrt{\frac nd}
    \right)M^*(K)B_2^n
    &\subset AK
    \subset
    \left(
        1+C\frac{R(K)}{M^*(K)}\sqrt{\frac nd}
    \right)M^*(K)B_2^n.
\end{aligned}
\]
\end{lemma}

In the proof, this result is applied to a trimmed auxiliary body, not to the
full ball $\widetilde B_p^d$. The condition $n\beta(p)\lesssim d$ is the
range in which the resulting lower Euclidean inclusion is nontrivial.

For sub-Gaussian projections, we use the following consequence of Kashin's
theorem; see \cite{litvak2005euclidean,szarek1990spaces}.

\begin{lemma}\label{Lem:Kashin}
Assume that the entries of the design matrix $\vec X$ are independent,
centered, variance-one, and sub-Gaussian. If $d\geq Cn$, then
\[
    cB_2^n\subset P\cC_d
\]
with probability at least $1-2\exp(-cd)$.
\end{lemma}

We also use the standard singular-value estimate for sub-Gaussian random
matrices; see, for example, \cite{aubrun2017alice}.

\begin{lemma}\label{Lem:SM}
Let $\vec X$ be an $n\times d$ random matrix with independent, centered,
variance-one, sub-Gaussian entries, and assume that $d\geq Cn$. Then
\[
    \Pr\left(
        c\sqrt d\leq\sigma_{\min}(\vec X)
        \leq\sigma_{\max}(\vec X)\leq C\sqrt d
    \right)
    \geq1-2\exp(-cd).
\]
\end{lemma}

\paragraph{Scaling reduction.}
Let $\vec\xi\sim N(0,I_n)$ denote the observation noise from the main model,
and set
\[
    \vec\zeta:=\frac{\vec\xi}{\sqrt n}\sim N\left(0,\frac1nI_n\right),
    \qquad
    v_*:=\sqrt{\frac dn}\,\lft,
    \qquad
    \widehat v:=\sqrt{\frac dn}\,\lerm.
\]
Then
\[
    P\widehat v=Pv_*+\vec\zeta,
    \qquad
    T_2
    =\frac nd\,
    \E\left\|
        \widehat v-\E_{\vec\zeta}[\widehat v\mid P]
    \right\|_2^2.
\]
In the zero-signal case, we write
\[
    \ermp:=\widehat v
    =\operatorname*{argmin}_{Pv=\vec\zeta}\|v\|_p.
\]
The map $\vec\zeta\mapsto\ermp$ is odd, and hence
\[
    T_2=\frac nd\,\E\|\ermp\|_2^2.
\]
We also write
\[
    \|z\|_n
    :=\inf\{\|v\|_p:Pv=z\}
    =\|z\|_{P\widetilde B_p^d}.
\]

We reduce the upper bound in \Cref{T:VarLinOpt} to the following zero-signal
statement.

\begin{theorem}\label{T:WRD}
Under the assumptions of \Cref{T:VarLinOpt}, suppose that $\lft=0$.
\begin{itemize}
\item For Gaussian covariates,
\[
    \E\|\ermp\|_2^2
    \lesssim C_2^{1/\eps_d}.
\]
\item Under the sub-Gaussian and CCP assumptions,
\[
    \E\|\ermp\|_2^2
    \lesssim
    \left(\frac{C}{\eps_d}\right)^{1/\eps_d}
    +
    \frac{n}{d\log(ed/n)}
    \left(CL^2\log(ed)\right)^{1/\eps_d}.
\]
\end{itemize}
\end{theorem}

The scaling identity gives the zero-signal upper bound in
\Cref{T:VarLinOpt}; Step~IV explains the extension to a fixed
$O(1)$-sparse ground truth, and note that we may always assume that $L \lesssim \sqrt{\log(n)}$ as discussed above.

\subsubsection*{Proof plan}

The main difficulty is that the $\ell_p$-MNI does not admit a closed-form
expression. We instead use its optimality and a dyadic decomposition.

\paragraph{Step I: Gaussian geometry and dyadic decomposition.}
For $z\in\R^d$, let $\Pi_k(z)$ retain the $k$ largest coordinates of $z$ in
absolute value. Set
\[
    r:=\|\ermp\|_p=\|\vec\zeta\|_n.
\]
For the dyadic indices $k$ up to $d$, including the terminal index $d$, write
\[
    \Pi_k(\ermp)-\Pi_{k/2}(\ermp)=\delta_kw_k,
    \qquad
    \|w_k\|_p=r.
\]
The shells have disjoint supports, exhaust $\ermp$, and satisfy
$\sum_k\delta_k^p=1$. We split them into
\[
\begin{aligned}
    \cR_1&:=\left\{k\text{ dyadic}:\ k\log(ed/k)\geq c_{\mathrm{ent}}n\right\},\\
    \cR_2&:=\left\{k\text{ dyadic}:\ k\log(ed/k)<c_{\mathrm{ent}}n\right\}.
\end{aligned}
\]
Here $c_{\mathrm{ent}}>0$ is a sufficiently small absolute constant.

Under $n\beta(p)\lesssim d$, Dvoretzky's theorem is applied to the trimmed
body constructed in the proof. Together with the mean-width estimate, this
gives, with high probability,
\[
    r=\|\vec\zeta\|_n\asymp\sqrt{\eps_d}
\]
and a lower Euclidean inclusion of radius comparable to
$\eps_d^{-1/2}$.

\paragraph{Step II: the bulk blocks.}
For $k\in\cR_1$, a sparse-net argument and Gaussian deviation estimates
control $\|Pw_k\|_2$. If $\delta_k$ were too large, the lower Euclidean
inclusion would provide another vector with the same image under $P$ and a
smaller normalized $\ell_p$ norm, contradicting the minimality of $\ermp$.
The resulting estimate is
\begin{equation}\label{Eq:weights}
    \|\delta_kw_k\|_2
    \lesssim
    r\left(\frac{k}{d}\right)^{1/2}
    \left(
        C\,\eps_d\log\left(\frac{ed}{k}\right)
    \right)^{1/(2\eps_d)}.
\end{equation}

\paragraph{Step III: the remaining largest coordinates.}
For $k\in\cR_2$, the expectation term in the sparse projection estimate
dominates. We delete the support of the relevant block and use the quotient
norm generated by the remaining columns. The leave-out Euclidean inclusion
makes this norm Lipschitz with respect to $\|\cdot\|_2$. Conditional
Gaussian concentration, followed by a union bound over supports and a
one-dimensional discretization, yields the same estimate as in
\eqref{Eq:weights}.

Since the shells have disjoint supports,
\[
\begin{aligned}
    \E\|\ermp\|_2^2
    &\lesssim
    r^2 \cdot \sum_{k\ \mathrm{dyadic}}
    \frac{k}{d}
    \left(
        C\,\eps_d\log\left(\frac{ed}{k}\right)
    \right)^{1/\eps_d}\\
    &\lesssim C_2^{1/\eps_d}.
\end{aligned}
\]
The final inequality follows by optimizing the dyadic summand. The
exceptional event is handled by comparison with the minimum-$\ell_2$
interpolator.

\paragraph{Step IV: a fixed sparse ground truth.}
Retain the scaled variables $v_*$ and $\widehat v$ introduced above, and let
$S:=\operatorname{supp}(\lft)$. The identity
\[
    \|a+u\|_p^p
    =\|a\|_p^p+\|u\|_p^p,
    \qquad
    \operatorname{supp}(a)\subset S,
    \quad
    \operatorname{supp}(u)\subset S^c,
\]
reduces the problem on $S^c$ to the zero-signal minimization with a
fixed-dimensional family of shifted noise vectors. Since $|S|=O(1)$, a fine
net and a union bound make the estimates from Steps~I--III uniform over
these shifts; uniform convexity transfers the estimates from the net to the
selected shift.

\paragraph{Step V: sub-Gaussian covariates.}
For sub-Gaussian covariates, Kashin's theorem replaces the trimmed
Dvoretzky inclusion. Since $\cC_d\subset\widetilde B_p^d$, with high
probability,
\[
    cB_2^n\subset P\cC_d\subset P\widetilde B_p^d.
\]
Together with the dual mean-width estimate, this gives
\[
    \sqrt{\eps_d}\lesssim\|\vec\zeta\|_n\lesssim1.
\]

For the bulk blocks, the same variational argument and the sub-Gaussian
deviation estimate give
\begin{equation}\label{Eq:weightsSG}
    \|\delta_kw_k\|_2
    \lesssim
    \left(\frac{k}{d}\right)^{1/2}
    \left(
        C\log\left(\frac{ed}{k}\right)
    \right)^{1/(2\eps_d)}.
\end{equation}
Therefore,
\[
    \sum_{k\in\cR_1}\|\delta_kw_k\|_2^2
    \lesssim
    \left(\frac{C}{\eps_d}\right)^{1/\eps_d}.
\]
This part uses only the sub-Gaussianity of the entries.

For the remaining largest coordinates, a leave-one-column quotient norm and
the CCP with constant $L$ give
\[
    \|\ermp\|_\infty^2
    \lesssim
    \frac1d
    \left(CL^2\log(ed)\right)^{1/\eps_d}.
\]
The shells in $\cR_2$ occupy at most $Cn/\log(ed/n)$ coordinates, and hence
their total squared $\ell_2$ contribution is at most
\[
    \frac{n}{d\log(ed/n)}
    \left(CL^2\log(ed)\right)^{1/\eps_d}.
\]
Combining the two ranges proves the sub-Gaussian part of \Cref{T:WRD}.

The different endpoint scales of $p$ come from the bulk contribution. Indeed,
\[
    \sup_{0<x\leq1}
    x\left(C\,\eps_d\log\frac ex\right)^{1/\eps_d}
    \lesssim C^{1/\eps_d},
\]
whereas
\[
    \sup_{0<x\leq1}
    x\left(C\log\frac ex\right)^{1/\eps_d}
    \lesssim
    \left(\frac{C}{\eps_d}\right)^{1/\eps_d}.
\]
Thus, keeping the bulk loss polynomial in $\log d$ requires
$\eps_d\gtrsim1/\log\log d$ in the Gaussian case and
\[
    \eps_d\gtrsim
    \frac{\log\log\log d}{\log\log d}
\]
in the sub-Gaussian case. The CCP contribution remains separate.

\paragraph{Remark: Gaussian lower bound.}
Let
\[
    \overline v:=\E_{\vec\zeta}[\widehat v\mid P],
    \qquad
    Z:=\widehat v-\overline v.
\]
Then $PZ=\vec\zeta$ and
$T_2=(n/d)\E\|Z\|_2^2$. The lower Euclidean inclusion, minimality, and
Jensen's inequality imply that
$\|Z\|_p\lesssim\sqrt{\eps_d}$ with high probability. If
$C_1>1$ is sufficiently close to $1$, the smaller truncation
\[
    C\sqrt{\eps_d}\,\widetilde B_p^d
    \cap
    C_1^{1/(2\eps_d)}B_2^d
\]
has sufficiently small mean width and Euclidean radius that the upper side
of Dvoretzky's theorem places its image inside $\frac12B_2^n$. Since
$\|\vec\zeta\|_2$ is bounded below with high probability, $Z$ cannot belong
to this truncation. Consequently,
\[
    T_2\gtrsim C_1^{1/\eps_d}\frac nd.
\]
In the sub-Gaussian case, the argument only gives the weaker lower bound
$T_2\gtrsim cn/d$.
\section{Proofs}\label{S:Pfs}
\subsection{Proof of Theorem \ref{T:UC2}}
Throughout this proof, $\vec\xi\sim N(0,I_n)$, and $M_n$ and $M_n^*$ are defined with respect to Gaussian noise. For simplicity, we assume that
$\|\ft\|=\|\ft\|_{\PP}=1$. For a fixed realization of $\vec X$, write
\[
    \erm(\vec z):=\erm(\vec X,\vec z),
    \qquad
    \|\vec z\|_n:=\|\vec z\|_{\cF_n}=\|\erm(\vec z)\|,
    \qquad
    m_{\vec X}:=\E_{\vec\xi}\|\vec\xi\|_n.
\]
We first record a moment-comparison lemma.

\begin{lemma}\label{lem:anderson_p_2}
Let $G$ be a standard Gaussian vector in $\mathbb R^n$, let $\|\cdot\|$ be a norm on $\mathbb R^n$, and let $x\in\mathbb R^n$ satisfy
\[
    \|x\|\leq \E\|G\|.
\]
For $p\in[1,2]$, define
\[
    \Delta_p(x):=\E\|x+G\|^p-\E\|G\|^p,
    \qquad
    m:=\E\|G\|.
\]
Then $\Delta_p(x)\geq0$, and, for every $T\geq4m$,
\[
    \Delta_2(x)
    \leq
    2T^{2-p}\Delta_p(x)
    +Cm^2\exp\!\left(-\frac{cT^2}{m^2}\right),
\]
where $c,C>0$ are universal constants.
\end{lemma}

\begin{proof}
Anderson's inequality gives
\[
    \PP(\|x+G\|>u)\geq \PP(\|G\|>u),
    \qquad u>0.
\]
Writing the moments as tail integrals, for every $T\geq0$,
\begin{align*}
\Delta_2(x)
&\leq
\int_0^T
2u\bigl(\PP(\|x+G\|>u)-\PP(\|G\|>u)\bigr)\,du
+\int_T^\infty2u\,\PP(\|x+G\|>u)\,du\\
&\leq
2T^{2-p}\Delta_p(x)
+\int_T^\infty2u\,\PP(\|x+G\|>u)\,du.
\end{align*}
Let $b:=\sup_{\theta\in\mathbb S^{n-1}}\|\theta\|$ be the Euclidean Lipschitz constant of the norm. By duality, $b\lesssim m$. Hence Gaussian concentration and $\|x\|\leq m$ give, for $u\geq4m$,
\begin{equation}\label{eq279}
    \PP(\|x+G\|>u)
    \leq
    \PP(\|G\|>u-m)
    \lesssim
    \exp\!\left(-\frac{cu^2}{m^2}\right).
\end{equation}
Integrating this estimate over $[T,\infty)$ proves the lemma.
\end{proof}

Let $C_0$ be the constant in \Cref{A:NS}. Since
\[
    \E_{\vec X}\Pr_{\vec\xi}
    \bigl(\|\erm(\vec X,\vec\xi)\|<C_0\bigr)
    \leq n^{-2},
\]
Fubini's theorem and Markov's inequality show that, outside an event of probability at most $2n^{-2}$ over $\vec X$,
\[
    \Pr_{\vec\xi}
    \bigl(\|\erm(\vec X,\vec\xi)\|\geq C_0\bigr)
    \geq\frac12,
\]
and therefore $m_{\vec X}\geq C_0/2$. Intersect this event with those in
\Cref{A:DeviationOfSuperma,A:SB}, and fix $\vec X$ in the resulting event
$\Omega$. Then
\[
    m_{\vec X}\asymp M_n(\cF),
\]
and the lower-isometry estimate holds uniformly on $\cF$. Moreover,
$\|\vec\ft\|_n\leq1$, since $\ft$ is a feasible interpolator of its own sample values.

Assume first that the norm is $US(p)$, for some $p\in(1,2]$, with constant $s$. By uniqueness, the map $\vec z\mapsto\erm(\vec z)$ is odd. Hence the $US(p)$ inequality gives
\begin{equation}\label{Eq:smoothness}
\begin{aligned}
\|\vec\xi\|_n^p+s\|\vec\ft\|_n^p
&=\|\erm(\vec\xi)\|^p+s\|\erm(\vec\ft)\|^p\\
&\geq
\frac{\|\erm(\vec\xi)+\erm(\vec\ft)\|^p}{2}
+\frac{\|\erm(-\vec\xi)+\erm(\vec\ft)\|^p}{2}\\
&\geq
\frac{\|\vec\xi+\vec\ft\|_n^p}{2}
+\frac{\|-\vec\xi+\vec\ft\|_n^p}{2}.
\end{aligned}
\end{equation}
After averaging over $\vec\xi$, this yields
\[
    \Delta_p
    :=\E_{\vec\xi}\|\vec\xi+\vec\ft\|_n^p
      -\E_{\vec\xi}\|\vec\xi\|_n^p
    \leq s.
\]
When no smoothness assumption is imposed, we set $p=s=1$; the same estimate follows from the triangle inequality. Since $\|\vec\ft\|_n\leq1\leq m_{\vec X}$, we may apply \Cref{lem:anderson_p_2} with
\[
    T=Cm_{\vec X}\sqrt{\log(e+m_{\vec X})},
\]
where $C$ is a sufficiently large universal constant. We obtain
\[
    \Delta_2
    :=\E_{\vec\xi}\|\vec\xi+\vec\ft\|_n^2
      -\E_{\vec\xi}\|\vec\xi\|_n^2
    \lesssim
    s\,m_{\vec X}^{2-p}
    \log(e+m_{\vec X})^{(2-p)/2}.
\]

Set
\[
    w_+:=\erm(\vec\xi+\vec\ft),
    \qquad
    w_-:=\erm(-\vec\xi+\vec\ft).
\]
The vector $(w_+-w_-)/2$ interpolates $\vec\xi$. Applying $UC(2)$ to $w_+$ and $-w_-$, and then using minimality, gives
\begin{equation}\label{Eq:uniformconvexity}
\begin{aligned}
\E_{\vec\xi}\|\vec\xi+\vec\ft\|_n^2
&=\E_{\vec\xi}\frac{\|w_+\|^2+\|w_-\|^2}{2}\\
&\geq
\E_{\vec\xi}\left\|\frac{w_+-w_-}{2}\right\|^2
+t\E_{\vec\xi}\left\|\frac{w_++w_-}{2}\right\|^2\\
&\geq
\E_{\vec\xi}\|\vec\xi\|_n^2
+t\E_{\vec\xi}\left\|\frac{w_++w_-}{2}\right\|^2.
\end{aligned}
\end{equation}
By symmetry and Jensen's inequality,
\[
    \Delta_2\geq t\|g_{\vec X}\|^2,
    \qquad
    g_{\vec X}:=\E_{\vec\xi}\erm(\vec\xi+\vec\ft).
\]
The function $g_{\vec X}$ interpolates $\vec\ft$. Set
$h_{\vec X}:=g_{\vec X}-\ft$. Since $s\geq1$, $t\leq1$, and
$m_{\vec X}$ is larger than a sufficiently large universal constant, the preceding bounds imply
\[
    \|h_{\vec X}\|^2
    \lesssim
    \frac{s}{t}\,m_{\vec X}^{2-p}
    \log(e+m_{\vec X})^{(2-p)/2}.
\]
Moreover, $\vec h_{\vec X}=\vec0$. Applying the lower-isometry estimate to
$h_{\vec X}/\|h_{\vec X}\|$ gives
\[
    \|h_{\vec X}\|_{\PP}^2
    \lesssim
    \cG_n(\cF)^2\|h_{\vec X}\|^2.
\]
Consequently,
\begin{align*}
\|g_{\vec X}-\ft\|_{\PP}^2
&\lesssim
\frac{s}{t}\,
M_n(\cF)^{2-p}\cG_n(\cF)^2
\log\!\bigl(e+M_n(\cF)\bigr)^{(2-p)/2}\\
&=
\frac{s}{t}\,
R_{MM^*}(\cF)^{2-p}\cG_n(\cF)^p
\log\!\bigl(e+M_n(\cF)\bigr)^{(2-p)/2},
\end{align*}
where we used $M_n^*(\cF)=n\cG_n(\cF)$ and
$R_{MM^*}(\cF)=M_n(\cF)\cG_n(\cF)$. Since
$\Pr_{\vec X}(\Omega)\geq1-Cn^{-2}$, this proves the theorem.

\begin{remark}\label{rmk:log}
The estimate in \eqref{eq279} follows from the universal bound
$b\lesssim m$ for the Euclidean Lipschitz constant $b$ of a Gaussian norm. Indeed, after scaling so that $m=1$, this is equivalent to
\begin{equation}\label{eq557}
    cB_n\subset K,
\end{equation}
where $K$ is the unit ball of the norm. The logarithmic factor above is needed only to make the tail term in \Cref{lem:anderson_p_2} negligible. An improvement of the form
\[
    b\lesssim \frac{m}{\sqrt{\log(e+m)}}
\]
for the quotient norms $\|\cdot\|_n$ would remove it.
\end{remark}

\subsection{Proof of Theorem \ref{T:UC2Loc}}

We prove the result first in isotropic position. The modifications for John's
position and for the proportional Gaussian $M$-position case are given at the
end. Throughout this proof, $\vec\xi\sim N(0,I_n)$. Recall that
$P:=d^{-1/2}\vec X$ and set
\[
    H=(\lft)^\perp,
    \qquad
    K_H=K\cap H,
    \qquad
    \vec\eta=\vec X\lft.
\]
We also define $K_H^\circ$ to be the polar of $K_H$ inside $H$. We write
$P_H=P|_H$, while $\Pi_H$ denotes the Euclidean orthogonal projection onto
$H$. In the non-Gaussian case, the $1$-sparsity of $\lft$ allows us, after a
permutation of the coordinates, to write $\lft=\theta e_1$, where
$|\theta|\asymp1$ by \Cref{A:One}. Hence $H=e_1^\perp$, and $P_H$ is
independent of $\vec\eta$. For Gaussian covariates, the same conclusion holds
for an arbitrary $\lft$ by rotational invariance.

We use the following standard consequence of Dudley's entropy bound and
Sudakov's minoration. For a centrally symmetric convex body $L\subset\R^m$,
recall that
\[
    \eps_k(L):=\inf\{\eps>0:\cN(L,\eps B_m)\leq 2^k\},
    \qquad k\in\mathbb N.
\]
\begin{lemma}\label{Lem:dudley}
For every centrally symmetric convex body $L\subset\R^m$,
\[
    \sup_{k\in\mathbb N}\eps_k(L)\sqrt{\frac{k}{m}}
    \lesssim M_s^*(L)
    \lesssim
    \inf_{r\in\mathbb N}
    \left\{
        \eps_{2^r}(L)
        +\frac1{\sqrt m}\sum_{j=0}^{r}2^{j/2}\eps_{2^j}(L)
    \right\}.
\]
In particular, the last expression is bounded, up to a factor
$C\log(em)$, by the first one.
\end{lemma}

We will also use the Blaschke--Santal\'o inequality
\[
    \Vol(L)\Vol(L^\circ)\leq \Vol(B_m)^2
\]
for centrally symmetric convex bodies $L\subset\R^m$.

We begin with the geometric estimate stated in \Cref{Lem:BMstar}. In
this proof only, for cosmetic reasons, $K$ is scaled so that
$\sqrt d\,K$ is isotropic. Note that the quantity $R_{bM^*}$ which will be used later is invariant under scaling.

\begin{proof}[Proof of \Cref{Lem:BMstar}]
Assume $\sqrt d\,K$ is isotropic. The isotropic mean-width estimates give
\begin{equation}\label{Eq:isowidth}
    M_s^*(K)\lesssim\log(d)^2,
    \qquad
    M_s(K)=M_s^*(K^\circ)\lesssim\log(d).
\end{equation}
Together with \Cref{Lem:dudley}, these estimates bound the corresponding
sums by $C\log(d)^3$ for $K$ and by $C\log(d)^2$ for $K^\circ$.
The chaining argument from Step~I in the proof of \Cref{T:VarLinOpt}, with
Bernstein's inequality at each entropy level, therefore gives, with
probability at least $1-\exp(-cn)$,
\begin{equation}\label{Eq:widthproj}
    M_s^*(PK)\lesssim\log(d)^3,
    \qquad
    M_s^*(P(K^\circ))\lesssim\log(d)^2.
\end{equation}

By the discussion preceding the theorem, the isotropic position of
$\sqrt d\,K$ is an $M$-position, meaning that
\[
    \Vol(K\cap B_d)\geq \exp(-Cd)\Vol(B_d).
\]
Since $B_d\subset(K\cap B_d)^\circ$, the Blaschke--Santal\'o inequality
implies
\[
    \Vol((K\cap B_d)^\circ)\leq \exp(Cd)\Vol(B_d).
\]
Thus $(K\cap B_d)^\circ$ has bounded volume ratio, and Kashin's theorem gives
\[
    cB_n\subset P(K\cap B_d)\subset PK.
\]
The same argument applies to $K^\circ$, since $M$-position is self-dual.
This proves the required estimates for $PK$ and $P(K^\circ)$. We now analyze the central section. Since $K_H\subset K$ and $K_H^\circ=\Pi_H(K^\circ)$,
we have
\[
    \eps_k(K_H)\leq\eps_k(K),
    \qquad
    \eps_k(K_H^\circ)\leq\eps_k(K^\circ).
\]
Therefore we again have
\begin{equation}\label{Eq:widthsection}
    M_s^*(P_H K_H)\lesssim\log(d)^3,
    \qquad
    M_s^*(P_H(K_H^\circ))\lesssim\log(d)^2.
\end{equation}
Let $B_H:=B_d\cap H$. By projection, we have
\[
    \cN(K_H,B_H)\leq\cN(K,B_d),
    \qquad
    \cN(K_H^\circ,B_H)\leq\cN(K^\circ,B_d),
    \qquad
    \cN(B_H,K_H^\circ)\leq\cN(B_d,K^\circ).
\]
Finally, using the symmetry assumption relative to $H$ \Cref{a:sym}, it is easily checked that
\[
    \cN(B_H,K_H)\leq\cN(B_d,K).
\]
Thus both $K_H$ and $K_H^\circ$ are in $M$-position. Repeating the
arguments given for $K$, we conclude that, with high probability, the convex
bodies $PK$, $P(K^\circ)$, $P_HK_H$, and $P_H(K_H^\circ)$ all contain a ball
$cB_n$, for some universal constant $c>0$, while their mean widths are bounded
above, up to a universal constant, by $\log(d)^3$. In other words, in those
four instances, the scale-invariant quantity $R_{bM^*}$ is bounded by
$C\log(d)^3$, for some universal constant $C>0$.
\end{proof}

\paragraph{John's position.}
In John's position, Corollary~4.8 of \cite{klartag2008volume} gives
\[
    R_{bM^*}(K)\lesssim t^{-1/2}.
\]
Repeating the preceding entropy, projection, and central-section argument gives,
with high probability,
\[
    R_{bM^*}(PK)
    +R_{bM^*}(P_HK_H)
    +R_{bM^*}(P_H(K_H^\circ))
    \lesssim \frac{\log(d)}{\sqrt t}.
\]

From now on, we return to our usual normalization
$\|\lft\|_K=1$. In what follows, $\Lambda\geq1$ denotes an upper
bound for the
$R_{bM^*}$ parameters. Thus $\Lambda\lesssim\log(d)^3$
in isotropic position, while
$\Lambda\lesssim\log(d)/\sqrt t$ in John's position. 

We now record a crude upper bound on the moments of the MNI that will be used on the bad event.

\begin{lemma}\label{lem183}
For every fixed $q\geq1$,
\begin{equation}\label{Eq:bad}
    \E_{\vec X,\vec\xi}\|\ermp\|_2^q\leq d^{C(q)},
\end{equation}
where $C(q)$ depends only on $q$ and on the sub-Gaussian constant.
\end{lemma}

\begin{proof}
Let $\widehat w_2$ be the Euclidean
minimum-norm interpolator. In each of the positions considered here, the norm induced by $K$ is polynomially comparable to the Euclidean norm. Thus
\[
    \|\ermp\|_2\leq d^C\|\widehat w_2\|_2.
\]
Now it is well known that $\|\widehat w_2\|_2$ has polynomial moments, by standard lower singular-value estimates for sub-Gaussian matrices.
\end{proof}

We first dispose of the range
\[
    n\lesssim t^{-1}\Lambda^4\log(d).
\]
This reduction concerns only isotropic and John's positions; the proportional
Gaussian $M$-position case is treated directly at the end. Let $\widetilde K$
be the homothetic copy of $K$ used in the corresponding geometric argument.
By Kashin's theorem and the standard radius bound in isotropic position,
and by John's inclusions together with the singular-value estimate in John's
position, outside an event of probability at most $\exp(-cd)$,
\[
    cB_n\subset P\widetilde K,
    \qquad
    \sup_{u\in\widetilde K}\|u\|_2\lesssim\sqrt d.
\]
Since the MNI is unchanged by homotheties, on this event
\[
    \|\ermp\|_2
    \leq
    \left(\sup_{u\in\widetilde K}\|u\|_2\right)
    \|\vec\eta+\vec\xi\|_{\vec X\widetilde K}
    \lesssim \|\vec\eta+\vec\xi\|_2.
\]
Jensen's inequality therefore gives $T_1\lesssim n$, while the exceptional
event is negligible by H\"older's inequality and \eqref{Eq:bad}. Since
\[
    n\lesssim\frac{\Lambda^8\log(d)^2}{t^2n}
\]
in the present range, the claimed bound follows in both cases. We may
therefore assume that
\begin{equation}\label{Eq:largen}
    n\gtrsim t^{-1}\Lambda^4\log(d),
\end{equation}
which we do in the rest of this section. The next lemma establishes a good event on which we will work.
\begin{lemma}\label{Lem:reg_event}
Under the assumptions of \Cref{T:UC2Loc}, for every $A>0$, there is an event
$\Omega_A$, depending only on $\vec X$, that has probability at least $1-d^{-A}$, and such that on this event the ball and mean-width estimates established above hold. In particular,
\begin{equation}\label{Eq:lanmbda}
    R_{bM^*}(PK)
    +R_{bM^*}(P_H K_H)
    +R_{bM^*}(P_H(K_H^\circ))
    \lesssim\Lambda.
\end{equation}
Moreover,
\begin{equation}\label{Eq:lower}
    \|v\|_2^2
    \leq
    C(A)\left(\frac1n\|\vec X v\|_2^2+\cG_n(K)^2\right)
    \qquad\text{for every }v\in K.
\end{equation}
\end{lemma}

\begin{proof}
The estimate \eqref{Eq:lanmbda} was proved above with probability $1-\exp(-cn)$. By adjusting the constant in \eqref{Eq:largen}, this probability is at least $1-d^{-A}$.

For \eqref{Eq:lower}, by
Paley-Zygmund,
\[
    \Pr\left(|\inner{X_1,v}|\geq c_0\|v\|_2\right)\geq c_1
    \qquad\text{for every }v\in\R^d,
\]
where $X_1$ is the first row of $\vec X$. Mendelson's small-ball method and the standard comparison between Rademacher
and Gaussian complexities then give, with probability at least $1-d^{-A}$, for all $v\in K$
\[
    \frac1n\|\vec X v\|_2^2
    \geq c\|v\|_2^2-C(A)\cG_n(K)^2
\]
which is equivalent to \eqref{Eq:lower}.
\end{proof}

We will also use the following conditional form of the basic inequality. On
$\Omega_A$, the proof of \Cref{T:UC2}, together with
\eqref{Eq:lanmbda} and \eqref{Eq:lower}, gives, up to a logarithmic factor,
\begin{equation}\label{Eq:conditionalBI}
    \left\|\E_{\vec\xi}\ermp-\lft\right\|_2^2
    \lesssim
    \cG_n(K)^2+\frac{s\Lambda^2\cG_n(K)^p}{t},
\end{equation}
with the convention $p=s=1$ when no smoothness assumption is imposed.

We now fix $\vec X\in\Omega_A$. For $\vec z\in\R^n$, set
\[
    \|\vec z\|_{H,n}
    =\inf\{\|u\|_K:u\in H,\ \vec X u=\vec z\}
    =\|\vec z\|_{\vec X K_H},
    \qquad
    M=\E_{\vec\xi}\|\vec\xi\|_{H,n}.
\]
The following estimates, which follow from \Cref{Lem:BMstar}, will be used repeatedly:
\begin{equation}\label{Eq:M206}
    M\cG_n(K)\lesssim\Lambda,
    \qquad
    M^{-1}\lesssim\Lambda\cG_n(K).
\end{equation}
We split $\Omega_A$ according to whether $M$ is larger than a sufficiently
large universal constant. On the part where it is not, \eqref{Eq:M206} gives
$\cG_n(K)\gtrsim\Lambda^{-1}$, and \eqref{Eq:conditionalBI} is dominated by
the final bound. We therefore work below on the remaining part of
$\Omega_A$, where $M$ is sufficiently large.

For $a\in\R$ and $\vec z\in\R^n$, define
\begin{equation}\label{Eq:maa}
    m_a(\vec z)
    =\inf\{\|a\lft+u\|_K:u\in H,\ \vec X u=\vec z\}.
\end{equation}
By \Cref{a:sym}, $a\mapsto m_a(\vec z)$ is even. The map
$(a,\vec z)\mapsto m_a(\vec z)$ is convex, and
\begin{equation}\label{Eq:maaLipschitz}
    |m_a(\vec z)-m_{a'}(\vec z')|
    \leq |a-a'|+\|\vec z-\vec z'\|_{H,n}.
\end{equation}

Since $m_0(\vec z)=\|\vec z\|_{H,n}$, the map
$\vec z\mapsto m_0(\vec z)$ is $b(\vec X K_H)$-Lipschitz. Moreover,
using $\vec X K_H=\sqrt d\,P_HK_H$,
$M_g(P_HK_H)\asymp\sqrt n\,M_s(P_HK_H)$, and
$M_s(P_HK_H)M_s^*(P_HK_H)\geq1$, we obtain
\[
    \frac{b(\vec X K_H)}{M}
    \lesssim
    \frac{b(P_HK_H)}
         {\sqrt n\,M_s(P_HK_H)}
    \leq
    \frac{R_{bM^*}(P_HK_H)}{\sqrt n}
    \leq
    \frac{\Lambda}{\sqrt n}.
\]
Gaussian concentration and \eqref{Eq:largen} therefore imply
that $m_0(\vec\xi)\asymp M$ outside a set of Gaussian measure at most
$d^{-A}$. Since $a\mapsto m_a(\vec\xi)$ is even and convex,
\eqref{Eq:maaLipschitz} gives
$m_0(\vec\xi)\leq m_a(\vec\xi)\leq m_0(\vec\xi)+2$ for
$|a|\leq2$. Since $M$ is sufficiently large, we conclude that, on the same
event,
\begin{equation}\label{Eq:maM}
    m_a(\vec\xi)\asymp M
    \qquad\text{uniformly for }|a|\leq2.
\end{equation}

We will need the following two lemmas, whose proofs are given below.

\begin{lemma}\label{lem:good2}
Conditionally on $\vec X|_H$, with probability at least
$1-d^{-A}$ over $(\vec\eta,\vec\xi)$, uniformly over
$(a,\mu)\in[0,2]\times[-1,1]$,
\begin{equation}\label{Eq:UC2bias}
    m_a(\vec\xi+\mu\vec\eta)
    \geq
    \left(
        1+\frac{ct\mu^2}{\Lambda^2}
        -\frac{C|\mu|\Lambda\log(d)}{\sqrt n}
    \right)m_a(\vec\xi).
\end{equation}
\end{lemma}

\begin{lemma}\label{Lem:Redu2}
For $0\leq\delta\leq1$,
\begin{align}
    0&\leq m_1(\vec\xi)-m_0(\vec\xi)\leq1,
    \label{eq:277}\\
    0&\leq m_1(\vec\xi)-m_{1-\delta}(\vec\xi)\leq\delta.
    \label{eq:279}
\end{align}
If $\|\cdot\|_K$ is $p$-uniformly smooth with constant $s$, then, on the
event \eqref{Eq:maM},
\begin{align}
    0&\leq m_1(\vec\xi)-m_0(\vec\xi)
       \lesssim sM^{1-p},
    \label{eq:287}\\
    0&\leq m_1(\vec\xi)-m_{1-\delta}(\vec\xi)
       \lesssim s\delta M^{1-p}.
    \label{eq:289}
\end{align}
For $-1\leq\delta<0$,
\[
    m_{1-\delta}(\vec\xi)\geq m_1(\vec\xi).
\]
\end{lemma}

After increasing the exponent in the preceding high-probability estimates
and reducing $\Omega_A$ if necessary, Fubini's theorem and Markov's
inequality allow us to assume that, for every $\vec X\in\Omega_A$ in the
remaining regime, there is an event $\mathcal E_{\vec X}$, depending only
on $\vec\xi$, such that
\[
    \Pr_{\vec\xi}(\mathcal E_{\vec X}^c)\leq d^{-A},
\]
and on which \eqref{Eq:maM} and the conclusions of
\Cref{lem:good2,Lem:Redu2} used below hold.

Set
\[
    \rho=\min\{1,sM^{1-p}\},
    \qquad
    \Delta=\frac1t
    \left(
        \Lambda^2\frac{\rho}{M}
        +\frac{\Lambda^3\log(d)}{\sqrt n}
    \right),
\]
with the convention $p=s=1$ when no smoothness assumption is imposed.
We may assume that $\Delta$ is smaller than a sufficiently small universal
constant. On the part of $\Omega_A$ where this fails,
\eqref{Eq:conditionalBI} is dominated by the final bound. Indeed, it suffices
to use \eqref{Eq:M206} and
\[
    a^2+\frac1n\geq\frac{2a}{\sqrt n},
\]
with $a=\cG_n(K)$, and with $a=s\cG_n(K)^p$ in the uniformly smooth
case.

Every interpolator of $\vec\eta+\vec\xi$ can be written uniquely as
\[
    w=(1-\delta)\lft+u,
    \qquad
    u\in H,
    \qquad
    \vec Xu=\vec\xi+\delta\vec\eta.
\]
Hence, if $\delta_*$ is the coefficient selected by the MNI,
\begin{equation}\label{eq:deltav}
    \delta_*\in\operatorname*{argmin}_{\delta\in\R}
    m_{1-\delta}(\vec\xi+\delta\vec\eta),
    \qquad
    \|\ermp\|_K
    =m_{1-\delta_*}(\vec\xi+\delta_*\vec\eta).
\end{equation}

The objective in \eqref{eq:deltav} is convex. On
$\mathcal E_{\vec X}$, its value at $\delta=1$ satisfies
\[
    m_0(\vec\xi+\vec\eta)
    \geq
    \left(
        1+\frac{ct}{\Lambda^2}
        -\frac{C\Lambda\log(d)}{\sqrt n}
    \right)
    \left(m_1(\vec\xi)-C\rho\right)
    \geq m_1(\vec\xi),
\]
where the last inequality follows from the assumption that $\Delta$ is small enough.
At $\delta=-1$, the same argument, together with the monotonicity in
\Cref{Lem:Redu2}, gives
\[
    m_2(\vec\xi-\vec\eta)\geq m_1(\vec\xi).
\]
It follows by convexity that $|\delta_*|\leq1$. Assume first that $\delta_*\geq0$. By minimality and
\Cref{lem:good2,Lem:Redu2},
\begin{align*}
    m_1(\vec\xi)
    &\geq m_{1-\delta_*}(\vec\xi+\delta_*\vec\eta)\\
    &\geq
    \left(
        1+\frac{ct\delta_*^2}{\Lambda^2}
        -\frac{C\delta_*\Lambda\log(d)}{\sqrt n}
    \right)
    \left(m_1(\vec\xi)-C\rho\delta_*\right).
\end{align*}
Since $m_1(\vec\xi)\asymp M$ on $\mathcal E_{\vec X}$, this yields
\[
    \delta_*
    \lesssim
    \frac1t
    \left(
        \Lambda^2\frac{\rho}{M}
        +\frac{\Lambda^3\log(d)}{\sqrt n}
    \right).
\]
If $\delta_*<0$, then
$m_{1-\delta_*}(\vec\xi)\geq m_1(\vec\xi)$, and the same argument gives
\[
    |\delta_*|
    \lesssim
    \frac{\Lambda^3\log(d)}{t\sqrt n}.
\]
Thus, in either case, we have
\begin{equation}\label{Eq:deltabound}
    |\delta_*|\lesssim\Delta.
\end{equation}

We now bound $E_1$. Since
$\cP_{\lft}(\ermp)=(1-\delta_*)\lft$ and $\|\lft\|_2\asymp1$, for
every fixed $\vec X$,
\[
    \left\|
        \cP_{\lft}\E_{\vec\xi}\ermp-\lft
    \right\|_2^2
    =
    \|\lft\|_2^2
    \left|\E_{\vec\xi}\delta_*\right|^2
    \asymp
    \left|\E_{\vec\xi}\delta_*\right|^2.
\]
Using
\eqref{Eq:deltabound}, we obtain
\[
    \left|\E_{\vec\xi}\delta_*\right|^2
    \lesssim
    \Delta^2
    +
    \left|
        \E_{\vec\xi}
        \left[
            \delta_*\mathbbm 1_{\mathcal E_{\vec X}^c}
        \right]
    \right|^2.
\]
After averaging over $\vec X$, the second term, as well as the contribution
of $\Omega_A^c$, is negligible by conditional Jensen's inequality,
H\"older's inequality, \eqref{Eq:bad}, and
\[
    |\delta_*|\lesssim1+\|\ermp\|_2.
\]
We defer this standard truncation to the end of the proof. Consequently,
using \eqref{Eq:M206}, we obtain, without a smoothness
assumption,
\begin{equation}\label{eq:e1final}
    E_1
    \lesssim
    \frac1{t^2}
    \left(
        \Lambda^6\cG_n(K)^2
        +\frac{\Lambda^6\log(d)^2}{n}
    \right).
\end{equation}
If the norm is $p$-uniformly smooth, then
\begin{equation}\label{eq:e1finalSmooth}
    E_1
    \lesssim
    \frac1{t^2}
    \left(
        \Lambda^8s^2\cG_n(K)^{2p}
        +\frac{\Lambda^6\log(d)^2}{n}
    \right).
\end{equation}

\paragraph{Upper bound on $E_2$.}
Let $\widetilde w_n(\vec\xi)\in H$ be the minimizer in the definition of
$m_1(\vec\xi)$, and set
\[
    \overline w_n:=\lft+\widetilde w_n.
\]
Both $\ermp$ and $\overline w_n$ interpolate
$\vec\eta+\vec\xi$. On $\mathcal E_{\vec X}$, comparison of the objective
in \eqref{eq:deltav} at $\delta=\delta_*$ and $\delta=0$,
together with \eqref{Eq:deltabound}, gives
\begin{equation}\label{Eq:normgap}
    0\leq
    \|\overline w_n\|_K-\|\ermp\|_K
    \lesssim t\Delta^2M.
\end{equation}
Moreover,
\[
    \|\overline w_n\|_K=m_1(\vec\xi)\asymp M
    \qquad\text{on }\mathcal E_{\vec X}.
\]
Since the midpoint of $\ermp$ and $\overline w_n$ is also an interpolator,
$2$-uniform convexity and the minimality of $\ermp$ imply
\[
    t\|\ermp-\overline w_n\|_K^2
    \lesssim
    \bigl(\|\overline w_n\|_K-\|\ermp\|_K\bigr)
    \bigl(\|\overline w_n\|_K+\|\ermp\|_K\bigr).
\]
It follows from \eqref{Eq:normgap} that
\begin{equation}\label{Eq:kdist}
    \|\ermp-\overline w_n\|_K\lesssim\Delta M
    \qquad\text{on }\mathcal E_{\vec X}.
\end{equation}

By \Cref{a:sym},
\[
    \widetilde w_n(-\vec\xi)
    =-\widetilde w_n(\vec\xi).
\]
Replacing $\mathcal E_{\vec X}$ by
$\mathcal E_{\vec X}\cap(-\mathcal E_{\vec X})$, we may assume that it is
symmetric. Hence
\[
    \E_{\vec\xi}
    \left[
        \widetilde w_n(\vec\xi)
        \mathbbm 1_{\mathcal E_{\vec X}}
    \right]
    =0.
\]

Set $h:=\ermp-\overline w_n$. Since both vectors interpolate the same
observations, $h\in\ker(\vec X)$. Moreover,
\eqref{Eq:kdist} gives $h\in C\Delta M K$ on
$\mathcal E_{\vec X}$. Therefore, by
\eqref{Eq:lower},
\[
    \|h\|_2^2
    \lesssim
    \Delta^2M^2\cG_n(K)^2
    \lesssim
    \Lambda^2\Delta^2
    \qquad\text{on }\mathcal E_{\vec X},
\]
where the last inequality follows from
\eqref{Eq:M206}. Consequently, for every
$\vec X\in\Omega_A$ in the remaining regime,
\begin{align*}
    \left\|
        \Pi_H\E_{\vec\xi}
        \left[
            \ermp\mathbbm 1_{\mathcal E_{\vec X}}
        \right]
    \right\|_2^2
    &=
    \left\|
        \Pi_H\E_{\vec\xi}
        \left[
            (\ermp-\overline w_n)
            \mathbbm 1_{\mathcal E_{\vec X}}
        \right]
    \right\|_2^2\\
    &\leq
    \E_{\vec\xi}
    \left[
        \|\ermp-\overline w_n\|_2^2
        \mathbbm 1_{\mathcal E_{\vec X}}
    \right]\\
    &\lesssim \Lambda^2\Delta^2.
\end{align*}
Adding the contribution of $\mathcal E_{\vec X}^c$, we obtain
\[
    \left\|\Pi_H\E_{\vec\xi}\ermp\right\|_2^2
    \lesssim
    \Lambda^2\Delta^2
    +
    \left\|
        \E_{\vec\xi}
        \left[
            \ermp\mathbbm 1_{\mathcal E_{\vec X}^c}
        \right]
    \right\|_2^2.
\]
After averaging over $\vec X$, the second term, the contribution of
$\Omega_A^c$, and the exceptional contribution left aside in the estimate
of $E_1$ are negligible by conditional Jensen's inequality, H\"older's
inequality, and \eqref{Eq:bad}. For the latter contribution, we
also use
\[
    |\delta_*|\lesssim 1+\|\ermp\|_2.
\]
Since the exponents in the definitions of $\Omega_A$ and
$\mathcal E_{\vec X}$ may be chosen arbitrarily large, these contributions
are smaller than $n^{-1}$. The parts of $\Omega_A$ discarded in the two
reductions above were already controlled by \eqref{Eq:conditionalBI}.

Using \eqref{Eq:M206} in the definition of $\Delta$ and combining the
preceding estimate with \eqref{eq:e1final} and
\eqref{eq:e1finalSmooth}, we obtain
\[
    T_1\lesssim \frac1{t^2}
    \begin{cases}
        \Lambda^8\left(\cG_n(K)^2+\dfrac{\log(d)^2}{n}\right),
            & \text{in general},\\[2mm]
        s^2\Lambda^{2p+6}\cG_n(K)^{2p}
        +\dfrac{\Lambda^8\log(d)^2}{n},
            & \text{under $p$-uniform smoothness}.
    \end{cases}
\]
Since $\Lambda\lesssim\log(d)^3$ in isotropic position, these estimates
give the claimed bounds.

It remains to prove \Cref{lem:good2,Lem:Redu2}.
\begin{proof}[Proof of \Cref{lem:good2}]
Condition on $\vec X|_H$ and fix $\vec\xi$ on the event
\eqref{Eq:maM}. For fixed $(a,\mu)$, let $u_+$ and $u_-$ be
minimizers defining $m_a(\vec\xi+\mu\vec\eta)$ and
$m_a(\vec\xi-\mu\vec\eta)$. Their midpoint is feasible for
$m_a(\vec\xi)$, while $(u_+-u_-)/2$ interpolates
$\mu\vec\eta$ in $H$. The $2$-uniform convexity inequality and the
symmetry of $\vec\eta$ give
\begin{equation}\label{eq:2moment}
    \E_{\vec\eta}m_a(\vec\xi+\mu\vec\eta)^2
    \geq
    m_a(\vec\xi)^2
    +t\mu^2\E_{\vec\eta}\|\vec\eta\|_{H,n}^2.
\end{equation}

Let $L:=\vec X K_H$. Since
$\E_{\vec\eta}\|\vec\eta\|_2^2
=n\|\lft\|_2^2\asymp n$, duality gives
\[
    n
    \lesssim
    \left(\E_{\vec\eta}\|\vec\eta\|_L^2\right)^{1/2}
    \left(\E_{\vec\eta}h_L(\vec\eta)^2\right)^{1/2}.
\]
The sub-Gaussian comparison inequality, followed by concentration, gives
\[
    \left(\E_{\vec\eta}h_L(\vec\eta)^2\right)^{1/2}
    \lesssim M_g^*(L).
\]
Since $L=\sqrt d\,P_HK_H$,
\[
    M_g^*(L)\asymp\sqrt{dn}\,M_s^*(P_HK_H),
    \qquad
    M\asymp\sqrt{\frac nd}\,M_s(P_HK_H).
\]
Furthermore,
\[
    M_s(P_HK_H)M_s^*(P_HK_H)
    \leq b(P_HK_H)M_s^*(P_HK_H)
    =R_{bM^*}(P_HK_H)
    \leq\Lambda.
\]
It follows that
\begin{equation}\label{eq:623}
    \left(\E_{\vec\eta}\|\vec\eta\|_{H,n}^2\right)^{1/2}
    \gtrsim\frac{n}{M_g^*(L)}
    \gtrsim\frac{M}{\Lambda}.
\end{equation}
Combining \eqref{eq:2moment}, \eqref{eq:623}, and
\eqref{Eq:maM}, we obtain
\begin{equation}\label{eq:2moment2}
    \E_{\vec\eta}m_a(\vec\xi+\mu\vec\eta)^2
    \geq
    m_a(\vec\xi)^2
    \left(1+\frac{ct\mu^2}{\Lambda^2}\right).
\end{equation}

The map $\vec\eta\mapsto m_a(\vec\xi+\mu\vec\eta)$ is convex and
$|\mu|b(L)$-Lipschitz. Since
\[
    \frac{b(L)}{m_a(\vec\xi)}
    \lesssim\frac1{\sqrt n}\,
    \frac{b(P_HK_H)}{M_s(P_HK_H)}
    \leq\frac{\Lambda}{\sqrt n},
\]
the CCP, with constant $O(\sqrt{\log(d)})$, implies
\[
    \frac{\Var_{\vec\eta}(m_a(\vec\xi+\mu\vec\eta))}
         {m_a(\vec\xi)^2}
    \lesssim\frac{\mu^2\Lambda^2\log(d)}{n}.
\]
This term is absorbed by the positive term in
\eqref{eq:2moment2} under
\eqref{Eq:largen}. Therefore
\[
    \E_{\vec\eta}m_a(\vec\xi+\mu\vec\eta)
    \geq
    \left(1+\frac{ct\mu^2}{\Lambda^2}\right)m_a(\vec\xi).
\]
A second application of the CCP gives, with probability at least $1-d^{-A}$,
\[
    m_a(\vec\xi+\mu\vec\eta)
    \geq
    \left(
        1+\frac{ct\mu^2}{\Lambda^2}
        -\frac{C|\mu|\Lambda\log(d)}{\sqrt n}
    \right)m_a(\vec\xi)
\]
for each fixed $(a,\mu)$. Taking the exponent in this probability bound large
enough and using a sufficiently fine polynomial net of $[0,2]\times[-1,1]$
makes the estimate uniform. Indeed, \eqref{Eq:maaLipschitz} and the corresponding
high-probability bound $\|\vec\eta\|_{H,n}\lesssim M\Lambda$ control the
error between neighboring net points. Finally, the complement of
\eqref{Eq:maM} has Gaussian measure at most $d^{-A}$. This proves
\Cref{lem:good2}.
\end{proof}

\begin{proof}[Proof of \Cref{Lem:Redu2}]
The map $a\mapsto m_a(\vec\xi)$ is even and convex, hence nondecreasing on
$[0,\infty)$. The crude estimates follow from
\eqref{Eq:maaLipschitz}.

Assume that $\|\cdot\|_K$ is $p$-uniformly smooth, and let $u_0$ be a
minimizer defining $m_0(\vec\xi)$. By central symmetry and \Cref{a:sym}, the vectors
$u_0+a\lft$ and $u_0-a\lft$ have the same norm in $K$. Uniform smoothness
therefore gives
\begin{equation}\label{Eq:smoothness2}
    m_a(\vec\xi)^p
    \leq\|a\lft+u_0\|_K^p
    \leq m_0(\vec\xi)^p+s|a|^p.
\end{equation}
Set $\phi(a):=m_a(\vec\xi)^p$. The function $\phi$ is convex. Hence,
for $0\leq\delta\leq1$,
\[
    \phi(1)-\phi(1-\delta)
    \leq\delta\bigl(\phi(2)-\phi(1)\bigr)
    \lesssim s\delta,
\]
and also $\phi(1)-\phi(0)\leq s$. Since
$m_a(\vec\xi)\asymp M$ for
$0\leq a\leq2$, the two smooth estimates follow.
\end{proof}

It remains to remove the exceptional events. Let $\mathcal B$ denote their
union. The exponents above can be chosen so that
$\Pr(\mathcal B)\leq d^{-A}$ for arbitrary fixed $A$. By conditional Jensen
and H\"older's inequality, for any fixed $q>2$,
\[
    \E_{\vec X}\left\|
        \E_{\vec\xi}[\ermp\mathbbm 1_{\mathcal B}\mid\vec X]
    \right\|_2^2
    \leq
    \E_{\vec X,\vec\xi}
        \bigl[\|\ermp\|_2^2\mathbbm 1_{\mathcal B}\bigr]
    \leq
    \left(\E\|\ermp\|_2^q\right)^{2/q}
    \Pr(\mathcal B)^{1-2/q}.
\]
By \eqref{Eq:bad}, this is smaller than the displayed bounds when
$A$ is chosen sufficiently large. This controls the exceptional contribution
to $E_2$. For $E_1$, use
$|\delta_*|\lesssim1+\|\ermp\|_2$, which follows from
$\|\lft\|_2\asymp1$.

\paragraph{Other positions and Gaussian covariates.}
In John's position, substituting
$\Lambda\lesssim\log(d)/\sqrt t$ into the preceding estimates gives
\[
    T_1
    \lesssim
    \frac{\log(d)^8}{t^6}\cG_n(K)^2
    +\frac{\log(d)^{10}}{t^6n}
    \lesssim
    \frac{\log(d)^{10}}{t^6}
    \left(\cG_n(K)^2+\frac1n\right).
\]
Under $p$-uniform smoothness, it gives
\[
    T_1
    \lesssim
    \frac{s^2\log(d)^{2p+6}}{t^{p+5}}\cG_n(K)^{2p}
    +\frac{\log(d)^{10}}{t^6n}.
\]
Since $t\leq1$ and $p\leq2$, the latter is also bounded by
\[
    T_1
    \lesssim
    \frac{\log(d)^{10}}{t^7}
    \left(s^2\cG_n(K)^{2p}+\frac1n\right).
\]
For Gaussian covariates, rotational invariance allows us to rotate $\lft$ to
the first coordinate, so the sparsity assumption is unnecessary; the
reflection symmetry in \Cref{a:sym} is still required in the direction of
$\lft$. Finally, in the proportional Gaussian regime, the stability of
$M$-position under random proportional-dimensional quotients and sections
gives \eqref{Eq:lanmbda} with $\Lambda\lesssim1$, and the same
argument applies.

\subsection{Proof of \Cref{T:UC2Var}}
To prove the lower bound on the conditional variance, we first record the
conditional ``intrinsic variance'' estimate that is used in the argument. In
contrast to the previous formulation, all quantities below are defined after
conditioning on the design.

\begin{theorem}[Intrinsic variance under $CO(2)$---reverse Efron--Stein]
\label{T:CO2Var}
Let $G=(G_1,\ldots,G_n)\sim N(0,I_n)$, and let
$F\in L_2(\gamma_n;\bana)$ satisfy $\E F(G)=0$. Define
\[
    \alpha_i:=\E\big[G_iF(G)\big],
    \qquad i\in[n].
\]
If $\bana$ is $CO(2)$ with lower constant $t>0$, then
\begin{equation}\label{Eq:covarintrinsic}
    \E\|F(G)\|^2
    \gtrsim
    \frac{t^2}{K_{\|\cdot\|}^2}
    \sum_{i=1}^n\|\alpha_i\|^2.
\end{equation}
\end{theorem}

\begin{proof}[Proof of \Cref{T:CO2Var}]
Consider the linearization of $F$ with respect to the span of the Hermite
polynomials of degree one,
\[
    L_F(G):=\sum_{i=1}^n\alpha_iG_i.
\]
Using the $K$-convexity estimate,
\[
    \E\|L_F(G)\|^2
    \lesssim
    K_{\|\cdot\|}^2\E\|F(G)\|^2.
\]
On the other hand, the Gaussian formulation of cotype $2$ gives
\[
    \E\left\|\sum_{i=1}^n\alpha_iG_i\right\|^2
    \gtrsim
    t^2\sum_{i=1}^n\|\alpha_i\|^2.
\]
Combining the last two displays proves \eqref{Eq:covarintrinsic}.
\end{proof}

\subsubsection{Proof of \Cref{T:UC2Var}}
\begin{proof}[Proof of \Cref{T:UC2Var}]
Recall that $\|\cdot\|_{\PP}$ denotes the $L_2(\PP)$ norm. We lower-bound
\[
    T_2
    =
    \E_{\vec X}\E_{\vec\xi}
    \left\|
        \erm-
        \E_{\vec\xi}[\erm\mid\vec X]
    \right\|_{\PP}^2.
\]
Let $\Omega$ be the regular event in Assumption~\ref{A:DeviationOfSuperma} on
which
\[
    M_g(\cF_n)\asymp M_n(\cF).
\]
For a fixed realization of $\vec X$, write
\[
    \widehat f_{\vec X}(\vec z):=\erm(\vec X,\vec z),
    \qquad \vec z\in\R^n,
\]
and set
\[
    \vec f_\star
    :=
    \bigl(\ft(X_1),\ldots,\ft(X_n)\bigr).
\]
On $\Omega$, define the centered nonlinear map
\begin{equation}\label{Eq:MXDefinition}
    M_{\vec X}(\vec\xi)
    :=
    \widehat f_{\vec X}(\vec f_\star+\vec\xi)
    -
    \E_{\vec\xi'}
    \widehat f_{\vec X}(\vec f_\star+\vec\xi'),
\end{equation}
where $\vec\xi'$ is an independent copy of $\vec\xi$. On $\Omega^c$, set
$M_{\vec X}:=0$. Thus, on $\Omega$,
\[
    \E_{\vec\xi}M_{\vec X}(\vec\xi)=0
\]
and, for every $j\in[n]$,
\begin{equation}\label{Eq:MXInterpolatesNoise}
    M_{\vec X}(\vec\xi)(X_j)=\xi_j.
\end{equation}

For $i\in[n]$, define the coefficients of the linearization  as 
\begin{equation}\label{Eq:AlphaFirstChaos}
    \alpha_i
    :=
    \E_{\vec\xi}
    \big[\xi_iM_{\vec X}(\vec\xi)\big].
\end{equation}
Equivalently, if
$\vec\xi_\pm^i=(\xi_1,\ldots,\pm|\xi_i|,\ldots,\xi_n)$, then
\[
    \alpha_i
    =
    \E_{\vec\xi}\left[
        \frac{|\xi_i|}{2}
        \left(
            M_{\vec X}(\vec\xi_+^i)
            -
            M_{\vec X}(\vec\xi_-^i)
        \right)
    \right].
\]
By \eqref{Eq:MXInterpolatesNoise} and
$\E(\xi_i\xi_j)=\delta_{ij}$, on $\Omega$,
\begin{equation}\label{Eq:AlphaInterpolatesSpike}
    \bigl(\alpha_i(X_1),\ldots,\alpha_i(X_n)\bigr)
    =
    \vec e_i.
\end{equation}
Thus, each $\alpha_i$ interpolates the corresponding one-point spike.

We next upper-bound the conditional Banach-valued fluctuation on the regular
event. By Jensen's inequality,
\[
    \E_{\vec\xi}\|M_{\vec X}(\vec\xi)\|^2
    \leq
    4\E_{\vec\xi}
    \left\|
        \widehat f_{\vec X}(\vec f_\star+\vec\xi)
    \right\|^2.
\]
The function
$\ft+\widehat f_{\vec X}(\vec\xi)$ is a feasible interpolator of
$\vec f_\star+\vec\xi$, and hence minimality gives
\[
    \left\|
        \widehat f_{\vec X}(\vec f_\star+\vec\xi)
    \right\|
    \leq
    \|\ft\|+\|\vec\xi\|_{\cF_n}.
\]
Therefore, Gaussian moment comparison for norms (cf.\ Borell's lemma; see,
e.g., \cite{brazitikos2014geometry}) yields, on $\Omega$,
\begin{align}
    \E_{\vec\xi}\|M_{\vec X}(\vec\xi)\|^2
    &\lesssim
    \|\ft\|^2+
    \E_{\vec\xi}\|\vec\xi\|_{\cF_n}^2\notag\\
    &\lesssim
    1+M_g(\cF_n)^2
    \lesssim
    M_n(\cF)^2.
    \label{Eq:MXSecondMomentUpper}
\end{align}
In the last step, we used Assumptions~\ref{A:One} and~\ref{A:NS} to absorb
the absolute constant into $M_n(\cF)^2$.

Apply \Cref{T:CO2Var} conditionally on $\vec X$, with $F=M_{\vec X}$. Combining
\eqref{Eq:covarintrinsic} and \eqref{Eq:MXSecondMomentUpper}, we obtain, on
$\Omega$,
\begin{equation}\label{Eq:AlphaBanachAverage}
    \frac1n\sum_{i=1}^n\|\alpha_i\|^2
    \lesssim
    \frac{K_{\|\cdot\|}^2M_n(\cF)^2}{t^2n}.
\end{equation}
This is the point at which the square of the $K$-convexity constant is
needed.

Let $I$ be uniform on $[n]$ and independent of the design. For
\[
    r
    :=
    \frac{C K_{\|\cdot\|}M_n(\cF)}{t\sqrt n},
\]
where $C>0$ is a sufficiently large absolute constant,
\eqref{Eq:AlphaBanachAverage}, Assumption~\ref{A:DeviationOfSuperma}, and
Markov's inequality imply
\begin{equation}\label{Eq:AlphaSmallNormEvent}
    \Pr_{\vec X,I}
    \left(
        \Omega,\ \|\alpha_I\|\leq r
    \right)
    \geq\frac58.
\end{equation}

Recall that
\[
    \rho_{\vec X}(\vec v,r)
    :=
    \inf\left\{
        \|f\|_{\PP}:
        f\in r\cF,\
        \bigl(f(X_1),\ldots,f(X_n)\bigr)=\vec v
    \right\}
\]
By exchangeability of the observations and the definition of the median,
\begin{equation}\label{Eq:PsiMedianEvent}
    \Pr_{\vec X,I}
    \left(
        \rho_{\vec X}(\vec e_I,r)
        \geq
        \Psi_n(\vec e_1,r)
    \right)
    \geq\frac12.
\end{equation}
The events in \eqref{Eq:AlphaSmallNormEvent} and \eqref{Eq:PsiMedianEvent}
therefore intersect with probability at least $1/8$. On this intersection,
\eqref{Eq:AlphaInterpolatesSpike} and $\|\alpha_I\|\leq r$ imply that
$\alpha_I\in r\cF$, and hence
\[
    \|\alpha_I\|_{\PP}
    \geq
    \rho_{\vec X}(\vec e_I,r)
    \geq
    \Psi_n(\vec e_1,r).
\]
Consequently,
\begin{equation}\label{Eq:AlphaPopulationLower}
    \E_{\vec X,I}\|\alpha_I\|_{\PP}^2
    \gtrsim
    \Psi_n(\vec e_1,r)^2.
\end{equation}
Finally, since $L_2(\PP)$ is a Hilbert space, orthogonal projection onto
the first Gaussian chaos is contractive in
$L_2(\gamma_n;L_2(\PP))$. Therefore,
\begin{align*}
    T_2
    &\geq
    \E_{\vec X,\vec\xi}
    \|M_{\vec X}(\vec\xi)\|_{\PP}^2\\
    &\geq
    \E_{\vec X,\vec\xi}
    \left\|
        \sum_{i=1}^n\alpha_i\xi_i
    \right\|_{\PP}^2\\
    &=
    \sum_{i=1}^n
    \E_{\vec X}\|\alpha_i\|_{\PP}^2\\
    &=
    n\,\E_{\vec X,I}\|\alpha_I\|_{\PP}^2.
\end{align*}
Here, the second inequality is Hilbert-space orthogonality, the first equality
uses the orthogonality of the Gaussian coordinates, and the final equality is
simply the average over the independent uniform index $I$. Combining the last
display with \eqref{Eq:AlphaPopulationLower} proves
\[
    T_2
    \gtrsim
    n\,
    \Psi_n\left(
        \vec e_1,
        \frac{C K_{\|\cdot\|}M_n(\cF)}{t\sqrt n}
    \right)^2,
\]
which is \eqref{Eq:UC2CV}.
\end{proof}

\subsubsection{Concluding remarks}

\begin{remark}[On Efron--Stein-type bounds in Banach spaces]
\label{R:ReverseEfronStien}
Recall the classical Efron--Stein inequality: if
$F=F(X_1,\ldots,X_n)$ is scalar-valued and $X_i'$ is an independent copy of
$X_i$, then
\[
    \Var(F)
    \leq
    \frac12\sum_{i=1}^n
    \E\left[
        \bigl(
            F(X_1,\ldots,X_n)
            -
            F(X_1,\ldots,X_i',\ldots,X_n)
        \bigr)^2
    \right].
\]
Related coordinate-increment inequalities for maps from the discrete cube to
Hilbert and type-$2$ spaces originate in the work of
\cite{enflo1970nonexistence,enflo1970uniform}; see also
\cite{pisier2006probabilistic,ivanisvili2020rademacher}. For instance, if
$\varepsilon^{(i)}$ denotes $\varepsilon$ with its $i$-th coordinate flipped,
then a Hilbert-valued map satisfies
\[
    \E_\varepsilon
    \|F(\varepsilon)-\E_\varepsilon F(\varepsilon)\|_{\cH}^2
    \lesssim
    \sum_{i=1}^n
    \E_\varepsilon
    \|F(\varepsilon)-F(\varepsilon^{(i)})\|_{\cH}^2.
\]

The mechanism in the proof above runs in the reverse direction. The first
Gaussian-chaos projection extracts the coordinate influences
$\alpha_1,\ldots,\alpha_n$ from the nonlinear map
$\vec\xi\mapsto\erm(\vec X,\vec f_\star+\vec\xi)$. The $K$-convexity
estimate ensures that this linear projection cannot be much larger than the
full nonlinear fluctuation, while cotype $2$ gives
\[
    \frac{t^2}{K_{\|\cdot\|}^2}
    \sum_{i=1}^n\|\alpha_i\|^2
    \lesssim
    \E_{\vec\xi}\|M_{\vec X}(\vec\xi)\|^2.
\]
In this sense, this is the closest analog in our setting
to a reverse Efron--Stein inequality.
\end{remark}

\begin{remark}
From a statistical point of view, the interpolation identity
\[
    \bigl(\alpha_i(X_1),\ldots,\alpha_i(X_n)\bigr)=\vec e_i
\]
explains the appearance of $\Psi_n$. The linearization coefficients belongs to
the class of functions that interpolate $n-1$ sample points by zero and one
sample point by one. The theorem lower-bounds the conditional variance by the
population $L_2(\PP)$ cost of interpolating these one-point spikes at the
Banach-norm radius of the coefficients of the linearization of the MNI.
\end{remark}

\begin{remark}
An analogous formulation is possible for Rademacher noise by using the first degree Walsh coefficients and the corresponding $K$-convexity projection
constant.
\end{remark}

\subsection{Proof of \Cref{T:VarLinOpt}}
We recommend reading \Cref{S:ThmLinear} before moving to this proof and retain
its notation. In particular,
\[
    \vec\zeta\sim N(0,I_n/n),
    \qquad
    \eps_d:=p-1.
\]
Throughout this proof, $\|\cdot\|_p$ denotes the gauge of
$\widetilde B_p^d$, and $M,M^*$ abbreviate $M_s,M_s^*$, respectively.
For $S\subset[d]$, $\vec X^{(S)}$ denotes the submatrix formed by the
columns indexed by $S$, and $\vec X^{(-1)}$ is obtained by deleting the
first column. As discussed there, it is enough to prove
\Cref{T:WRD}.

\subsubsection{Step I: $\ell_p$-geometry}

The classical result \citep[Thm. 1]{schutt1984entropy} states that the covering numbers of the $\widetilde B_p^d$ ball with respect to the $\ell_2$ norm satisfy the following:
\[
\log\cN\lp\widetilde B_p^d,\lambda_{k}B_2^d\rp = k,
\]
where
\[
   \lambda_{k} \asymp \left\{\begin{array}{lr}
       \lp\frac{d\log(ed/k)}{k}\rp^{\frac1p-\frac12} & \text{for } \log(d) \leq k \leq d\\
d^{\frac1p-\frac12} \quad  &\text{for } 1 \leq k \leq \log(d)\\
        \end{array}\right\}.
\]
We also use the following\footnote{One can construct a net using these $\cV_k$, but as our proof does not require this, we omit this part.} sets $\cV_k$ for indices $k \geq \lfloor\log d\rfloor$:
\begin{equation}\label{Eq:Vk}
\cV_k := \bigcup_{\sigma \subset [d]} A_\sigma,
\quad \text{where } \frac{k}{2\log(d/k)} \leq |\sigma| \leq \frac{k}{\log(d/k)}.
\end{equation}
where $A_{\sigma}\subset \ell_p^{d}$ is defined as
\[
A_{\sigma}
:=\left(\frac{d}{k}\right)^{\frac1p-\frac12}\frac{1}{\sqrt{|\sigma|}}
\begin{cases}
\left(\log\frac{d}{k}\right)^{\frac1p-\frac12}\, B_{\infty}^{\sigma}
& C_1\,\dfrac{k}{\log(d/k)} \leq |\sigma| \leq 2C_1\,\dfrac{k}{\log(d/k)}\\[6pt]
B_{\infty}^{\sigma} & |\sigma|\asymp 1
\end{cases}
\]
Now, we construct a net on $\widetilde{\cV}_k$ by taking a $(k/d)^{C_1}$-net (in the $\ell_2$ metric) of each cube in $\cV_k$. To see that $\log |\widetilde{\cV}_k| \lesssim k$, note that
\[
    \log |\widetilde{\cV}_k| \lesssim \log\lp\binom{d}{\frac{c_1k}{\log(d/k)}}\cdot \lp\frac{(d/k)^{\frac1p-\frac12}}{(k/d)^{C_1}}\rp^{\frac{c_1k}{\log(d/k)}}\rp \lesssim k
\]
where we used that for any $k \geq 0$,
\[
\log \binom{d}{\frac{c_1k}{\log(d/k)}} \lesssim k,
\]
and  
\[
\cN(\tfrac{R}{\sqrt{|\sigma|}} B_{\infty}^{\sigma},\tau B_2^{\sigma}) \leq \cN(RB_2^{\sigma},\tau B_2^{\sigma}) \leq (1+2R/\tau)^{|\sigma|}.
\]
For any vector $v\in \R^d$, we have:
\begin{equation}\label{Eq:JL}
    \Pr_{P}\left( \|Pv\|_{2} > (1 + \tau)\sqrt{\frac{n}{d}}\|v\|_2\right) \leq \exp(-cn\tau^2).
\end{equation}
We used the fact that $\E \|Pv\|_{2}^2 = \frac{n}{d}\|v\|_{2}^2$, and therefore $\E \|Pv\|_{2} \leq  \sqrt{\frac{n}{d}}\|v\|_{2}$. When $P$ is Gaussian, inequality \eqref{Eq:JL} is immediate since $Pv$ is again Gaussian. When $P$ is sub-Gaussian, the inequality still holds; it is the well-known fact that a sum of independent sub-Gaussian variables is again sub-Gaussian (see, e.g., \cite{vershynin2018high}). We will refer to it as the (sub-Gaussian) Hoeffding inequality.

Now, using the Hoeffding inequality with $\tau \asymp \max\{\sqrt{k/n},1\}$ over $\widetilde{\cV}_k$ for any $1 \leq k \leq d$, it holds that\footnote{Note that $\cV_k$ is not a finite set, the statement requires an elementary net argument which is omitted. }
\begin{align}\label{Eq:JL-net}
\Pr\Bigg(
\forall v \in \cV_k:\;
\frac{\|Pv\|_2}{\|v\|_p}
&\;\lesssim\;
\log\!\Big(\frac{d}{k}\Big)^{\frac{1}{p}-\frac12}\Big(\frac{d}{k}\Big)^{\frac{1}{p}-\frac12}
\left[
\underbrace{\sqrt{\frac{k}{d}}}_{(*)}
+
\underbrace{
\sqrt{\frac{n}{d}}}_{(**)}
\right]
\Bigg)
\geq 1-\exp\!\big(-c\,\max\{k,n\}\big).
\end{align}
Note that we can take a union bound over $1\leq k\leq d$ to ensure that this holds uniformly for all $k$, with probability at least $1-\exp(-cn)$. When $k \gtrsim n$ in the last estimate, the term $(*)$ arises from the \textbf{deviations}, i.e., from taking a union bound over $2^{k}$ elements, and is therefore larger than the term $(**)$, which arises from the \textbf{expectation} of each of these random variables. To be consistent with Step 0 above, we use the notation
\begin{equation}
\bar{\lambda}_k = \sup_{w_k \in \cV_{k\log(d/k)}}\|w_k\|_2 \asymp (d/k)^{\frac1p-\frac12}.    
\end{equation}
Note that under the high-probability event defined above, we have 
\begin{equation}\label{Eq:bR1}
    \sup_{w_k \in \cV_{k\log(d/k)}}\|Pw_k\|_2 \leq (k/d)^{1-1/p}\sqrt{\log(d/k)}.
\end{equation}
Our goal is to establish the following high-probability estimate.
\begin{lemma}\label{lem359}
    With probability at least $1-\exp(-cn)$, we have 
    \[
    \|\vec\zeta\|_{n} \asymp \sqrt{\eps_d}.
    \]
\end{lemma}
First, we prove that this holds after averaging over the noise. That is, we establish that with probability at least $1-\exp(-cn)$ over $\vec X$, we have 
\begin{equation}\label{eq364}
    \E_{\vec\zeta}\|\vec\zeta\|_{n} \asymp M(P\widetilde B_p^d) \asymp \sqrt{\eps_d}.
\end{equation}
The first equivalence is the usual comparison between spherical and Gaussian averaging. We begin by upper bounding $M(P \widetilde B_p^d)$ by showing that $P \widetilde B_p^d$ contains a sufficiently large Euclidean ball.

\begin{lemma}\label{lem_dvoretzky_inclusion}
With probability at least $1-\exp(-cn)$, we have
\begin{equation}\label{Eq:IncPrjball}
          c\eps_d^{-1/2} \cdot B_2^n \subset  P\left(\mathrm{conv}\left\{\cV_{\frac{d}{\log(d)}}\right\}\right) \subset P \widetilde B_p^d 
\end{equation} 
where $c>0$ is a universal constant.
\end{lemma}

\begin{proof}
The key difficulty is that we cannot apply Dvoretzky's theorem to the entire $\ell_{p}$-ball, as its critical dimension is too small, namely $d(B_p^d) \ll n$. Thus, for a suitable choice of $k$, we apply it to the subset 
\[
    \mathrm{conv}\{\cV_k\},
\]
which has a significantly smaller $\ell_2$-diameter while retaining the same Gaussian complexity as the full ball, up to constants. As we shall see, the optimal choice is $k \asymp d/\log d$. To avoid unnecessary notational burden, we assume that $k = d/\log d$ is an integer. Set
\[
L:=\log(d/k)=\log\log d,
\qquad
s:=\frac{k}{L}\asymp \frac{d}{\log d\,\log\log d}.
\]
In this regime $\cV_k$ is a union of sets $A_\sigma$ with $|\sigma|\asymp s$, where
\[
A_\sigma = a\,B_\infty^\sigma,
\qquad
a:=\Big(\frac{d}{k}\Big)^{\frac1p-\frac12}\frac{1}{\sqrt{s}}\,L^{\frac1p-\frac12}.
\]
Thus, the diameter of $\cV_k$ is 
\[
R(\cV_k)\asymp \Big(\frac{d}{k}\Big)^{\frac1p-\frac12}L^{\frac1p-\frac12}
= (\log d)^{\frac1p-\frac12}(\log\log d)^{\frac1p-\frac12}.
\]

We now want to compute the width of $\cV_k$. For any $u\in\R^d$,
\[
h_{\cV_k}(u)
=\sup_{x\in \cV_k}\langle u,x\rangle
=\sup_{|\sigma|\asymp s}\ \sup_{x\in A_\sigma}\langle u,x\rangle
=\sup_{|\sigma|\asymp s}\ a\sum_{i\in\sigma}|u_i|.
\]
Let $|u|_{(1)}\ge\cdots\ge|u|_{(d)}$ denote the non-increasing rearrangement of $(|u_i|)_{i\leq d}$.
\[
h_{\cV_k}(u)\asymp a\sum_{i=1}^s |u|_{(i)}.
\]
Consequently,
\[
M^*(\cV_k)\asymp \frac{a}{\sqrt d}\,\E\sum_{i=1}^s |G|_{(i)},
\qquad G\sim\cN(0,I_d).
\]

Let $q:=s/d\asymp(\log d\,\log\log d)^{-1}$ and let $t>0$ be defined by $\P(|g|\geq t)=q$ for $g\sim\cN(0,1)$.
Then $t\asymp\sqrt{\log(1/q)}\asymp\sqrt{\log\log d}$.
Standard gaussian tail estimates yield
\[
\E\sum_{i=1}^s |G|_{(i)}\asymp s\,t\asymp s\sqrt{\log\log d}.
\]
Plugging this into the expression for $M^*(\cV_k)$ gives
\[
M^*(\cV_k)
\asymp
\Big(\frac{d}{k}\Big)^{\frac1p-\frac12}L^{\frac1p-\frac12}
\sqrt{\frac{s}{d}}\;\sqrt{\log\log d}.
\]
Since $\sqrt{s/d}=\sqrt{k/(dL)}=(\log d\,\log\log d)^{-1/2}$ and $d/k=\log d$, we conclude that
\[
M^*(\cV_k)
\asymp
(\log d)^{\frac1p-1}(\log\log d)^{\frac1p-1}\,\sqrt{\log\log d}
\asymp \sqrt{\log\log d},
\]
where we used that $\eps_d=\Theta(1/\log\log d)$, hence $(\log d)^{-\eps_d}\asymp1$ and $(\log\log d)^{-\eps_d}\asymp1$.

By what precedes we have $\frac{M^*(\cV_k)}{R(\cV_k)}\asymp \frac{1}{\sqrt{\log d}}$, hence the critical dimension of $\cV_k$ is of order $d/\log d \gg n$.
Now, we apply Dvoretzky's  \Cref{Lem:Dvo}, and conclude that
\[
   c_1 M^*(\cV_{\frac{d}{\log(d)}})B_2^n \subset P\left(\mathrm{conv}\left\{\cV_{\frac{d}{\log(d)}}\right\}\right). 
\]
As a consequence,
\[
M(P\widetilde B_p^d) \lesssim  \sqrt{\eps_d}.
\]
\end{proof}
Now, we prove the corresponding lower bound for $M(P \widetilde B_p^d)$. Since for any convex body $K$ we have 
\begin{equation}\label{eq451}
    MM^*(K)\geq 1,
\end{equation}
it suffices to upper bound $M^*( P \widetilde B_p^d)$. Note that when $P$ is Gaussian, it is an exercise to check that 
\[
\EE M^*(PK) \asymp \EE M^*(K).
\]
Applying this to $\widetilde B_p^d$ would yield the desired upper bound in expectation, since 
\[
M^*(\widetilde B_p^d) \asymp \eps_d^{-1/2}.
\]
However, we are interested in high-probability bounds, and using Lipschitz concentration of $P \mapsto M^*(PK)$ is insufficient for our purposes here. Instead, we use a more direct and computational approach, which works when $P$ is merely sub-Gaussian.
\begin{lemma}
Let $1<p\leq 2$, let $q=\frac{p}{p-1}$, and assume that $q\leq n$. Recall that $P=\frac{1}{\sqrt d}X\in\mathbb R^{n\times d}$
where the entries of $X$ are iid, centered, variance-one and subgaussian. Then, with probability at least $1-\exp(-cn)$,
\[
M^*(P\widetilde B_p^d)
\lesssim \sqrt q
\lesssim \eps_d^{-1/2},
\]
where the hidden constants depend only on the subgaussian constant of the entries.
\label{lem_M*_bound}
\end{lemma}

\begin{proof}
Write $P_1,\ldots,P_d\in\mathbb R^n$ for the columns of $P$. We first
establish the following high-probability estimate:
\begin{equation}
\label{eq287}
\left(
\frac{1}{d}\sum_{j=1}^d
\left(
\sqrt{\frac d n}\,\|P_j\|_2
\right)^q
\right)^{1/q}
\lesssim 1.
\end{equation}

Set
\[
Z_j:=\sqrt{\frac d n}\,\|P_j\|_2
=
\left(
\frac1n\sum_{i=1}^n X_{ij}^2
\right)^{1/2}.
\]
Since $X_{ij}^2-1$ is sub-exponential, Bernstein's inequality implies
that, for every $t\geq 2$,
\begin{equation}
\label{eq308}
\mathbb P(Z_j\geq t)
\leq 2\exp(-c_0nt^2),
\end{equation}
where $c_0>0$ depends only on the sub-Gaussian constant.

Choose a sufficiently large constant $A\geq 2$, to be fixed below. For
$m\geq 1$, define
\[
N_m:=\#\{j\leq d:Z_j\geq A2^m\},
\qquad
a_m:=2^{-m(q+2)}.
\]
We will prove that with probability at least
$1-\exp(-cn)$,
\begin{equation}
\label{eq_324}
N_m\leq da_m
\quad\text{for }1\leq m<m_0,
\qquad
N_{m_0}=0.
\end{equation}
Note that $N_m$ follows a binomial law with parameters $(d,p_m)$, where, by \eqref{eq308},
\[
p_m:=\mathbb P(Z_j\geq A2^m)
\leq 2\exp(-c_An4^m),
\]
where $c_A=c_0A^2$ can be made arbitrarily large by increasing $A$. We split indices as follows, let
\[
m_0:=\min\{m\geq 1:da_m<1\}.
\]
For $1\leq m<m_0$, we have $da_m\geq 1$. Moreover, since $q\leq n$ and
$4^m\geq 4m$, by choosing $A$ sufficiently large we have
\[
\log\left(\frac{a_m}{ep_m}\right)
\geq c_An4^m-m(q+2)\log 2-\log(2e)
\geq c_1n4^m.
\]
Since $N_m$ is binomial with parameters $d$ and $p_m$, the standard
binomial upper-tail bound gives
\[
\mathbb P(N_m\geq da_m)
\le
\left(\frac{ep_m}{a_m}\right)^{da_m}
\le
\exp\left(-c_1nda_m4^m\right).
\]
Using the definition of $a_m$, and $da_m\geq 1$, we verify that $d\,2^{-mq}
\geq d^{2/(q+2)}$. Consequently,
\begin{equation}
\label{eq358}
\mathbb P(N_m\geq da_m)
\le
\exp\left(-c_1n d^{2/(q+2)}\right)
\qquad
(1\leq m<m_0).
\end{equation}

It remains to upper-bound the probability that $N_{m_0}>0$. By a simple union bound,
\[
\mathbb P(N_{m_0}\geq 1)
\leq dp_{m_0}
\le
2\exp\left(\log d-c_An4^{m_0}\right).
\]
By the definition we have $4^{m_0}>d^{2/(q+2)}$. Furthermore, using that $q\leq n$,
\[
\log d=\frac{q+2}{2}\log (d^{2/(q+2)})
\leq \frac{n+2}{2}\,d^{2/(q+2)}.
\]
Thus, by taking $A$ sufficiently large,
\begin{equation}
\label{eq:lemma7-last-level}
\mathbb P(N_{m_0}\geq 1)
\le
\exp\left(-c_2n d^{2/(q+2)}\right).
\end{equation}

Combining \eqref{eq358} and
\eqref{eq:lemma7-last-level}, and taking a union bound over
$1\leq m<m_0$, we conclude that, with probability at least
$1-\exp(-cn)$,
\begin{equation}
\label{eq:lemma7-level-event}
N_m\leq da_m
\quad\text{for every }1\leq m<m_0,
\qquad
N_{m_0}=0.
\end{equation}
Indeed, there are at most $1+\log_2 d$ indices, which are absorbed by the exponential factor.

We work on the event \eqref{eq:lemma7-level-event} and we group the columns
according to the size of $Z_j$ :
\begin{align*}
\sum_{j=1}^d Z_j^q
&\le
d(2A)^q
+
\sum_{m=1}^{m_0-1}
N_m(A2^{m+1})^q
\\
&\le
d(2A)^q
+
dA^q2^q
\sum_{m=1}^{m_0-1}
2^{mq}a_m
\\
&=
d(2A)^q
+
dA^q2^q
\sum_{m=1}^{m_0-1}2^{-2m}
\\
&\le
d(CA)^q.
\end{align*}
This proves \eqref{eq287}. Equivalently, on the same
event,
\begin{equation}
\label{eq:lemma7-column-lq}
\left(\sum_{j=1}^d\|P_j\|_2^q\right)^{1/q}
\lesssim
d^{1/q}\sqrt{\frac nd}.
\end{equation}

We are now in position to bound the mean width. Let $G\sim N(0,I_n)$ and set $G_n:=\frac{G}{\sqrt n}$, so that $M^*(K) \asymp \mathbb E h_K(G_n)$
for every symmetric convex body $K\subset\mathbb R^n$. Therefore,
conditionally on $P$,
\begin{align*}
M^*(P\widetilde B_p^d)
&\lesssim
\mathbb E_{G_n}
h_{P\widetilde B_p^d}(G_n)
\\
&=
d^{\frac1p-\frac12}
\mathbb E_{G_n}\|P^T G_n\|_q.
\end{align*}
By Jensen's inequality,
$$
\mathbb E_{G_n}\|P^T G_n\|_q=
\mathbb E_{G_n}
\left(
\sum_{j=1}^d
|\langle P_j,G_n\rangle|^q
\right)^{1/q}
\le
\left(
\sum_{j=1}^d
\mathbb E_{G_n}
|\langle P_j,G_n\rangle|^q
\right)^{1/q}.
$$
Conditionally on $P$, the random variable $\langle P_j,G_n\rangle$ is a centered Gaussian with variance $\frac{\|P_j\|_2^2}{n}$. Thus
\[
\mathbb E_{G_n}
|\langle P_j,G_n\rangle|^q
=
\gamma_q^q
\frac{\|P_j\|_2^q}{n^{q/2}},
\]
where $\gamma_q\simeq \sqrt{q}$ is the $q$-th moment of a standard Gaussian. Consequently,
\[
\mathbb E_{G_n}\|P^\top G_n\|_q
\le
\frac{\gamma_q}{\sqrt n}
\left(
\sum_{j=1}^d\|P_j\|_2^q
\right)^{1/q}.
\]
Applying \eqref{eq:lemma7-column-lq}, we obtain
\begin{align*}
M^*(P\widetilde B_p^d)
&\lesssim
d^{\frac1p-\frac12}
\frac{\gamma_q}{\sqrt n}
d^{1/q}\sqrt{\frac nd}
\\
&=
\gamma_q\,
d^{1/p+1/q-1}
\\
&=
\gamma_q 
\end{align*}
We conclude that 
\[
M^*(P\widetilde B_p^d)
\lesssim
\sqrt q
\lesssim
\frac1{\sqrt{\eps_d}}
=
\eps_d^{-1/2}.
\]
This completes the proof.
\end{proof}

Combining Lemma \ref{lem_dvoretzky_inclusion} for the upper bound and Lemma \ref{lem_M*_bound} together with \eqref{eq451} for the lower bound, we have now established \eqref{eq364}. To prove Lemma \ref{lem359}, we work on the event where $\vec X$ satisfies \eqref{Eq:IncPrjball} and where \eqref{eq364} holds. We then use the convex Lipschitz concentration of the map $\vec\zeta \mapsto \frac{\|\vec\zeta\|_n}{\E \|\vec\zeta\|_n}$ (which is $O(1)$-Lipschitz), and conclude that with probability at least $1-\exp(-cn)$, it holds that 
\[
    \|\vec\zeta\|_{n} \asymp \sqrt{\eps_d},
\]
which we condition on \textit{throughout our proof}, as mentioned in Step 0. On this event, the following holds:
\begin{itemize}
    \item For all $k \geq 1$, it holds that $w_k \in \sqrt{\eps_d} \cdot \cV_{\lfloor k\log(d/k) \rfloor}$. 
    \item $\sum_{k \in \cR_1 \cup \cR_2}\delta_k^{p} \asymp 1$. 
    \item We have the estimates
\begin{equation}\label{Eq:stars}
M(P\widetilde B_p^d) \asymp M^*(P\widetilde B_p^d)^{-1} \asymp b(P\widetilde B_p^d) \asymp \sqrt{\eps_d}.
\end{equation}
\end{itemize}
If we knew that $\ermp \in  C\sqrt{\log(d)} \cdot \cC_d$, we could compare the $\ell_2$ norm with the $\ell_p$ norm at a small cost, and we would be done. In the following, we show that all $\delta_k$ are sufficiently small for
\[
    \|\ermp\|_2 = \widetilde{O}(1)
\]
to hold with high probability over $\vec X$, for \textbf{any} fixed vector $\vec\zeta$ such that $\|\vec\zeta\|_{n} \asymp \sqrt{\eps_d}$.

\subsubsection{Step II: Entries in $\cR_1$}
Here, we show that for every $k \in \cR_1$,
\begin{equation}\label{Eq:StationaryR}
    \delta_k \lesssim \left(\frac{k}{d}\right)^{1/p}
    \log(d/k)^{\frac{1}{2\eps_d}}.
\end{equation}
Recall that in the range $\cR_1$, the deviation term dominates the expectation term in \eqref{Eq:JL-net}.

Let
\[
    r:=\|\ermp\|_p \asymp \sqrt{\eps_d}.
\]
Recall from the normalization in the decomposition above that
\[
    \delta_k w_k
    = \Pi_k(\ermp)-\Pi_{k/2}(\ermp),
    \qquad
    \|w_k\|_p=r,
    \qquad
    0\leq \delta_k\leq 1,
\]
and that
\[
    \frac{w_k}{r}
    \in C\cV_{\lfloor k\log(d/k)\rfloor}.
\]
If $\delta_k=0$, there is nothing to prove, so assume below that $\delta_k>0$.
Since $\delta_k w_k$ and $\ermp-\delta_k w_k$ have disjoint supports,
\begin{equation}\label{Eq:BlockRemoval}
    \|\ermp-\delta_k w_k\|_p
    =\left(1-\delta_k^p\right)^{1/p}\|\ermp\|_p.
\end{equation}

By Lemma~\ref{lem_dvoretzky_inclusion},
\[
    c\eps_d^{-1/2} B_2^n
    \subset
    P\!\left(\operatorname{conv}\left\{\cV_{d/\log(d)}\right\}\right).
\]
Hence, there exists an interpolator $\widetilde w_k$ of $\delta_kPw_k$, namely
$P\widetilde w_k=\delta_kPw_k$, such that
\[
    \widetilde w_k
    \in
    C\sqrt{\eps_d}\delta_k\|Pw_k\|_2\,
    \operatorname{conv}\left\{\cV_{d/\log(d)}\right\}.
\]
In particular, since
$\operatorname{conv}\{\cV_{d/\log(d)}\}\subset \widetilde B_p^d$,
\begin{equation}\label{Eq:ReplacementNorm}
    \|\widetilde w_k\|_p
    \lesssim
    \sqrt{\eps_d}\delta_k\|Pw_k\|_2.
\end{equation}
Consider the interpolator of $\vec\zeta$ defined by
\begin{equation}\label{Eq:VariationalOne}
    \bar w_n:=\ermp-\delta_k w_k+\widetilde w_k.
\end{equation}
By the minimality of $\ermp$, the triangle inequality,
\eqref{Eq:BlockRemoval}, and \eqref{Eq:ReplacementNorm},
\[
    r
    \leq \|\bar w_n\|_p
    \leq r\left(1-\delta_k^p\right)^{1/p}
    +C\sqrt{\eps_d}\delta_k\|Pw_k\|_2.
\]
Since $1-(1-u)^{1/p}\geq u/p$ for $u\in[0,1]$, it follows that
\[
    \delta_k^{\eps_d}
    \lesssim
    \frac{\sqrt{\eps_d}}{r}\|Pw_k\|_2
    \lesssim \|Pw_k\|_2.
\]
On the other hand, since $w_k/r\in C\cV_{\lfloor k\log(d/k)\rfloor}$,
\eqref{Eq:bR1} and $r\asymp\sqrt{\eps_d}$ imply
\[
    \|Pw_k\|_2
    \lesssim
    \sqrt{\eps_d}\left(\frac{k}{d}\right)^{1-1/p}
    \sqrt{\log(d/k)}.
\]
Therefore,
\begin{equation}\label{Eq:StatioaryOne}
\begin{aligned}
    \delta_k^{\eps_d}
    &\lesssim
    \left(\frac{k}{d}\right)^{\frac{\eps_d}{p}}
    \left(\eps_d\log(d/k)\right)^{1/2},
    \\
    \delta_k
    &\lesssim
    \left(\frac{k}{d}\right)^{1/p}
    \left(C_1\eps_d\log(d/k)\right)^{\frac{1}{2\eps_d}}
\end{aligned}
\end{equation}
for some absolute constant $C_1>0$. In particular, since $\sqrt{\eps_d}\leq 1$,
this implies \eqref{Eq:StationaryR}.

Finally, by the definition of $\bar\lambda_k$ and the inclusion
$w_k/r\in C\cV_{\lfloor k\log(d/k)\rfloor}$,
\begin{align}
    \delta_k\|w_k\|_2
    &\lesssim r\delta_k\bar\lambda_k \\
    &\lesssim
    \sqrt{\eps_d}\left(\frac{k}{d}\right)^{1/2}
    \left(\frac{C_1\log(d/k)}{\log\log d}\right)^{\frac{1}{2\eps_d}} \\
    &\lesssim
    \left(\frac{k}{d}\right)^{1/2}
    \left(\frac{C_1\log(d/k)}{\log\log d}\right)^{\frac{1}{2\eps_d}}.
    \label{Eq:R1BlockBound}
\end{align}
Here we used $\eps_d=C_0/\log\log d$, and absorbed $C_0$ into $C_1$.
For $C_0$ sufficiently large, the right-hand side of
\eqref{Eq:R1BlockBound} is maximized at the largest index
$k\asymp d/\log d$. Consequently, there exists an absolute constant
$C_2>0$ such that, uniformly over $k\in\cR_1$,
\[
    \delta_k^2\|w_k\|_2^2\lesssim \log(d)^{C_2}.
\]
Since the vectors $(\delta_k w_k)_{k\in\cR_1}$ have pairwise disjoint supports,
\[
    \left\|\sum_{k\in\cR_1}\delta_k w_k\right\|_2^2
    =\sum_{k\in\cR_1}\delta_k^2\|w_k\|_2^2
    \lesssim \log(d)^{C_2+1}.
\]

\subsubsection{Step III: Entries in $\cR_2$}
In this regime, we use a different argument, since the expectation term dominates when
$k \lesssim n/\log(d/n)$; see Step~I. We first control the maximal entry and then treat the remaining dyadic blocks.

We begin with a consequence of Lemma~\ref{lem_dvoretzky_inclusion} that will be used throughout this step. Let
\[
    k_*:=\left\lfloor \frac{n}{\log(d/n)}\right\rfloor.
\]
The proof of Lemma~\ref{lem_dvoretzky_inclusion}, applied after deleting at most $k_*$ coordinates, and a union bound over the possible supports, shows that with high probability,
\begin{equation}\label{Eq:LeaveOutBall}
    c\eps_d^{-1/2} B_2^n
    \subset
    P_{S^c}\bigl(\widetilde B_p^d\cap\mathbb R^{S^c}\bigr)
\end{equation}
simultaneously for every $S\subset[d]$ with $|S|\leq k_*$. Indeed, the proof gives a failure probability of order $\exp(-cd/\log d)$ for each fixed $S$, whereas
$\log\sum_{j\leq k_*}\binom dj\lesssim n\ll d/\log d$.

For such a set $S$, define the normalized quotient norm
\begin{equation}\label{Eq:LeaveOutNorm}
    \|z\|_{\widetilde n,S}
    :=
    \eps_d^{-1/2}
    \inf\bigl\{\|v\|_p:\operatorname{supp}(v)\subset S^c,\ Pv=z\bigr\}.
\end{equation}
By \eqref{Eq:LeaveOutBall},
\begin{equation}\label{Eq:LeaveOutLipschitz}
    \|z\|_{\widetilde n,S}\lesssim \|z\|_2.
\end{equation}
We also work on the event $\|\vec\zeta\|_2\lesssim1$, which holds with probability at least $1-\exp(-cn)$.

\paragraph{Control of the maximal entry}
We write
\[
    \ermp=\delta\sqrt{\eps_d}\lambda_p e_1+w^\perp,
    \qquad
    \lambda_p=d^{\frac1p-\frac12},
    \qquad
    w^\perp\perp e_1,
\]
where we allow $\delta$ to be signed. Since $\|\ermp\|_p\asymp\sqrt{\eps_d}$, we may restrict to $|\delta|\leq C_0$ for an absolute constant $C_0$.
Let $\|\cdot\|_{\widetilde n}:=\|\cdot\|_{\widetilde n,\{1\}}$. By the splitting property of the $\ell_p$ norm,
\begin{equation}\label{Eq:PierreOne_clean4}
    \delta
    =
    \operatorname*{argmin}_{t\in[-C_0,C_0]}
    \left\{
        |t|^p+
        \left\|\vec\zeta-
        \frac{\sqrt{\eps_d}\lambda_p t}{\sqrt d}\vec X^{(1)}
        \right\|_{\widetilde n}^p
    \right\}.
\end{equation}

We discretize $[-C_0,C_0]$ with mesh $d^{-100}$ and denote the resulting grid by $\mathcal Z$. Let $\widetilde\delta$ be the minimizer of the objective in \eqref{Eq:PierreOne_clean4} over $\mathcal Z$. By strict convexity, $\widetilde\delta$ is one of the two grid points adjacent to $\delta$, and therefore
\begin{equation}\label{Eq:GridApproximation}
    |\widetilde\delta-\delta|\leq d^{-100}.
\end{equation}
Condition on $\vec X^{(-1)}$. For fixed $t$, the map
\[
    \vec X^{(1)}\longmapsto
    \left\|\vec\zeta-
    \frac{\sqrt{\eps_d}\lambda_p t}{\sqrt d}\vec X^{(1)}
    \right\|_{\widetilde n}
\]
is convex and, by \eqref{Eq:LeaveOutLipschitz}, has Euclidean Lipschitz constant at most
$C\sqrt{\eps_d}\lambda_p |t|/\sqrt d$. Moreover, Jensen's inequality gives that its expectation is at least $\|\vec\zeta\|_{\widetilde n}$. Hence Gaussian concentration, followed by a union bound over $\mathcal Z$, yields, with probability at least $1-d^{-100}$,
\begin{equation}\label{Eq:PierreMain_clean4}
    \forall t\in\mathcal Z,\qquad
    \left\|\vec\zeta-
    \frac{\sqrt{\eps_d}\lambda_p t}{\sqrt d}\vec X^{(1)}
    \right\|_{\widetilde n}
    \geq
    \|\vec\zeta\|_{\widetilde n}
    -C\sqrt{\log d}\,
    \frac{\sqrt{\eps_d}\lambda_p |t|}{\sqrt d}.
\end{equation}

Fix such a realization. Since $\widetilde\delta$ minimizes the discretized objective and $0\in\mathcal Z$, \eqref{Eq:PierreMain_clean4} implies
\begin{align*}
    |\widetilde\delta|^p
    &\leq
    \|\vec\zeta\|_{\widetilde n}^p
    -
    \left(
        \|\vec\zeta\|_{\widetilde n}
        -C\sqrt{\log d}\,
        \frac{\sqrt{\eps_d}\lambda_p |\widetilde\delta|}{\sqrt d}
    \right)_+^p \\
    &\lesssim
    \sqrt{\log d}\,
    \frac{\sqrt{\eps_d}\lambda_p |\widetilde\delta|}{\sqrt d},
\end{align*}
where we used \eqref{Eq:LeaveOutLipschitz}, $\|\vec\zeta\|_2\lesssim1$, and
$a^p-(a-b)_+^p\leq pa^{\eps_d}b$. Therefore,
\[
    |\widetilde\delta|
    \lesssim
    \left(
        \frac{\sqrt{\eps_d}\lambda_p\sqrt{\log d}}{\sqrt d}
    \right)^{\frac1{\eps_d}}.
\]
Together with \eqref{Eq:GridApproximation}, and absorbing the negligible mesh size, this gives the same estimate for $|\delta|$. Applying the argument to every coordinate and taking a union bound, we obtain
\begin{equation}\label{Eq:InfinityBoundStepIII}
    \|\ermp\|_\infty
    \lesssim
    \lambda_p|\delta|
    \lesssim
    \left(
        \frac{C\log d}{\log\log d}
    \right)^{\frac1{2\eps_d}}d^{-1/2}.
\end{equation}

\paragraph{Remaining $1\lesssim k\lesssim n/\log(d/n)$ entries}
We now adapt the same argument to a dyadic block $k\in\cR_2$. Let $S_k$ be the support of $w_k$ and set
\[
    m_k:=\left\lfloor k\log(ed/k)\right\rfloor,
    \qquad
    \bar\lambda_k:=\lambda_{m_k}.
\]
Recall from Step~I that
\[
    \frac{w_k}{\sqrt{\eps_d}}\in C\cV_{m_k},
    \qquad
    \|w_k\|_2\lesssim\sqrt{\eps_d}\bar\lambda_k,
    \qquad
    \frac{\|w_k\|_p}{\sqrt{\eps_d}}\asymp1.
\]
Let $a_k:=\|w_k\|_p/\sqrt{\eps_d}$. By the splitting property,
\begin{equation}\label{Eq:BlockVariationalStepIII}
    \delta_k
    =
    \operatorname*{argmin}_{t\in[0,C_0]}
    \left\{
        a_k^p t^p+
        \|\vec\zeta-tPw_k\|_{\widetilde n,S_k}^p
    \right\}.
\end{equation}

For fixed $S$, $w$, and $t$, the map
\[
    \vec X^{(S)}\longmapsto
    \|\vec\zeta-tPw\|_{\widetilde n,S}
\]
is convex and has Frobenius Lipschitz constant at most
$Ct\|w\|_2/\sqrt d$. Using the nets $\widetilde{\cV}_{m_k}$ constructed in Step~I, Gaussian concentration, and the same approximation argument as in Step~I, we obtain, with probability at least $1-\exp(-c m_k)$,
\begin{equation}\label{Eq:UniformBlockConcentration}
    \|\vec\zeta-tPw\|_{\widetilde n,S}
    \geq
    \|\vec\zeta\|_{\widetilde n,S}
    -C\sqrt{m_k}\,\frac{t\|w\|_2}{\sqrt d}
\end{equation}
simultaneously over all supports $|S|\leq k$, all
$w/\sqrt{\eps_d}\in C\cV_{m_k}$ supported on $S$, and all $t$ in the discretization above. Taking a union bound over the dyadic values of $k$, the estimate applies to the random pair $(S_k,w_k)$.

Let $\widetilde\delta_k$ be the minimizer of the objective in \eqref{Eq:BlockVariationalStepIII} over the grid. As above, strict convexity gives
$|\widetilde\delta_k-\delta_k|\leq d^{-100}$. Comparing the value at $\widetilde\delta_k$ with the value at $0$ and using \eqref{Eq:UniformBlockConcentration}, we obtain
\begin{align*}
    a_k^p\widetilde\delta_k^p
    &\leq
    \|\vec\zeta\|_{\widetilde n,S_k}^p
    -
    \left(
        \|\vec\zeta\|_{\widetilde n,S_k}
        -C\sqrt{m_k}\,
        \frac{\widetilde\delta_k\|w_k\|_2}{\sqrt d}
    \right)_+^p \\
    &\lesssim
    \sqrt{m_k}\,
    \frac{\widetilde\delta_k\|w_k\|_2}{\sqrt d}.
\end{align*}
Since $a_k\asymp1$ and $\|w_k\|_2\lesssim\sqrt{\eps_d}\bar\lambda_k$, and since the mesh error is negligible, it follows that
\begin{equation}\label{Eq:DeltaBlockStepIII}
    \delta_k^{\eps_d}
    \lesssim
    \sqrt{k\log(ed/k)}\,
    \frac{\sqrt{\eps_d}\bar\lambda_k}{\sqrt d}.
\end{equation}
Using
$\bar\lambda_k\asymp(d/k)^{\frac1p-\frac12}$ and $\eps_d\asymp1/\log\log d$, we conclude that
\begin{equation}\label{Eq:DeltaBlockFinalStepIII}
    \delta_k
    \lesssim
    \left(\frac{k}{d}\right)^{1/p}
    \left(
        \frac{C\log(ed/k)}{\log\log d}
    \right)^{\frac1{2\eps_d}}.
\end{equation}
Consequently,
\begin{equation}\label{Eq:BlockL2StepIII}
    \|\delta_k w_k\|_2
    \lesssim
    \left(\frac{k}{d}\right)^{1/2}
    \left(
        \frac{C\log(ed/k)}{\log\log d}
    \right)^{\frac1{2\eps_d}}.
\end{equation}

Summing over the dyadic indices $k\in\cR_2$, and using that the corresponding blocks have disjoint supports, we obtain, when $d\geq n\log(d)^C$ with $C$ large enough,
\[
    \left\|\sum_{k\in\cR_2}\delta_k w_k\right\|_2^2
    =
    \sum_{k\in\cR_2}\|\delta_k w_k\|_2^2
    \lesssim
    \log(d)^{-c_1(C)}
\]
for some $c_1(C)>0$.

This concludes Step~III. The argument above holds with probability at least $1-d^{-A}$ for every fixed $A>0$, after adjusting the constants in the concentration estimates. Indeed, to work on the exceptional set, if $w_2$ denotes the minimum-$\ell_2$ interpolator, then $\|\ermp\|_2\leq\lambda_p\|\ermp\|_p\leq\lambda_p\|w_2\|_p\leq\lambda_p\|w_2\|_2$, and the fixed moments of $\|w_2\|_2$ are bounded. Taking $A$ large makes the exceptional event negligible and yields the moment bounds stated in Theorem~\ref{T:WRD}.

\subsubsection{Step IV: Extending to $O(1)$-sparse and upper bounding $E_2$:}
Note that, if $w=\lerm(\vec Y)$, then
\[
    w=\argmin_{\vec X u=\vec Y}\|u\|_p^p,
    \qquad
    \|w\|_p^p=\|w_1\|_p^p+\|w_2\|_p^p,
\]
where $w_1$ is the projection of $w$ to the $k$ first coordinates of the $\ell_p$-ball in $\R^d$, and $w_2$ is onto the last $d-k$ coordinates. For simplicity assume that $k=1$, meaning that to control $T_2$ we need to study the $\ell_2$ norm of the $\ell_p$-MNI (w.r.t. $\ell_p^{d-1}$) of the vector
\[
    \vec\zeta + p(\vec\zeta,\vec X)P\lft
\]
for some $p(\vec\zeta,\vec X) \lesssim \sqrt{E_1} \lesssim \Tilde{O}(\cG_n(B_p^d)^{p})$, which follows from  \Cref{T:UC2Loc} just by controlling
\[
    \|\lft - \cP_{\lft}\lerm\|_{2} \lesssim \cG_n(B_p^d)^{p} 
\]
with high probability.  Note that our analysis holds for any \textbf{fixed} $\vec\zeta'$ that satisfies $\|\vec\zeta'\|_{n-1} \asymp \sqrt{\eps_d}$ with probability at least $1-\exp(-cn)$. So by taking a union bound over $\vec\zeta' = \vec\zeta + \alpha \cdot P\lft$ (by dividing the interval into, say, $d^{-100}$ intervals) and using the $2$-uniform convexity, the claim follows.
\subsubsection{Sub-Gaussian Covariates}

We now indicate the modifications needed for i.i.d.\ symmetric, isotropic
sub-Gaussian covariates satisfying the CCP with constant $L$.
The extension to a fixed $O(1)$-sparse ground truth is unchanged.

\paragraph{Step I.}
The main difference is that the Dvoretzky inclusion used in the Gaussian
case is no longer available. Since
$\cC_d\subset\widetilde B_p^d$, \Cref{Lem:Kashin} gives, with high
probability,
\[
    c B_2^n
    \subset P\cC_d
    \subset P\widetilde B_p^d.
\]
Together with the dual mean-width estimate from
\Cref{lem_M*_bound} and Gaussian concentration in $\vec\zeta$, this yields
\[
    \sqrt{\eps_d}
    \lesssim
    r=\|\vec\zeta\|_n
    \lesssim 1.
\]
Thus, compared with the Gaussian case, we lose the lower Euclidean
inclusion of radius comparable to $\eps_d^{-1/2}$.

\paragraph{Step II.}
For $k\in\cR_1$, the sparse projection estimate remains valid. Using the Kashin inclusion above in the
same variational argument as in the Gaussian case gives
\[
    \|\delta_k w_k\|_2
    \lesssim
    \left(\frac{k}{d}\right)^{1/2}
    \left(
        C\log\left(\frac{ed}{k}\right)
    \right)^{1/(2\eps_d)}.
\]
Here the factor $r$ is absorbed using $r\lesssim1$. Since the
corresponding shells have disjoint supports,
\[
    \sum_{k\in\cR_1}\|\delta_k w_k\|_2^2
    \lesssim
    \left(\frac{C}{\eps_d}\right)^{1/\eps_d}.
\]
This part of the argument only uses the sub-Gaussianity constant.

\paragraph{Step III.}
For the remaining largest coordinates, we apply Kashin's theorem after
deleting one column. A union bound makes the resulting Euclidean inclusion
simultaneous over all deleted columns. Consequently, conditionally on the
remaining columns, the corresponding quotient norm is
$C$-Lipschitz with respect to the Euclidean norm.

Applying the CCP to the deleted column, followed by a union bound over the
columns and a polynomial discretization of the scalar coefficient, gives
\[
    \|\ermp\|_\infty^2
    \lesssim
    \frac1d
    \left(
        C L^2\log(ed)
    \right)^{1/\eps_d}.
\]
The shells in $\cR_2$ are supported on at most
$Cn/\log(ed/n)$ coordinates. Hence
\[
    \left\|
        \sum_{k\in\cR_2}\delta_k w_k
    \right\|_2^2
    \lesssim
    \frac{n}{d\log(ed/n)}
    \left(
        C L^2\log(ed)
    \right)^{1/\eps_d}.
\]

Combining the two ranges and treating the exceptional event as in the
Gaussian case, by comparison with the minimum-$\ell_2$ interpolator, gives
\[
    \E\|\ermp\|_2^2
    \lesssim
    \left(\frac{C}{\eps_d}\right)^{1/\eps_d}
    +
    \frac{n}{d\log(ed/n)}
    \left(
        C L^2\log(ed)
    \right)^{1/\eps_d}.
\]
This proves the sub-Gaussian part of \Cref{T:WRD}, and hence that of
\Cref{T:VarLinOpt} by the scaling reduction.

The loss with respect to the Gaussian case comes from the bulk term:
\[
    \sup_{0<x\leq1}
    x\left(C\log\frac ex\right)^{1/\eps_d}
    \lesssim
    \left(\frac{C}{\eps_d}\right)^{1/\eps_d}.
\]
Thus, for this contribution to remain polylogarithmic in $d$, it is enough
to assume
\[
    \eps_d
    \gtrsim
    \frac{\log\log\log d}{\log\log d}.
\]
The CCP contribution remains separate.

\subsubsection{Proof of Corollary \ref{Cor:VarPaou}}
The quantity of interest is invariant by scaling in both $\vec X$ and $\vec \xi$. Here $\vec\xi$ denotes the fixed vector in the statement. Thus we work with $P = \frac{1}{\sqrt{d}}\vec X$ and assume that $\|\vec \xi\|_2=1$ for simplicity.
Recall the proof of \Cref{T:VarLinOpt} above. First, we showed that with probability of $1-\exp(-\Tilde{\Omega}(d^{\eps_d}))$
\[
    \ermp \in \widetilde B_p^d \cap \underbrace{C_1\sqrt{\eps_d} \cdot \exp(C/\eps_d) \cdot B_d}_{:=K}
\]
denote this event by $\cE_1$. Furthermore, we also show that  with probability of $1-C\exp(-cn)$,
\[
   c B_2^n \subset PK \text{ where $P = \frac{1}{\sqrt{d}}\vec X$}
\]
and denote this event by $\cE_2$. Truncating $\widetilde B_p^d$ to a $\ell_2$-ball with a smaller diameter will let us apply the Poincaré inequality or (Lipschitz concentration) in $P$ with a better constant rather than the entire $\widetilde B_p^d$. Furthermore, by rotational invariance of $\vec X$, we have the inequality in law $\|\vec\xi\|_{P\widetilde B_p^d} \sim \|U\vec\xi\|_{P\widetilde B_p^d}$ for any rotation $U\in O(n)$. In particular, we may assume that $\vec\xi$ is uniformly distributed on the unit sphere. From the proof of  \Cref{T:VarLinOpt}, we know that 
$$\|\vec\xi\|_{P\widetilde B_p^d} \lesssim 1$$
with probability at least $1-\exp(-cn)$. Calling this event $\cE_3$, we have that on $\cE:= \cE_1 \cap \cE_2\cap\cE_3$, 

$$\|\vec\xi\|_{P\widetilde B_p^d} = \|\vec\xi\|_{PK} \lesssim 1 \ \text{and } \ c B_2^n \subset PK.$$

Working on that event, we let $Pw = \vec \xi$ such that $\|w\|_{K} = \|\vec\xi\|_{PK}$. 
Note that  we may write
\[
\vec{\xi} = Pw = P'w + (P - P')w,
\]
Hence,
\[
\|\vec\xi\|_{P'K}
\le
\|P'w\|_{P'K} + \|(P - P')w\|_{P'K}.
\]
Since \(P'w \in P'K\), we have \(\|P'w\|_{P'K} \leq \|w\|_{K}\), and using that \(cB_2^n \subset P'K\) give
\[
\|(P - P')w\|_{P'K}
\le
\frac{1}{c}\|(P - P')w\|_2
\le
\frac{1}{c}\|P - P'\|_F \|w\|_2
\]
and therefore
\[
\|\vec\xi\|_{P'K}
\le
\|w\|_{K} + \frac{1}{c}\|P - P'\|_F \|w\|_2.
\]
Recalling that \(\|w\|_{K} = \|\vec \xi\|_{PK}\) and noting that \(\|w\|_2 \leq D(K)\|\vec \xi\|_{PK}\), we get
\[
\|\vec\xi\|_{P'K}
\le
\|\vec\xi\|_{PK}
\left(1 + \frac{D(K)}{c}\|P - P'\|_F\right).
\]
Exchanging the roles of $P$ and $P'$, we conclude that, on the event $\cE$,

\begin{equation}\label{eq1193}
    \bigl|\|\vec\xi\|_{PK} - \|\vec\xi\|_{P'K}\bigr|
\le
\frac{D(K)}{c}\|P' - P\|_F
\max\{\|\vec \xi\|_{PK}, \|\vec \xi\|_{P'K}\} \leq  \frac{D(K)}{c}\|P' - P\|_F,
\end{equation}
where the last inequality follows from the ball inclusion. 

Let $f(P) = \|\vec\xi\|_{P\widetilde B_p^d}$ and $g(P) = \|\vec\xi\|_{PK}$. From \eqref{eq1193} we learn that $f$ is $\frac{D(K)}{c}$-Lipschitz on the event $\cE$. To deal with the low probability event $\cE^c$ we will need the following standard lemma.
\begin{lemma}
Let $G\sim \mathcal N(0,\sigma^2 I_d)$, and let $f:\R^d\to\R$ be measurable. Assume that there exists a measurable set $A\subset \R^d$ such that
\[
\PP(G\in A)\geq 1-\eta,
\]
and such that $f_{|A}$ is $L$-Lipschitz. Set
\[
M=\EE[f(G)^4].
\]
Then
\[
\Var(f(G)) \leq 2\sigma^2L^2 + 10\sqrt{M\eta}.
\]
\end{lemma}

\begin{proof}
Let $T_R(t)=\max(-R,\min(t,R))$. Since $T_R$ is $1$-Lipschitz, $(T_R\circ f)_{|A}$ is still $L$-Lipschitz. Let $h_0:\R^d\to\R$ be a McShane extension of $(T_R\circ f)_{|A}$, and set $h=T_R\circ h_0$. Then $h$ is globally $L$-Lipschitz, satisfies $h=T_R\circ f$ on $A$, and $|h|\leq R$ everywhere. Hence Gaussian Poincar\'e gives
\[
\Var(h(G))\leq \sigma^2L^2.
\]
Therefore
\[
\Var(f(G))
\leq 2\Var(h(G))+2\EE(f(G)-h(G))^2
\leq 2\sigma^2L^2+2\EE(f-h)^2.
\]

Now
\[
\EE(f-h)^2=\EE\bigl[(f-h)^2\mathbf{1}_A\bigr]+\EE\bigl[(f-h)^2\mathbf{1}_{A^c}\bigr].
\]
On $A$, since $h=T_R(f)$,
\[
(f-h)^2\leq f^2\mathbf{1}_{\{|f|>R\}},
\]
hence
\[
\EE\bigl[(f-h)^2\mathbf{1}_A\bigr]\leq \EE\bigl[f^2\mathbf{1}_{\{|f|>R\}}\bigr]\leq \frac{\EE[f^4]}{R^2}=\frac{M}{R^2}.
\]
On $A^c$, since $|h|\leq R$,
\[
(f-h)^2\leq 2f^2+2R^2,
\]
so by Cauchy--Schwarz,
\[
\EE\bigl[(f-h)^2\mathbf{1}_{A^c}\bigr]
\leq 2\EE\bigl[f^2\mathbf{1}_{A^c}\bigr]+2R^2\eta
\leq 2(\EE[f^4])^{1/2}\PP(A^c)^{1/2}+2R^2\eta
\leq 2\sqrt{M\eta}+2R^2\eta.
\]
Thus
\[
\EE(f-h)^2\leq \frac{M}{R^2}+2\sqrt{M\eta}+2R^2\eta.
\]
Choosing $R^2=\sqrt{M/\eta}$ yields
\[
\EE(f-h)^2\leq 5\sqrt{M\eta},
\]
and therefore
\[
\Var(f(G)) \leq 2\sigma^2L^2 + 10\sqrt{M\eta}.
\]
\end{proof}
In view of the above lemma, we will also need an estimate on $\Ef^4(P)$. Since $B_2\subset\widetilde B_p^d$, assuming that $d>2n+4$ we have the upper bound
\begin{align*}
    \EE f^4(P)
    &\leq \EE\|\vec\xi\|_{PB_2^d}^4
    = \EE \langle \vec\xi , (PP^T)^{-1}\vec\xi\rangle^2 \\
    &= d^2 \, \EE \langle \vec\xi , (XX^T)^{-1}\vec\xi\rangle^2
    = d^2 \, \EE \big((XX^T)^{-1}_{11}\big)^2 \\
    &= \frac{d^2}{(d-n-1)(d-n-3)}
    \leq 4.
\end{align*}
Since $\cE$ has probability at least $1-\exp\left(-c\min(n, d^{\eps_d})\right)$, we obtain that, when $\|\vec\xi\|_2=1$, 

\begin{align*}
    \Var f(P) &\lesssim \frac{D^2(K)}{d} + \exp\left(-c\min(n, d^{\eps_d})\right)\\
    & \lesssim \frac{D^2(K)}{d} + \exp\left(-cd^{\eps_d}\right) \lesssim \frac{D^2(K)}{d}
\end{align*}
The proof is complete since $\EE\|\vec \xi\|_{P\widetilde B_p^d} \gtrsim \sqrt{\eps_d}$.

\begin{remark}
Note that one can improve our variance estimate with substantial work by considering the $L_1-L_2$ Talagrand inequality that would give us another $\eps_d$ factor, when $p \geq 1 + \Omega(1/\log(d))$. As we know that the $\ermp$ is essentially, $\Theta(d \cdot 2^{-\tfrac{1}{\eps_d}})$ sparse. Consider,  \cite{Paorlp} bound mentioned above of order
$
   \frac{2^{q}}{d \cdot q^2}.
$
 The term $\frac{2^{q}}{d}$ emerges from truncation diameter, one $q$ term from $M^*(P\widetilde B_p^d)^2\asymp q$ and the additional factor of $q$ emerges from super-concentration via $L_1-L_2$ that is the log of the sparsity level, which is of order $q$.

\end{remark}
\bibliography{Bib}

\end{document}